\documentclass[11pt, letter]{article}

\usepackage{amsfonts}
\usepackage{amssymb}
\usepackage{amsmath}
\usepackage{amsthm}
\usepackage{latexsym}
\usepackage{color}
\usepackage{comment}
\usepackage{cancel}

\numberwithin{equation}{section}

\def\r{\mathbb R}

\def\s{\mathbb S}

\DeclareMathOperator{\im}{Im}
\def\C{\mathcal C}

\DeclareMathOperator{\divergence}{div}

\DeclareMathOperator{\distance}{dist}
\DeclareMathOperator{\Hyperg}{\mbox{ }_2 F_{1}}

\newtheorem{theorem}{Theorem}[section]

\newtheorem{proposition}[theorem]{Proposition}
\newtheorem{lemma}[theorem]{Lemma}
\newtheorem{corollary}[theorem]{Corollary}

\theoremstyle{definition}
\newtheorem{remark}[theorem]{Remark}

\begin{document}

 \author{Weiwei Ao\\ {\small Wuhan University} \and Azahara  DelaTorre  \\
{\small  Sapienza Universit\`a di Roma}
% {\small  Albert-Ludwigs-Universit\"at Freiburg}
 \and Mar\'ia  del  Mar  Gonz\'alez\\ \small{Universidad Aut\'onoma de Madrid and ICMAT} }

\title{Symmetry and symmetry breaking for the fractional Caffarelli-Kohn-Nirenberg inequality}
\date{}
\maketitle

\begin{abstract}

In this paper, we will consider the fractional Caffarelli-Kohn-Nirenberg inequality
\begin{equation*}
{\Lambda} \left(\int_{\r^n}\frac{|u(x)|^{p}}{|x|^{{\beta} {p}}}\,dx\right)^{\frac{2}{p}}\leq
\int_{\r^n}\int_{\r^n}\frac{(u(x)-u(y))^2}{|x-y|^{n+2\gamma}|x|^{{\alpha}}|y|^{{\alpha}}}\,dy\,dx
\end{equation*}
where $\gamma\in(0,1)$, $n\geq 2$, and $\alpha,\beta\in\r$ satisfy
\begin{equation*}
\alpha\leq \beta\leq \alpha+\gamma, \ -2\gamma<\alpha<\frac{n-2\gamma}{2},
\end{equation*}
and the exponent $p$ is chosen to be
\begin{equation*}
p=\frac{2n}{n-2\gamma+2(\beta-\alpha)},
\end{equation*}
such that the inequality is invariant under scaling. We first study the existence and nonexistence of extremal solutions. Our next goal is to show some results on the symmetry and symmetry breaking region for the minimizers; these suggest the existence of a Felli-Schneider type curve separating both regions but, surprisingly, we find a novel behavior as $\alpha\to -2\gamma$. The main idea in the proofs, as in the classical case, is to reformulate the fractional Caffarelli-Kohn-Nirenberg inequality in cylindrical variables.  Then, in order to find the radially symmetric solutions  we need to solve a non-local ODE.

For this equation we also get uniqueness of minimizers in the radial symmetry class; indeed, we show that the unique continuation argument of Frank-Lenzmann (Acta'13) can be applied to more general operators with good spectral properties. We provide, in addition, a completely new proof of non-degeneracy which  works for all critical points. It is based on the variation of constants approach and the non-local Wronskian of Ao-Chan-DelaTorre-Fontelos-Gonz\'alez-Wei (Duke'19).

\end{abstract}

\section{Introduction and statement of the results}

The Caffarelli-Kohn-Nirenberg (CKN) inequality was introduced in $1984$ (see \cite{CaffarelliKohnNirenberg} and \cite{Lin}), and the existence or non-existence of extremal solutions and their symmetry properties have been extensively studied since then.  A particular case of this inequality establishes that
for all $\alpha\leq \beta \leq \alpha+1$ and $\alpha\neq \tfrac{n-2}{2}$, in space dimension $n>2$, it holds that
\begin{equation}\label{CKN_classical}
\Lambda_{\alpha,\beta}^n\left(\int_{\r^n}\frac{|u|^{p}}{|x|^{\beta p}}\,dx\right)^{2/p}\leq \int_{\r^n}\frac{|\nabla u|^2}{|x|^{2\alpha}}\,dx,\quad \forall\ u\in C^{\infty}_c(\mathbb R^n),
\end{equation}
where
\begin{equation}\label{value-p}
p=\frac{2n}{n-2+2(\beta-\alpha)},
\end{equation}
and  $(\Lambda_{\alpha,\beta}^n)^{-1}$ denotes the optimal constant.
 This inequality represents an interpolation between the usual Sobolev inequality $(\alpha=0,\beta=0)$ and the Hardy inequality $(\alpha=0,\beta=1)$ or weighted Hardy  inequality $(\beta=\alpha+1)$.

In this work we consider a new fractional version of the Caffarelli-Kohn-Nirenberg  inequality %\eqref{classical}
 \eqref{CKN_classical}, for $\gamma\in(0,1)$, $n\geq 2$, and we will always assume that $\alpha,\beta\in\r$ satisfy
\begin{equation}\label{parameter}
\alpha\leq \beta\leq \alpha+\gamma, \ -2\gamma<\alpha<\frac{n-2\gamma}{2}.
\end{equation}
We also set
\begin{equation}\label{p*}
p=\frac{2n}{n-2\gamma+2(\beta-\alpha)}.
\end{equation}

\begin{proposition}[Fractional CKN inequality]\label{prop:inequality}
There exists a constant $\Lambda>0$ such that
\begin{equation}\label{ineq_u}
{\Lambda} \left(\int_{\r^n}\frac{|u(x)|^{p}}{|x|^{{\beta} {p}}}\,dx\right)^{\frac{2}{p}}\leq
\int_{\r^n}\int_{\r^n}\frac{(u(x)-u(y))^2}{|x-y|^{n+2\gamma}|x|^{{\alpha}}|y|^{{\alpha}}}\,dy\,dx
\end{equation}
for every $u\in C^{\infty}_c (\mathbb R^n)$.
\end{proposition}

Let $D_{\alpha}^\gamma(\mathbb R^n)$ be the completion of $C^\infty_c(\mathbb R^n)$ with respect to the inner product
\begin{equation*}
\langle u_1,u_2 \rangle_{\gamma,\alpha}:=
\int_{\r^n}\int_{\r^n}\frac{(u_1(x)-u_2(y))(u_1(x)-u_2(y))}{|x-y|^{n+2\gamma}|x|^{{\alpha}}|y|^{{\alpha}}}\,dy\,dx.
\end{equation*}
Its associated norm is denoted by
\begin{equation}\label{weighted-norm}
\|u\|_{\gamma,\alpha}^2:=
\int_{\r^n}\int_{\r^n}\frac{(u(x)-u(y))^2}{|x-y|^{n+2\gamma}|x|^{{\alpha}}|y|^{{\alpha}}}\,dy\,dx.
\end{equation}
The functional space $D_{\alpha}^\gamma $  will be studied in detail in Section \ref{subsection:function-spaces}. The proof of the fractional CKN inequality \eqref{ineq_u} is a consequence of the fractional Sobolev embedding on the cylinder as it will be explained in Remark \ref{proof_ineq}.

Note that the restriction $-2\gamma<\alpha$ in \eqref{parameter} is not present in the local setting; however in our context it is necessary because the way we have defined the norm in  \eqref{weighted-norm}.  A more general definition to include all range of $\alpha$ should be possible but it falls outside the scope of this paper. We refer the reader to Remark \ref{remark:restriction} for a more detailed explanation.\\

Several versions of the CKN inequality for the fractional norm \eqref{weighted-norm} have already appeared in the literature: \cite{Abdellaoui-Bentifour} (for $\alpha=\beta$), and \cite{Nguyen-Squassina} (for $\beta=0$), and the very old paper (in Russian) \cite{Ilin}, which actually contains the first proof of the standard CKN inequality.

   Similarly to the classical case, inequality \eqref{ineq_u} is an interpolation between the fractional Sobolev inequality $\alpha=0,\beta=0$ (see \cite{Rupert_Lieb}) and the weighted fractional Hardy inequality, for $\beta=\alpha+\gamma$ (see \cite{Abdellaoui-Bentifour} or \cite{Frank-Lieb-Seiringer}). In the Sobolev inequality,  extremals are the so-called ``bubbles'', this is, solutions for the fractional Yamabe problem.  When $\beta=\alpha+\gamma$  the best constant is universal but never attained. Note that this last fact, although well known in the literature, can be also obtained as a consequence the non-local ODE theory of \cite{acdfgw} (see Remark \ref{remark:Hardy}).\\

Now, in order to understand the extremals for inequality \eqref{ineq_u} for any given parameters satisfying \eqref{parameter} and \eqref{p*}, we consider   the energy functional
\begin{equation}\label{defi-E}
E_{\alpha,\beta}(u)=\frac{\|u\|_{\gamma,\alpha}^2}{\Big(\int_{\r^n}|x|^{-\beta p}|u|^p\,dx\Big)^{2/p}}.
\end{equation}
The best constant in this inequality is given by
\begin{equation}\label{eq-S}
S(\alpha,\beta)=\inf_{u\in D_{\alpha}^\gamma(\r^n)\setminus \{0\}}E_{\alpha,\beta}(u).
\end{equation}
Extremal solutions for \eqref{eq-S}  satisfy the following equation:
\begin{equation}\label{eq-extremal}
\mathcal L_{\gamma,\alpha}(u)
=c\frac{|u(x)|^{p-2}u(x)}{|x|^{\beta p}},
\end{equation}
for some constant $c$, where we have defined
\begin{equation}\label{formula-L}
\mathcal L_{\gamma,\alpha}(u):=\int_{\r^n}\frac{u(x)-u(y)}{|x-y|^{n+2\gamma}|x|^{\alpha}|y|^{\alpha}}\,dy.
\end{equation}
%We normalize the Lagrange multiplier $c$ in the equation to $c=\kappa_{\alpha,\gamma}^n$, given in \eqref{kappa}. The reason for this normalization will be clear in \eqref{??}.\textcolor{green}{to do}

\begin{comment}The value of the best constant $S(\alpha,\beta)$ is given by
\begin{equation*}
	S(\alpha,\beta)=2\kappa_{\alpha,\gamma}^n\Big( \int_{\r^n}\frac{|u|^{2_\gamma^*}}{|x|^{\beta 2_\gamma^*}}dx\Big)^{\frac{2}{2_\gamma^*}-1}.
\end{equation*}
\end{comment}
\begin{comment}
The Euler-Lagrange equation associated to \eqref{ineq_u} is
\begin{equation}\label{EL}
{\kappa}^n_{\alpha,\gamma}\frac{|u(x)|^{{2^*_{\gamma}}-2}u(x)}{|x|^{\beta {p}}}=\int_{\r^n}\frac{(u(x)-u(y))}{|x-y|^{n+2\gamma}|x|^{\alpha}|y|^{\alpha}}\,dy,
\end{equation}
%\begin{proof}
where the constant ${\kappa}^n_{\alpha,\gamma}$ is normalized as in \eqref{kappa}. %Note that for the
\end{comment}

The first goal of this paper is to study the existence, nonexistence and some symmetry properties of extremal solutions to \eqref{eq-S}.

\begin{comment}

\textcolor{blue}{In order to make our statement below more clear, let us first introduce some constants we will use in the statement of our theorems and in the rest of the paper:
\begin{align*}
	\kappa^{n,\bar{\alpha}}_{{\alpha},\gamma}&=\text{P.V.}\int_{\r^n}\frac{(1 -|{\zeta}|^{-\bar{\alpha}})}{|e_1-{\zeta}|^{n+2\gamma}|{\zeta}|^{{\alpha}}}\,d{\zeta}    \, (\mbox{\, defined \, in }(\ref{kappa_2-1}))\\
	\kappa_{\alpha,\gamma}^n&=\kappa^{n,\nu}_{{\alpha},\gamma}\mbox{ \, for \, }\nu=\frac{n-2\gamma}{2}-\alpha   (\mbox{ defined \, in }(\ref{kappa}))\\
	C(\alpha)&=\varsigma_{n,\gamma}\kappa_{\alpha,\gamma}^n-c_{n,\gamma} \mbox{\, for \, some \, normalization \, constant \, }\varsigma_{n,\gamma}, c_{n,\gamma} (\mbox{ defined \, in } (\ref{constant})).
\end{align*}}
\end{comment}
\medskip

\begin{theorem}[Best constants, existence and non-existence of extremal solutions]\label{thm1}
It holds:
\begin{itemize}
\item[i.]
$S(\alpha,\beta)$ is continuous in the full parameter domain \eqref{parameter}.
%\item[ii.]For $\beta=\alpha+\gamma$, we have $S(\alpha,\alpha+\gamma)=2\kappa_{\alpha,\gamma}^n$, where this constant is defined in \eqref{kappa}, and it is not achieved.
\item[ii.]
For  $\alpha=\beta$, $0<\alpha<\frac{n-2\gamma}{2}$, $S(\alpha, \alpha)$ is achieved and the extremal solution is radially symmetric and non-increasing in the radial variable.
\item[iii.]
For $\alpha=\beta$, $-2\gamma<\alpha<0$, $S(\alpha, \alpha)=S(0,0)$ and it is not achieved.
\item[iv.]
For $\alpha<\beta<\alpha+\gamma$, $S(\alpha,\beta)$ is always achieved. Moreover, for $\alpha\geq 0$,  $\alpha< \beta<\alpha+\gamma$, the extremal solution is radially symmetric and non-increasing in the radial variable.
\end{itemize}
\end{theorem}

Note that, in the case $\alpha=\beta$ or $0\leq \alpha<\beta<\alpha+\gamma$ with $0\leq \alpha<\frac{n-2\gamma}{2}$ (statements \emph{ii.} and part of \emph{iv.}), given an extremal solution $ u$ of \eqref{eq-extremal}, if we use
$\tilde u(x)=|x|^{-\alpha}u(x)$,
then $\tilde u$ will satisfy
\begin{equation}\label{new-equation}
(-\Delta)^\gamma \tilde u+\tau\frac{\tilde u(x)}{|x|^{2\gamma}}
=\frac{\tilde u^{p-1}}{|x|^{(\beta-\alpha)p}},
\end{equation}
for some constant $\tau$.
%(actually, $\tau=\kappa_{\alpha,\gamma}^{n,-\alpha}$, which is  defined in \eqref{kappa_general})
This equation has been considered in  \cite{Ghoussoub-Shakerian,Dipierro-Montoro-Peral-Sciunzi}, where they studied  attainability of the best constant and radial symmetry of the minimizers (via a rearrangement argument). In any case, we provide here a proof of these facts directly for the energy $E_{\alpha,\beta}$ in order to make the paper self-contained (Proposition \ref{prop:symmetrization}).\\

Our initial approach in the proofs is, similarly to the local case of \cite{Catrina_Wang2}, to rewrite the inequality in cylindrical coordinates. Thus we set
\begin{equation}\label{cylindrical-coordinates}
t=\ln |x|,\quad \theta=\frac{x}{|x|},\quad t\in\mathbb R,\theta\in\mathbb S^{n-1},
\end{equation}
and
\begin{equation}\label{u-v}
v(t, \theta)=e^{\frac{n-2\gamma-2\alpha}{2}t}u(e^{t}\theta).
\end{equation}
Many times it will be more preferable to work with the function $v$ instead of $u$.
While this is an immediate change of variables in the local framework, in the non-local case it is a non-trivial step. The key is to rewrite the operator $\mathcal L_{\gamma,\alpha}(u)$ from \eqref{formula-L}
in terms of the conformal fractional Laplacian on the cylinder, denoted by $P_\gamma v$, which is defined by
\begin{equation*}
P_\gamma v= r^{\frac{n+2\gamma}{2}}(-\Delta)^\gamma(r^{-\frac{n-2\gamma}{2}} v),
\end{equation*}
on the cylinder $\mathcal C=\mathbb R\times\mathbb S^{n-1}$, with coordinates $t\in\mathbb R$, $\theta\in\mathbb S^{n-1}$. The definition of $P_\gamma$ comes from conformal geometry: it is the conformally covariant pseudo-differential operator (of order $2\gamma$) that arises as the Dirichlet-to-Neumann operator on Anti-de-Sitter space. In analytic terms, this operator is just a conjugation of the usual fractional Laplacian  on $\mathbb R^n$, but it has a simpler expression in polar coordinates. This theory has been developed in a series of papers following DelaTorre's PhD thesis  and it is explained in Section \ref{section:preliminaries} below. We also note that in the paper \cite{Frank:proceedings} the author proves a sharp Sobolev inequality for $(-\Delta)^{1/2}$ in dimension $n=3$ with a Hardy term subtracted, which yields a sharp CKN inequality. The idea is still to use cylindrical coordinates; however, it is inspired in a complete different problem in fluid mechanics (the intermediate long-wave equation).

Then in Section \ref{section:cylindrical}  we study the CKN inequality in cylindrical coordinates \eqref{cylindrical-coordinates}-\eqref{u-v}. In particular, the Euler-Lagrange equation \eqref{eq-extremal} is equivalent to
\begin{equation}\label{EL-v}
P_\gamma v +C(\alpha)v=cv^{p-1},\quad t\in\mathbb R, \,\theta\in\mathbb S^{n-1},
\end{equation}
where $C(\alpha)$ a real constant defined in \eqref{constant} and its properties are studied in the Appendix.

We notice that, if $u$ is an extremal solution, then $u$ must be strictly positive in $\mathbb R^n\setminus\{0\}$. This is so because of the integral definition of $P_\gamma$ in \eqref{P-cylinder} (see Proposition \ref{prop:positivity} for details).\\

An objective of this paper is to study radial symmetry of a minimizer  for the range of parameters where  symmetrization is not available, this is,   $-2\gamma<\alpha< 0$,  $\alpha< \beta<\alpha+\gamma$.  On the one hand, if we restrict to the radially symmetric class of functions, a minimizer always exists (see Proposition \ref{prop:radial-minimizer}). On the other hand, this radially symmetric solution may not achieve the minimum (see Theorem \ref{thm3} below on symmetry breaking). In both cases, minimizers always satisfy  the following  inversion symmetry:
%\begin{remark}
%In the case $\beta=\alpha+\gamma$ we recover the fractional Hardy inequality.
%\end{remark}

\begin{theorem}[Symmetry property]\label{thm2}
For $\alpha\leq\beta<\alpha+\gamma$, any solution of \eqref{eq-extremal} in $D_{\alpha}^\gamma(\r^n)$ satisfying $u(x)>0$ in $\r^n\setminus\{0\}$ satisfies the modified inversion symmetry
\begin{equation*}
u\left(\frac{x}{|x|^2}\right)=|x|^{n-2\gamma-2\alpha}u(x),
\end{equation*}
after a dilation $u(x)=\lambda^{\frac{n-2\gamma-2\alpha}{2}}u(\lambda x)$,  $\lambda>0$ if necessary. Moreover, in cylindrical coordinates we have that the function $v$ defined in \eqref{u-v}
is even in $t$ and monotonically decreasing for $t>0$.
\end{theorem}

We now look at the symmetry and symmetry breaking regions for the minimizers of \eqref{eq-S}. Already in the local case $\gamma=1$, to have a full answer is highly non-trivial and, indeed, it was only recently completed using non-linear flows (see the seminal paper  \cite{Dolbeault_Esteban_Loss:rigidity_flows}). But in this process there were many crucial developments: \cite{Catrina_Wang2}, the first paper where symmetry breaking occurs for the CKN inequality, and  \cite{Felli_Schneider}, where they proved the existence of a curve separating the symmetry and the symmetry breaking regions. Our next Theorem suggests that this should also occur in the non-local case:

\begin{theorem}[Symmetry breaking]\label{thm3} It holds:
\begin{itemize}
	
\item[(i)]
 For $-2\gamma<\alpha<0$, {there exists an open subset $H$ inside this region containing the set $\{(\alpha, \alpha)\in {\r^2}, \ \alpha\in (-2\gamma,0)\}$ such that for any $(\alpha,\beta)\in H$ with $\alpha<\beta$},
the extremal solution to $S(\alpha, \beta)$ is non-radial.

\item[(ii)] There exists $\alpha_0\in (-2\gamma, 0)$ and a smooth curve $\beta=h(\alpha)$ satisfying that
$$\alpha<h(\alpha)<\alpha+\gamma\quad\text{in the interval } \alpha\in(-2\gamma,\alpha_0)$$
and
 \begin{equation}\label{FS}h(\alpha)- \alpha\to\gamma\quad\text{as} \quad\alpha\to -2\gamma,\end{equation}
     such that for  $-2\gamma<\alpha<\alpha_0$ and $\alpha<\beta<h(\alpha)$, the extremal solution to $S(\alpha, \beta)$ is non-radial.
\end{itemize}
\end{theorem}

We first remark that, while the first statement in Theorem \ref{thm3} is expected, \eqref{FS} is quite striking in comparison with the behavior of the Felli-Schneider curve in the local case. This does not yield a contradiction since our claim is not uniform in $\gamma$ as $\gamma\to 1$, as we will see in the proof. In any case, this is another interesting example where the non-local version presents a contrasted behavior from its local counterpart.

As we have mentioned, we expect that there should exist a Felli-Schneider type curve separating the symmetry and symmetry breaking regions in the fractional setting (see  Remark \ref{remark:FS}). However,  the proof in the local case relies on the explicit knowledge of the spectrum of a Schr\"odinger type  operator, which is not available here.

Let us notice that for equation \eqref{new-equation}, the recent papers  \cite{Musina-Nazarov,Musina-Nazarov1}  look at non-degeneracy, symmetry and symmetry breaking issues of solutions. In particular, they prove
 symmetry breaking for  $\tau$ large enough (which, in our setting,  corresponds to $\alpha $ close to $-2\gamma$).\\

Next, we turn to non-degeneracy issues.  For the rest of the paper, we will work on the symmetry region, where we know that minimizers are radially symmetric. This includes, but it is not restricted to, the value of the parameters $0\leq \alpha < \frac{n-2\gamma}{2}$ and $\alpha\leq \beta <\alpha+\gamma $.

Thus let $\bar u$ be a positive radially symmetric (energy) solution of \eqref{eq-extremal}, and set $\bar v$ as in \eqref{u-v}.  Then $\bar v=\bar v(t) $ is a positive energy solution to of \eqref{EL-v}, that is even in the $t$ variable by Theorem \ref{thm2}. Consider the linearized operator given by
 \begin{equation}\label{linearoperator}
 \bar L w:=P_\gamma w+C(\alpha)w-c(p-1){ \bar{v}}^{p-2}w,\quad t\in\mathbb R,\, \theta\in\mathbb S^{n-1},
 \end{equation}
 and project it over the radial sector (in the notation of Section \ref{section:preliminaries})
 \begin{equation}\label{linearoperator0}
 \bar L^{(0)} w=P_\gamma^{(0)} w+C(\alpha)w-c(p-1){ \bar{v}}^{p-2}w,\quad w=w(t),\ t\in\mathbb R.
 \end{equation}
 One advantage of this point of view is that it allows to exploit the invariances that are present. Indeed, since the equation is translation invariant in the variable $t$, then  $\partial_t\bar{v}:=\bar v_t$ belongs to the kernel of $\bar L^{(0)}$. Our next result shows that this is the only possibility:

\begin{theorem}[Non-degeneracy]\label{thm:non-degeneracy}
Let $\bar v(t)$ be a positive solution of \eqref{EL-v} in the energy space $\tilde D^\gamma$. Then, in the $L^2$-radial symmetry class,
\begin{equation*}
\ker(\bar L^{(0)})=\langle \bar v_t\rangle.
\end{equation*}
\end{theorem}

Non-degeneracy for non-local problems has been considered by several authors. For instance, the (non-local) bubble in the critical case for the $\gamma$-Yamabe equation was studied in \cite{Davila-DelPino-Sire}. But the most important developments were the seminal works \cite{Frank-Lenzmann} and \cite{Frank-Lenzmann-Silvestre} on the non-local Schr\"odinger equation \begin{equation}\label{Schrodinger}
(-\Delta)^\gamma u+u=u^{p-1}.
\end{equation}
 The authors developed there a non-local Sturm-Liouville oscillation theory and highlighted the significant of Hamiltonian identities. Subsequent work  includes \cite{Musina-Nazarov}, where non-degeneracy of minimizers for the fractional Hardy-Sobolev inequality is proved, and \cite{Alarcon-Barrios-Quaas}, for the critical fractional H\'enon equation.

However, in all these papers it is assumed that $\bar u$ is  minimizer (or a local minimizer) for the energy functional. Instead, here we use the ODE-type arguments developed  in  \cite{acdfgw} (see also the survey article \cite{acdfgw:survey}) for equations such as \eqref{EL-v} to give an alternative proof of non-degeneracy that is valid for any bounded, positive, radially symmetric solution, not necessarily a minimizer. We use a (non-local) Wronskian together with  Frobenius type theorem to control the asymptotic expansion of the solution in terms of the indicial roots of the problem. This approach to non-degeneracy works because we already know one kernel by the scaling invariance of the equation,  while this type of argument cannot be used for the non-linear Schr\"odinger equation \eqref{Schrodinger}. This is the content of  Section \ref{section:nondegeneracy}.\\

As a consequence of non-degeneracy, in Section \ref{sec8} we  prove uniqueness of minimizers in the symmetry range:

\begin{theorem}[Uniqueness]\label{thm:uniqueness}
Let $n>2$. Then, in the symmetry region, minimizers for $F_{\alpha,\beta}$ are unique up to translation in the $t$-variable (and thus, uniqueness up to rescaling holds for minimizers of $E_{\alpha,\beta}$).
\end{theorem}

\begin{remark}
In fact, by our argument, we obtain non-degeneracy and uniqueness of positive radial solutions to  equation \eqref{EL-v} up to translation  in the energy space $\tilde{D}^\gamma$.	
\end{remark}

Existence and uniqueness theorems for a non-local equation are not available in general, since one cannot reduce it to the study of a phase portrait as in the local case. Variational arguments will lead to the existence part but uniqueness is usually the hardest.

In this regard, the only available proof  of uniqueness for equation \eqref{Schrodinger} is that of \cite{Frank-Lenzmann} and \cite{Frank-Lenzmann-Silvestre}. The proof of Theorem \ref{thm:uniqueness} follows their scheme: the key idea is to perform a unique continuation argument in $\gamma$ for $\gamma\to 1$; if we had two solutions for the $\gamma$-problem, this would yield a contradiction since for the local equation ($\gamma=1$) uniqueness holds. For our equation \eqref{EL-v} we need to deal with a different non-local operator (see Proposition \ref{prop:symbol} below for the precise formulas). However, we show that this approach only depends on having good spectral properties for its  Fourier symbol as $\gamma \to 1$.

Additionally, we have just heard about the recent work by Alarc\'on, Barrios and Quaas \cite{Alarcon-Barrios-Quaas}, mentioned above, where  they consider the critical fractional H\'enon equation in $\r^n$. They prove that the ground state solution is unique up to scaling. While they still use unique continuation as $\gamma\to 1$ as in \cite{Frank-Lenzmann}, they simplify the proof of the a-priori estimates. Indeed, they use the blow-up method of Gidas and Spruck for the operator \eqref{Pm} combined with a version of Liouville's theorem for fractional equations (see, for instance, \cite{Felmer-Quaas} if $\gamma\geq 1/2$) to prove a-priori global estimates, thus avoiding the use of spectral arguments.\\

Finally, we will use the notation $$2^*_\gamma:=\frac{2n}{n-2\gamma}$$
for the critical exponent in the fractional Sobolev embedding in $\mathbb R^n$.\\

The paper is organized as follows. In Section \ref{section:preliminaries}, some preliminaries for the conformal fractional Laplacian on the cylinder are given. In Section \ref{section:cylindrical}, we reformulate our problem in cylindrical variables. The main Theorems are proved in Sections \ref{sec4}-\ref{sec8}, while the Appendix contains some numerology.

\section{The conformal fractional Laplacian on the cylinder}\label{section:preliminaries}

We first set up some notation. Let $x=(r,\theta)$, $r>0$, $\theta\in\mathbb S^{n-1}$ be polar coordinates in $\mathbb R^n$ and set $r=e^{t}$. The cylinder will be denoted by $\mathcal C=\mathbb R\times\mathbb S^{n-1}$, with coordinates $t\in\mathbb R$, $\theta\in\mathbb S^{n-1}$, together with the cylindrical metric
$$g_0:=\frac{1}{r^2} |dx|^2=\frac{dr^2+r^2d\theta^2}{r^2}=dt^2+d\theta^2.$$
Here $d\theta^2$ denotes the canonical metric on $\mathbb S^{n-1}$. Let $d\mu$ be the corresponding measure on the cylinder, this is, $d\mu=dt\,d\theta$.

Following the notation in \cite{acdfgw} (see also the survey \cite{acdfgw:survey}), we define the conformal fractional Laplacian on the cylinder (first introduced in \cite{DelaTorre-Gonzalez,DelaTorre-delPino-Gonzalez-Wei}) by
\begin{equation}\label{P-cylinder}P_{\gamma}^{g_0}v(t,\theta)=\varsigma_{n,\gamma}P.V.
 \int_{\mathcal C} K(t,\tilde t,\theta,\tilde \theta)
 (v(t,\theta)-v(\tilde t,\tilde \theta))\,d\tilde\mu + c_{n,\gamma} v(t,\theta),
 \end{equation}
for
\begin{equation}\label{constants1}
\varsigma_{n,\gamma}=\pi^{-\tfrac{n}{2}}2^{2\gamma}
\tfrac{\Gamma\left(\tfrac{n}{2}+\gamma\right)}{\Gamma(1-\gamma)}\gamma\quad\text{and}\quad c_{n, \gamma}=2^{2\gamma}\left(\frac{\Gamma(\frac{1}{2}(\frac{n}{2}+\gamma))}
{\Gamma(\frac{1}{2}(\frac{n}{2}-\gamma))}\right)^2>0,
\end{equation}
and the integration kernel
\begin{equation}\label{kernel}
K(t,\tilde t,\theta,\tilde\theta)=\frac{e^{-\frac{n+2\gamma}{2}|t-\tilde t|}}
{(1+e^{-2|t-\tilde t|}-2e^{-|t-\tilde t|}\langle\theta,\tilde \theta\rangle)^{\frac{n+2\gamma}{2}}},
\quad t, \tilde t\in\mathbb R,\theta,\tilde\theta\in\mathbb S^{n-1}.
\end{equation}
To simplify the notation, we will drop the reference to the metric in the operator and we will call it simply $P_{\gamma}$. It is a non-local self-adjoint operator of order $2\gamma$ and its relevance comes from the conformal property
thus it can be understood as a conjugation of the standard fractional Laplacian on $\mathbb R^n$. However, the advantage of the operator $P_\gamma$ on the cylinder $\mathcal C$ over the Euclidean $(-\Delta)^\gamma$ is the fact that it can be easily decomposed into spherical harmonics.

With some abuse of notation, let $\vartheta_m=m(m+n-2)$, $m=0,1,2,\dots$ be the eigenvalues of $\Delta_{\mathbb S^{n-1}}$. Then any function $v(t,\theta)$, $t\in\mathbb R$, $\theta\in \mathbb S^{n-1}$ defined on the cylinder  $\mathcal C$ may be decomposed as
$$v(t,\theta)=\sum_{m,k} v_m(t) E_{m,k}(\theta),$$
where $\mathcal H_{m}:=\langle E_{m,k}:k=1,\ldots,k_m\rangle$ is the eigenspace corresponding to the eigenvalue $\vartheta_m$. Let
\begin{equation*}\label{fourier}
\hat{v}(\xi)=\frac{1}{\sqrt{2\pi}}\int_{\mathbb R}e^{-i\xi t} v(t)\,dt
\end{equation*}
be our normalization for the one-dimensional Fourier transform.

A crucial fact is that the operator $P_\gamma$ diagonalizes under such eigenspace decomposition, and each projection can be fully characterized. More precisely:

\begin{proposition}[\cite{DelaTorre-Gonzalez}]\label{prop:symbol}
 Fix $\gamma\in (0,\tfrac{n}{2})$ and let $P^{(m)}_{\gamma}$ be the projection of the operator $P_\gamma$ over each eigenspace $\langle E_{m,k}\rangle$. Then
$$\widehat{P_\gamma^{(m)} (v_m)}=\Theta^{(m)}_\gamma(\xi) \,\widehat{v_m},$$
where the Fourier symbol is given by
\begin{equation*}\label{symbol-isolated}
\Theta^{(m)}_{\gamma}(\xi)=2^{2\gamma}\frac{\Big|\Gamma\Big(\frac{n}{4}+\frac{\gamma}{2}+\frac{m}{2}
+\frac{\xi}{2}i\Big)\Big|^2}
{\Big|\Gamma\Big(\frac{n}{4}-\frac{\gamma}{2}+\frac{m}{2}
+\frac{\xi}{2}i\Big)\Big|^2}.
\end{equation*}
Moreover,
\begin{equation}\label{Pm}
 P_\gamma^{(m)}(v_m)(t)=\varsigma_{n,\gamma}\int_{\mathbb R} {\mathcal K}_m(t-\tilde t)[v_m(t)-v_m(\tilde t)]\,d\tilde t+c_{n,\gamma}v_m(t),
\end{equation}
for a convolution kernel ${\mathcal K}_m$ on $\mathbb R$
with the asymptotic behavior
\begin{equation*}
{\mathcal K}_m(t)\asymp
\begin{cases}
 |t|^{-1-2\gamma} &\mbox{ as }|t|\to 0,\\
e^{-\big(1+\gamma+\sqrt{(\frac{n-2}{2})^2+m(m+n-2)}\big)|t|} &\mbox{ as }t \to \pm\infty.
\end{cases}
\end{equation*}
In the particular case that $m=0$,
\begin{equation*}\label{kernel-K}
{\mathcal K}_0(t)=
e^{-\frac{n+2\gamma}{2}|t|}\Hyperg\big( \tfrac{n+2\gamma}{2},1+\gamma;\tfrac{n}{2};e^{-2|t|}\big).
\end{equation*}
\end{proposition}

It is interesting to observe that, given $\xi\in\mathbb R$, $\Theta_m(\xi)=\Theta_m(-\xi)$
and
\begin{equation}\label{symbol-limit}
\Theta_m(\xi)\asymp |m+\xi i|^{2\gamma},\quad \text{as}\quad |\xi|\to\infty,
\end{equation}
and this limit is uniform in $m$. This also shows that, for fixed $m$, the behavior at infinity is the same as the one for the standard fractional Laplacian $(-\Delta)^\gamma$, i.e., $|\xi|^{2\gamma}$.

This operator can also be understood as the Dirichlet-to-Neumann operator for an elliptic extension problem to a manifold in one more dimension in the spirit of \cite{Caffarelli-Silvestre,Chang-Gonzalez,Case-Chang}. For this, we need to introduce some notation. Define
\begin{equation*}\tilde d_\gamma =-
\frac{2^{2\gamma-1}\Gamma(\gamma)}{\gamma\Gamma(-\gamma)},\quad
a
=\frac{\Gamma(\frac{n}{2})\Gamma(\gamma)}
{\Gamma\big(\gamma+\frac{n}{4}-\frac{1}{4})
\Gamma\big(\frac{n}{4}+\frac{1}{4}\big)}.
\end{equation*}

The extension manifold is $X^{n+1}=(0,2)\times\mathcal C$ with coordinates $R\in(0,2)$, $t\in\mathbb R$, $\theta\in\mathbb S^{n-1}$
with canonical metric
\begin{equation*}
\bar g=dR^2+\Big(1+\tfrac{R^2}{4}\Big)^2 dt^2+\Big(1-\tfrac{R^2}{4}\Big)^2 d\theta^2.
\end{equation*}
Note that the apparent singularity at $R=2$ is of the same type  as the origin in polar coordinates. This fact will be implicitly assumed in the following exposition without further mention.

The boundary of $X^{n+1}$, given by $\{R=0\}$, is precisely the cylinder $\mathcal C$.
 Now we make the change of variables from the coordinate $R$ to
\begin{equation*}\label{rho*}
\rho(R)=
\left[a^{-1}\big(\tfrac{4R}{4+R^2}\big)^{\tfrac{n-2\gamma}{2}}
\Hyperg\Big(\tfrac{\gamma}{p-1},\tfrac{n-2\gamma}{2}-\tfrac{\gamma}{p-1};
\tfrac{n}{2};\big(\tfrac{4-R^2}{4+R^2}\big)^2\Big)\right]^{2/(n-2\gamma)},
\quad R\in(0,2).
\end{equation*}
The function $\rho$ is known as the special (or adapted) defining function. It is strictly monotone with respect to $R$, which implies that we can write $R=R(\rho)$ even if we do not have a precise formula and, in particular,
$\rho\in(0,\rho_0)$ for
\begin{equation*}\label{rho*0}
\rho_0:=\rho(2)=a^{-\frac{2}{n-2\gamma}}.
\end{equation*}
Moreover, near the conformal infinity we have the asymptotic expansion $\rho(R)=R\left[ 1+O(R^{2\gamma})\right]$.
In the new manifold $X^*=(0,\rho_0)\times\mathcal C$ consider the metric $\bar g^*:=(\frac{\rho}{R})^2\bar g$, which satisfies
\begin{equation}\label{barg*}
\bar g^*=d\rho^2(1+O(\rho^{2\gamma}))+g_0(1+O(\rho^{2\gamma})).
\end{equation}
Then for the conformal fractional Laplacian, one has the following extension construction:
\begin{proposition}[\cite{DelaTorre-Gonzalez}]\label{prop:divV*}
Let $v$ be a smooth function on $\mathcal C$.
The extension problem
\begin{equation}\label{divV*}
\left\{\begin{array}{@{}r@{}l@{}l}
 \divergence_{\bar{g}^*}(\rho^{1-2\gamma}\nabla_{\bar{g}^*}V)&\,=0\quad &\text{in } (X^*,\bar g^*), \\
V|_{\rho=0}&\,=v\quad &\text{on }\mathcal C,
\end{array}\right.
\end{equation}
has a unique solution $V$. In addition, the conformal fractional Laplacian can be recovered from the Neumann data as
\begin{equation}\label{Neumann-V*}
 P_\gamma(v)=-\tilde d_\gamma \lim_{\rho \to 0}\rho^{1-2\gamma} \partial_{\rho} V+c_{n,\gamma}v.
\end{equation}
Moreover, the first equation in \eqref{divV*} can be expanded to
\begin{equation}\label{extension1}
 \partial_{\rho}\left(e_1(\rho)\rho^{1-2\gamma}\partial_{\rho}V\right)
+e_2(\rho)\rho^{1-2\gamma}\partial_{tt}V+e_3(\rho)\rho^{1-2\gamma} \Delta_\theta V=0,
\end{equation}
for $\rho\in(0,\rho_0)$, $t\in\mathbb R$, $\theta\in \mathbb S^{n-1}$,
where $e_i(\rho)$ are some non-negative functions satisfying
\begin{equation}\label{e_i}
e_i(\rho)=1+O(\rho^2) \quad \text{as}\quad \rho\to 0,\quad i=1,2,3.
\end{equation}
\end{proposition}

\section{The inequality in cylindrical variables}\label{section:cylindrical}

Let us now go back to the fractional CKN inequality. With the new tools we have just introduced in Section \ref{section:preliminaries}, it is possible to  transform our original problem into an equivalent one defined on the cylinder $\r\times \mathbb S^{n-1}$, similarly to the classical case from \cite{Catrina_Wang2}.

The Euler-Lagrange equation for our minimization problem \eqref{eq-S} is given, in the weak form, by
\begin{equation*}
\int_{\r^n}\int_{\mathbb R^n}\frac{(u(x)-u(y))(\varphi(x)-\varphi(y))}{|x-y|^{n+2\gamma}|x|^{\alpha}|y|^{\alpha}}\,dydx=
2c\int_{\mathbb R^n}\frac{|u(x)|^{p-2}u(x)}{|x|^{\beta p}}\varphi(x)\,dx,
\end{equation*}
for every $\varphi\in\mathcal C^\infty_c(\mathbb R^n)$, which yields
\begin{equation*}\label{equationofu}
\mathcal L_{\gamma,\alpha}(u):=\int_{\r^n}\frac{u(x)-u(y)}{|x-y|^{n+2\gamma}|x|^{\alpha}|y|^{\alpha}}\,dy
=c\frac{|u(x)|^{p-2}u(x)}{|x|^{\beta p}}
\end{equation*}
for some Lagrange multiplier $c\in\mathbb R$.

First, we will consider the homogeneous singular solution to  equation \eqref{eq-extremal}. From the arguments in the Appendix we have:

\begin{proposition}\label{prop_nu}
Let
\begin{equation}\label{nu0}
\nu:=\frac{n-2\gamma}{2}-{\alpha}.
\end{equation}
Then the function $u(x)=|x|^{-\nu}$ is a solution of Euler-Lagrange equation \eqref{eq-extremal} with the constant normalized as $c=\kappa^n_{\alpha,\gamma}$ with $0<\kappa^n_{\alpha,\gamma}<\infty$ defined in \eqref{kappa}.
\end{proposition}

Now use polar coordinates for $x\in\mathbb R^n$ ($r=|x|>0$, $\theta\in\mathbb S^{n-1}$), and set $r=e^{t}$. We  fix the value $\nu$ as in  \eqref{nu0} for the rest of the paper.

In the light of Proposition \ref{prop_nu}, we can write any function $u\in D^{\gamma}_{\alpha}$  as
\begin{equation}\label{uv4_1}
u(x)=|x|^{-\nu}v\Big(\log r,\frac{x}{|x|}\Big),
\end{equation}
where {$v\in \tilde{D}^{\gamma}_{\alpha,\beta}$}; this functional space will be explained in \eqref{tildeD} below. One can see that the space $\tilde{D}^{\gamma}_{\alpha,\beta}$ is independent of the parameters $\alpha,\beta$, so it will be denoted simply by $\tilde{D}^{\gamma}$.

We start with a preliminary result:

\begin{lemma}\label{lemma:conjugation}
In the notation above one has
\begin{equation}\label{equation100}
\begin{split}
r^{\frac{n+2\gamma}{2}+\alpha}\mathcal L_{\gamma,\alpha} u(r,\theta)&=
\kappa_{\alpha,\gamma}^{n}v(t,\theta)+\int_{\mathcal C}
K(t,\tilde t,\theta,\tilde\theta)(v(t,\theta)-v(\tilde t,\tilde \theta))\,d\tilde\mu,
\end{split}
\end{equation}
where the kernel is given in \eqref{kernel} and $d\tilde\mu=\,d\tilde{\theta}\,d\tilde t$ is the cylindrical volume element. This is equivalent to
\begin{equation}\label{equation10}
\varsigma_{n,\gamma}r^{\frac{n+2\gamma}{2}+\alpha}\mathcal L_{\gamma,\alpha} u=(\varsigma_{n,\gamma}\kappa_{\alpha,\gamma}^n-c_{n,\gamma})v+P_\gamma v.
\end{equation}
\end{lemma}

\begin{proof}
This is a simple calculation. In polar coordinates ($r=|x|,\ \theta\in\s^{n-1}$ and $s=|y|,\ \tilde\theta\in\s^{n-1}$),
\begin{equation*}
\mathcal L_{\gamma,\alpha}u(r,\theta)=\int_{\s^{n-1}}\int_0^{\infty}
\frac{(r^{-\nu}v(r,\theta)-s^{-\nu}v(s,\tilde\theta))s^{n-1-\alpha}r^{-\alpha}}
{(s^2+r^2-2sr\langle\theta,\tilde\theta\rangle)^{\frac{n+2\gamma}{2}}}\,ds\,d\tilde\theta,
\end{equation*}
which, after the change of variable $\bar{s}=\tfrac{s}{r}$, is equivalent to
\begin{equation*}
\begin{split}
r^{\frac{n+2\gamma}{2}+\alpha}\mathcal L_{\gamma,\alpha}u(r,\theta)&=
\int_{\s^{n-1}}\int_0^{\infty}\frac{(v(r,\theta)-\bar{s}^{-\nu}v(r\bar{s},\sigma))\bar{s}^{n-1-\alpha}}
{(1+\bar{s}^2-2\bar{s}\langle\theta,\sigma\rangle)^{\frac{n+2\gamma}{2}}}\,d\bar{s}\,d\sigma \\
&=\kappa_{\alpha,\gamma}^{n}v(r,\theta)
+\int_{0}^{\infty}\int_{\s^{n-1}}\frac{(v(r,\theta)-v(r\bar{s},\sigma))\bar{s}^{n-1-\alpha-\nu}}
{(1+\bar{s}^2-2\bar{s}\langle\theta,\sigma\rangle)^{\frac{n+2\gamma}{2}}}\,d\sigma\,d\bar{s},
\end{split}
\end{equation*}
where we have applied the trivial equality $$v(r,\theta)=(1-\bar{s}^{-\nu})v(r,\theta)+\bar{s}^{-\nu}v(r,\theta).$$
Next, using the  Emden-Fowler change of variable ($r=e^t$, $s=e^{\tilde t}$ and thus, $\bar{s}=e^{-(t-\tilde t)}$), we arrive to \eqref{equation100}.

The second assertion follows from \eqref{P-cylinder} by simple inspection.

\end{proof}

%\begin{lemma}\label{lemma:transformation} We have
%\begin{equation*}
%\varsigma_{n,\gamma}\mathcal L_{\gamma,\alpha} u=
%(\varsigma_{n,\gamma}\kappa_{\alpha,\gamma}^n-c_{n,\gamma})r^{-2\gamma-\alpha}u
%+(-\Delta)^\gamma (r^{-\alpha}u),
%\end{equation*}
%and for $v=v(t,\theta)$, $u=r^{-\nu} v$,
%\end{lemma}

%\begin{proof}
%Similarly to the proof of Lemma \ref{lemma:conjugation}, we calculate
%\begin{equation*}
%\begin{split}
%(-\Delta)^\gamma&(r^{-\nu-\alpha} v(r,\theta))\\
%&=r^{-\nu-\alpha-\gamma}
%\left[\tilde\kappa^n_{\alpha,\gamma} v(r,\theta)+\int_{\mathbb S^{n-1}}\int_0^\infty
%\frac{e^{-(\nu-n+\alpha)(t-\tau)}(v(t,\theta)-v(\tau,\sigma))}
%{(1+e^{2(t-\tau)}-2e^{(t-\tau)}<\theta,\sigma\rangle)^{\frac{n+2\gamma}{2}}}\,d\sigma\,d\tau.\right],
%\end{split}
%\end{equation*}
%for the constant $\tilde \kappa^n_{\alpha,\gamma}=\kappa^{n,\nu+\alpha-n+1}_{n-1,\gamma}$.
%By simple inspection, from  Lemma \ref{lemma:conjugation}, we have
%\begin{equation*}
%\varsigma_{n,\gamma} \mathcal L u= r^{-\nu-2\gamma-2\alpha} \left[ %(\varsigma_{n,\gamma}\kappa_{\alpha,\gamma}^n-c_{n,\gamma})v+P_\gamma^{g_0}v\right].
%\end{equation*}
%Now use the conformal property of the operator $P_\gamma$,
%\begin{equation*}
%(-\Delta)^\gamma(r^{-\alpha} u)=r^{-\frac{n+2\gamma}{2}} P_{\gamma} (r^{\frac{n-2\gamma}{2}-\alpha} u).
%\end{equation*}
%Substituting above we complete the proof of the Lemma.
%\end{proof}

Now, from the expression in \eqref{equation10}, it is natural to define
\begin{equation*}\label{linearized operator}
\tilde{\mathcal L}_{\gamma,\alpha} v := \varsigma_{n,\gamma}r^{\frac{n+2\gamma}{2}+\alpha}\mathcal L_{\gamma,\alpha} (r^{-\nu}v)
\end{equation*}
for a function  $v(t,\theta)$ on the cylinder given by \eqref{uv4_1}. From Lemma \ref{lemma:conjugation}, we have the following formulas:
\begin{equation}\label{tilde-l}
\begin{split}
\tilde{\mathcal L}_{\gamma,\alpha}v(t,\theta)&
=(\varsigma_{n,\gamma}\kappa_{\alpha,\gamma}^n-c_{n,\gamma})v(t,\theta)+P_\gamma v(t,\theta)\\
&=\varsigma_{n,\gamma}
\kappa_{\alpha,\gamma}^{n}v(t,\theta)+\varsigma_{n,\gamma}\int_{\mathcal C}K(t,\tilde t,\theta,\tilde\theta)(v(t,\theta)-v(\tilde t,\tilde \theta))
\,d\tilde\mu.
\end{split}
\end{equation}
Moreover,
 the Euler-Lagrange equation \eqref{eq-extremal}{ with the constant normalized as $c=\kappa^n_{\alpha,\gamma}$,} can be written as
\begin{equation}\label{EL_cylindrical}
\tilde{\mathcal L}_{\gamma,\alpha} v=\varsigma_{n,\gamma}\kappa_{\alpha,\gamma}^n |v|^{p-2}v.
\end{equation}
%
\begin{comment}
In the particular case that $v=v(t)$ is radially symmetric, this equation reduces to
\begin{equation*}\label{EL_ode}
\int_{\mathbb R}(v(t)-v(\tau))\mathcal K_0(t-\tau)\,d\tau
+\kappa_{\alpha,\gamma}^{n}v(t)=\kappa_{\alpha,\gamma}^n |v(t)|^{p-2}v(t),
\end{equation*}
where the integral kernel is given in \eqref{kernel-K}.
\textcolor{red}{maybe there is a $\varsigma_{n,\gamma}$ missing?}\\
\end{comment}
For simplicity in the notation in expression \eqref{tilde-l}, we will call
$$C(\alpha):=\varsigma_{n,\gamma}\kappa_{\alpha,\gamma}^n-c_{n,\gamma}.$$
Additionally, we can also write the equivalent formulation of the Euler-Lagrange equation using the extension problem from Proposition \ref{prop:divV*}. More precisely, \eqref{EL_cylindrical} is equivalent to
\begin{equation*}\label{EL-extension}
\left\{\begin{split}
& \partial_{\rho}\left(e_1(\rho)\rho^{1-2\gamma}\partial_{\rho}V\right)
+e_2(\rho)\rho^{1-2\gamma}\partial_{tt}V+e_3(\rho)\rho^{1-2\gamma} \Delta_\theta V=0\quad\text{in }X^*,\\
&-\tilde d_\gamma\lim_{\rho\to 0}\rho^{1-2\gamma}\partial_\rho V(\rho,t)+\varsigma_{n,\gamma}\kappa_{\alpha,\gamma}^n v-\varsigma_{n,\gamma}\kappa_{\alpha,\gamma}^n|v|^{p-2}v=0 \quad\mbox{on }\mathcal C,
\end{split}\right.	
\end{equation*}
with $e_i$ as in \eqref{e_i}.

Next, we give a formula for the energy \eqref{defi-E} on the cylinder using similar ideas. Recall the value of $\nu$ from  \eqref{nu0}. Taking polar coordinates as in Lemma \ref{lemma:conjugation},
\begin{equation*}
\begin{split}
\|u\|_{\gamma,\alpha}^2&=
\int_{\r^n}\int_{\r^n}\frac{(u(x)-u(y))^2}{|x-y|^{n+2\gamma}|x|^{{\alpha}}|y|^{{\alpha}}}\,dy\,dx\\
&=\displaystyle\int_{\s^{n-1}}\int_{\s^{n-1}}\int_0^{\infty}r^{n-1-2(\alpha+\gamma+\nu)}
\int_0^{\infty}\frac{(v(r,\theta)-\bar{s}^{-\nu}v(r\bar{s},\tilde\theta))^2\bar{s}^{n-1-\alpha}}
{(1+\bar{s}^2-2\bar{s}\langle\theta,\tilde\theta\rangle)^{\frac{n+2\gamma}{2}}}\,d\bar{s}\,dr\,d\tilde\theta\,d\theta
\end{split}\end{equation*}
which, using the trivial equalities
$$v^2(r,\theta)=(1-\bar{s}^{-\nu})v^{2}(r,\theta)+\bar{s}^{-\nu}v^2(r,\theta)$$
 and
 $$v^2(r\bar{s},\tilde\theta)=(1-\bar{s}^{\nu})v^2(r\bar{s},\tilde\theta)+\bar{s}^{\nu}v^2(r\bar{s},\tilde\theta)$$
  is equivalent to
\begin{equation*}
\begin{split}
\|u\|_{\gamma,\alpha}^2&=\kappa_{\alpha,\gamma}^{n}\int_{\s^{n-1}}\int_0^{\infty}r^{n-1-2(\alpha+\gamma+\nu)}v^2(r,\theta)\,dr\,d\theta\\
&\quad+\int_{\s^{n-1}}\int_{\s^{n-1}}\int_0^{\infty}\int_0^{\infty}
\frac{r^{n-1-2(\alpha+\gamma+\nu)}v^2(r\bar{s},\sigma)\bar{s}^{n-1-\alpha-2\nu}(1-\bar{s}^{\nu})}
{(1+\bar{s}^2-2\bar{s}\langle\theta,\tilde\theta\rangle)^{\frac{n+2\gamma}{2}}}\,dr\,d\bar{s}\,d\theta\,d\tilde\theta\\
&\quad+\int_{\s^{n-1}}\int_{\s^{n-1}}\int_0^{\infty}r^{n-1-2(\alpha+\gamma+\nu)}
\int_{0}^{\infty}\frac{(v(r,\theta)-v(r\bar{s},\tilde\theta))^2\bar{s}^{n-1-\alpha-\nu}}
{(1+\bar{s}^2-2\bar{s}\langle\theta,\tilde \theta\rangle)^{\frac{n+2\gamma}{2}}}\,d\bar{s}\,dr\,d\theta\,d\tilde\theta.
\end{split}\end{equation*}
Now, we change variables  $s=\bar{s}r$ and $\bar{r}=rs^{-1}$ on the second integral in the right hand side in the formula above. Recalling the value $\kappa_{\alpha,\gamma}^{n}$, we have
\begin{equation*}\begin{split}
\|u\|_{\gamma,\alpha}^2&= 2\kappa_{\alpha,\gamma}^{n}\int_{\s^{n-1}}\int_0^{\infty}r^{-1}v^2(r,\theta)\,dr\,d\theta\\
%&+\kappa_{\alpha+2\nu,\gamma}^{n,-\nu}\int_{\s^{n-1}}\int_0^{\infty}r^{n-1-2(\alpha+\gamma+\nu)}v^2(r\bar{s},\sigma)\,d\bar{s}\,d\sigma\\Es cierto, pero no llegamos hasta aqui para poder reducir y quedarnos con un solo termino.
&+\int_{\s^{n-1}}\int_{\s^{n-1}}\int_0^{\infty}r^{-1}
\int_{0}^{\infty}\frac{(v(r,\theta)-v(r\bar{s},\tilde\theta))^2\bar{s}^{n-1-\alpha-\nu}}
{(1+\bar{s}^2-2\bar{s}\langle\theta,\tilde\theta\rangle)^{\frac{n+2\gamma}{2}}}\,d\bar{s}\,dr\,d\theta\,d\tilde\theta.
\end{split}\end{equation*}
In cylindrical coordinates ($r=e^t$, $s=e^{\tau}$ and $\bar{s}=e^{-(t-\tau)}$),  it becomes
\begin{equation*}
\begin{split}
\|u\|_{\gamma,\alpha}^2=2\kappa_{\alpha,\gamma}^{n}\displaystyle\int_{\mathcal C}v^2(t,\theta)\,d\mu
+\int_{\mathcal C}\int_{\mathcal C}K(t,\tilde t,\theta,\tilde\theta)(v(t,\theta)-v(\tilde t,\tilde\theta))^2\,d\mu\,d\tilde\mu,
\end{split}\end{equation*}
while, in addition,
\begin{equation*}
\int_{\r^n}\frac{|u|^{p}}{|x|^{\beta p}}\,dx=\int_{\mathcal C} |v(t,\theta)|^p\,d\mu.
\end{equation*}
Thus we can define an energy functional on the cylinder $\mathcal C$ by
\begin{equation}\label{eq-F}
\begin{split}
F_{\alpha,\beta}(v)=\frac{\displaystyle{\int_{\mathcal C}}\int_{\mathcal C}
K(t,\tilde t,\theta,\tilde\theta)(v(t,\theta)-v(\tilde t,\tilde\theta))^2\,d\mu d\tilde \mu+2\kappa_{\alpha,\gamma}^n\displaystyle{\int_{\mathcal C}}v^2\,d\mu}
{\Big(\displaystyle{\int_{\mathcal C}}|v|^p\,d\mu\Big)^{2/p}}.
\end{split}
\end{equation}

Finally, note that one can write $F_{\alpha,\beta}$ as
\begin{equation}\label{functional-F}
F_{\alpha,\beta}(v)=  2\varsigma^{-1}_{n,\gamma}\,\frac{\displaystyle\int_{\mathcal C} v \tilde{\mathcal L}_{\gamma,\alpha} v\,d\mu}
{\Big(\displaystyle{\int_{\mathcal C}}|v|^p\,d\mu \Big)^{2/p}}.
\end{equation}
This follows simply by taking into account that
\begin{equation*}
(v(t,\theta)-v(\tilde t,\tilde \theta))^2=
v(t,\theta)[v(t,\theta)-v(\tilde t,\tilde \theta)]+
v(\tilde t,\tilde \theta)[v(\tilde t,\tilde \theta)-v(t,\theta)],
\end{equation*}
and the fact that the integral kernel inside the double integral $\int_{\mathcal C}\int_{\mathcal C}$ is symmetric under the change $t\to\tilde t$ and $\theta \to \tilde\theta$.

\bigskip

%To make simpler the expression for the ODE \eqref{EL_ode}, we choose
%\begin{equation}\label{nu}
%-2^*_{\gamma}(\nu+\beta)+2(\nu+\alpha+\gamma)=0
%\nu:=\tfrac{-2^*_{\gamma}\beta+2(\alpha+\gamma)}{2-2^*_{\gamma}}.
%\end{equation}
%Thus, \eqref{EL_ode} becomes
% \begin{equation}\label{EL_ode_nu}
 %\kappa_{\alpha,\gamma}^n |v(t)|^{2^*_{\gamma}-2}v(t)=\kappa_{\alpha,\gamma}^{n,\nu}v(t)+\int_{-\infty}^{\infty}(v(t)-v(\tau))K(t-\tau)\,d\tau.
 %\end{equation}

\subsection{Function spaces and the proof of Proposition \ref{prop:inequality}}\label{subsection:function-spaces}

%\textcolor{green}{based in Catrina-Wang}

Recall that we have defined  $D_{\alpha}^\gamma(\mathbb R^n)$ to be the completion of $C^\infty_c(\mathbb R^n)$ with respect to the inner product
\begin{equation*}
\langle u_1,u_2 \rangle_{\gamma,\alpha}:=
\int_{\r^n}\int_{\r^n}\frac{(u_1(x)-u_2(y))(u_1(x)-u_2(y))}{|x-y|^{n+2\gamma}|x|^{{\alpha}}|y|^{{\alpha}}}\,dy\,dx,
\end{equation*}
with associated norm
\begin{equation*}
\|u\|^2_{\gamma, \alpha}=	\int_{\r^n}\int_{\r^n}\frac{(u(x)-u(y))^2}{|x-y|^{n+2\gamma}|x|^{{\alpha}}|y|^{{\alpha}}}\,dy\,dx.
\end{equation*}
Define also the weighted space $L^p_\beta(\mathbb R^n)$ as the completion of $C^\infty_c(\mathbb R^n)$ with respect to the norm
\begin{equation*}
\|u\|_{L^p_\beta}:=\left(\int_{\r^n}\frac{|u(x)|^{p}}{|x|^{{\beta} {p}}}\,dx\right)^{\frac{1}{p}}.
\end{equation*}
As in Lemma 2.1 in \cite{Catrina_Wang2}, one can easily verify that both spaces can be understood as the completion of $\mathcal C^\infty_c(\mathbb R^n\setminus\{0\})$ under their respective norms. In fact, from the embedding result below, one has
\begin{equation*}
D_{\alpha}^\gamma(\r^n)=\{u\in L_\beta^p(\r^n)\,:\,  \|u\|_{\gamma,\alpha}<\infty\}.	
\end{equation*}

Now define the  functional space on the cylinder by
\begin{equation}\label{tildeD}
\tilde{D}^{\gamma}(\mathcal{C}):=\{v\in L^2(\mathcal C):\ \int_{\mathcal C}
\int_{\mathcal C}K(t,\tilde t,\theta,\tilde\theta)(v(t,\theta)-v(\tilde t,\tilde\theta))^2
\,d\mu\,d\tilde\mu<\infty\},
\end{equation}
given in terms of the scalar product
\begin{equation}\label{scalar-product}
\begin{split}
\langle v_1, v_2\rangle :&=2\kappa_{\alpha,\gamma}^{n}\displaystyle\int_{\mathcal C}v_1(t,\theta)v_2(t,\theta)\,d\mu \\
&+\int_{\mathcal C}\int_{\mathcal C}K(t,\tilde t,\theta,\tilde\theta)(v_1(t,\theta)-v_1(\tilde t,\tilde\theta))(v_2(t,\theta)-v_2(\tilde t,\tilde\theta))\,d\mu\,d\tilde\mu.
\end{split}
\end{equation}
From the above definition it is clear that the space $\tilde{D}^\gamma(\mathcal C)$ is independent of the parameters $\alpha$ and $\beta$.

Moreover, the above calculations show that the Hilbert spaces $D_{\alpha}^\gamma(\mathbb R^n)$ and $\tilde D^\gamma(\mathcal C)$ are isomorphic. We have the following:

%\textcolor{green}{there is a big of a mess in the notation for the function space for $u$: $D^\gamma_{\alpha,\beta}$ or $D^\gamma_{\alpha}$?}

\begin{lemma}\label{lemma:relation}
The mapping given by \eqref{uv4_1} is an isomorphism from $D_{\alpha}^\gamma(\r^n)$ to $\tilde{D}^{\gamma}(\mathcal{C})$.
Moreover, if  $u\in D_{\alpha}^\gamma (\mathbb R^n)$ and $v\in \tilde D^\gamma(\mathcal C)$ are related by \eqref{uv4_1}, then
\begin{equation*}
E_{\alpha,\beta}(u)=F_{\alpha,\beta}(v)
\end{equation*}
and
\begin{equation}\label{eq-S1}
S(\alpha,\beta)=\inf_{v\in \tilde{D}^\gamma(\mathcal C)\setminus\{0\}}F_{\alpha,\beta}(v).
\end{equation}
Similarly, the solutions of \eqref{eq-extremal} and \eqref{EL_cylindrical} are one-to-one.

\end{lemma}

\medskip

Let us rewrite the energy in $\tilde D^\gamma(\mathcal C)$ in terms of the extension problem \eqref{extension1}. With some abuse of notation, $d\rho d\mu$ will denote the volume element of the metric $\bar g^*$ given in \eqref{barg*}. We have
\begin{equation*}
\begin{split}
0&= \int_0^{\rho_0}\int_{\mathcal C}V\left\{-\partial_{\rho}\left(e_1(\rho)\rho^{1-2\gamma}\partial_{\rho}V\right)
-e_2(\rho)\rho^{1-2\gamma}\partial_{tt}V-e_3(\rho)\rho^{1-2\gamma} \Delta_\theta V\right\}\,d\rho d\mu\\
&=\int_0^{\rho_0}\int_{\mathcal C} \rho^{1-2\gamma}\left\{e_1(\rho)(\partial_{\rho}V)^2
+e_2(\rho)(\partial_{t}V)^2+e_3(\rho) |\nabla_\theta V|^2\right\}\,d\rho d\mu\\
&+I_{\mathcal C},
\end{split}
\end{equation*}
 where
\begin{equation*}
I_{\mathcal C}=\int_{\mathcal C} v\rho^{1-2\gamma}\partial_\rho V|_{\rho=0}\,d\mu.
\end{equation*}
For the boundary term, equation \eqref{Neumann-V*} yields
\begin{equation*}
\begin{split}
\tilde{d}_{\gamma}I_{\mathcal C}&=-\int_{\mathcal C}vP_\gamma v\,d\mu+c_{n,\gamma}\int_{\mathcal C} v^2\,d\mu\\
& =-\varsigma_{n,\gamma}\int_{\mathcal C} v(t,\theta)\int_{\mathcal C} K(t,\tilde t,\theta,\tilde \theta)
(v(t,\theta)-v(\tilde t,\tilde \theta))\,d\tilde \mu d\mu\\
&=\frac{-\varsigma_{n,\gamma}}{2}\int_{\mathcal C}\int_{\mathcal C} K(t,\tilde t,\theta,\tilde\theta)(v(t,\theta)-v(\tilde t,\tilde\theta))^2\,d\mu d\tilde\mu,
\end{split}
\end{equation*}
where we have used the symmetry of the kernel in the last equality.
We conclude that
\begin{equation}\label{energy-extension}
\begin{split}
\| v\|^2_{\tilde D^\gamma(\mathcal C)}&=2\kappa_{\alpha,\gamma}^{n}\displaystyle\int_{\mathcal C}v^2\,d\mu\\
+& \frac{2\tilde d_\gamma}{\varsigma_{n,\gamma}}\int_{0}^{\rho_0}\int_{\mathcal C} \rho^{1-2\gamma}\left\{e_1(\rho)(\partial_{\rho}W)^2
+e_2(\rho)(\partial_{t}W)^2+e_3(\rho)  |\nabla_\theta W|^2\right\}\,d\rho d\mu.
\end{split}
\end{equation}

Sobolev embeddings for the (fractional-order) Hilbert space $H^\gamma$ on a complete manifold for integer powers $\gamma$ have been well studied (see, for instance, \cite{Hebey} where they give examples of manifolds where the embedding is false).
 However, the literature is not so extensive in the case of fractional operators. In particular, usual embeddings are true for manifolds with bounded geometry \cite{Grosse-Schneider}, for which the proof is reduced to that of $\mathbb R^n$ by using normal coordinates around a point. See also \cite{Banica-Gonzalez-Saez} for the characterization of the fractional Laplacian on a manifold via extension to one more dimension in the spirit of \cite{Caffarelli-Silvestre}.

 In our case, we are dealing with the space $\tilde D^\gamma(\mathcal C)$, which is not exactly $H^\gamma(\mathcal C)$ since the integration kernel for the fractional order operator is different. In any case, we still have the usual embeddings as in the Euclidean setting due to the specific form of the energy. More precisely:

\begin{proposition}\label{embedding}
We have the continuous embedding
$$\tilde D^\gamma(\mathcal C) \hookrightarrow L^q(\mathcal C),\quad q\in[1,2_\gamma^*].$$
The embedding is compact for $q\in[1,2_\gamma^*)$, if we restrict to compact subsets in $\mathcal C$.
\end{proposition}

\begin{proof}

 Inspired by the ideas of \cite{Banica-Gonzalez-Saez} and \cite{DelaTorre-Gonzalez}, we write the norm in $\tilde D^\gamma(\mathcal C)$ as a local weighted energy in the extension to one more dimension, which is given by   \eqref{energy-extension}.

Recalling the asymptotic behavior of the $e_i(\rho)$, $i=1,2,3$, from \eqref{e_i}, this norm is equivalent to the weighted norm in $W^{1,2}(X^*,\rho^{1-2\gamma})$, and we can use the  standard trace embedding in the extension $X^*$ to complete the proof.

\end{proof}

\begin{remark}\label{proof_ineq}
 Proposition \ref{prop:inequality} is an immediate consequence of Proposition \ref{embedding}.
 \end{remark}

In the radially symmetric case, the above compactness result yields the existence of minimizer in the radial symmetry class. For this, we define the radially symmetric function space which we denote by $\tilde{D}_{r}^\gamma$, i.e. the subspace of $\tilde{D}^\gamma$ which contains functions that are independent of $\theta$. The scalar product in this new space is given by \eqref{scalar-product} simply by substituting the kernel $K$ by the one-dimensional $\mathcal K_0$ and all the above arguments follow similarly.

\begin{proposition}\label{prop:radial-minimizer}
There always exists a minimizer of $F_{\alpha,\beta}$ in $\tilde{D}_{r}^\gamma$.
\end{proposition}

\begin{proof}
By the definition of $S(\alpha, \beta)$  and \eqref{eq-S1}, we know that $S(\alpha, \beta)\geq 0$ is finite and there exists a minimizing sequence $(v_n)$ in $\tilde{D}_r^\gamma$ such that $\|v_n\|_{L^p(\mathcal C)}=1$ and
\begin{equation*}
\lim_{n\to \infty}F_{\alpha, \beta}(v_n)=S(\alpha, \beta).
\end{equation*}
For all $\gamma \in (0,1)$, we have the compact embedding of $\tilde{D}_{r}^\gamma$ into $L^q$ for $q\in\big(1,\frac{2}{1-2\gamma}\big)$ if $\gamma\leq \frac{1}{2}$ and any $q\geq 1$ for $\frac{1}{2}<\gamma<1$. In our case, $p=\frac{2n}{n-2\gamma+2(\beta-\alpha)}$. So when $n\geq 2$, this exponent is subcritical for dimension $1$ . So there exists $v_*\in \tilde{D}_{r}^\gamma$ such that  $(v_n)$ converges to $v_*$ weakly. This implies that
\begin{equation*}
\int_{\mathbb R}\int_{\mathbb R}\mathcal K_0(t-t')(v_*(t)-v_*(t'))^2\,dtdt'+2\kappa_{\alpha,\gamma}^n \int_{\mathbb R} v_*(t)^2\,dt\leq \liminf_{n} F_{\alpha,\beta}(v_n).
\end{equation*}
Moreover, by the compact embedding, we have $\|v_*\|_{L^p}=\displaystyle\lim_{n\to \infty}
\|v_n\|_{L^p}=1$, which yields that $F_{\alpha,\beta}(v_*)=S(\alpha,
\beta)$. Thus $v_*$ is an extremal solution.\\
\end{proof}

From now on, we denote the value of the optimal constant in the radially symmetric case by
\begin{equation}\label{min-radial}
R(\alpha, \beta)=\inf_{v\in \tilde{D}_r^\gamma\setminus\{0\}}F_{\alpha, \beta}(v).	
\end{equation}
Contrary to the local case, where minimizers are positive solutions for a nonlinear second-order ODE which are explicitly known (see Proposition $2.6$ in \cite{Catrina_Wang2} and the introduction of Section $2$ in \cite{Felli_Schneider}),  in the non local setting  it will not be possible to find the explicit expression of such extremal solutions.\\

\begin{remark}\label{remark:restriction}
Note that, compared with the local case $\gamma=1$, in the fractional case the region for  the parameter $\alpha$ has extra constraint $\alpha>-2\gamma$. This  is a technical constraint due to the definition of fractional Laplacian as a singular integral.  Indeed, this additional restriction is only needed in order to apply the operator \eqref{formula-L} to an homogeneous distribution $|x|^{\nu}$ for $\nu$   as in \eqref{nu0}.
It should be possible to extend our results to the whole range $\alpha<\frac{n-2\gamma}{2}$.  As it happens with the standard fractional Laplacian operator, one can admit more singular distributions by giving a different regularization in the definition of $(-\Delta)^\gamma$, but we will not consider this case in order to avoid unnecessary technicalities.
\end{remark}

\section{Proof of Theorem \ref{thm1}}\label{sec4}

In the previous section we have shown the embedding theory for the function space $\tilde{D}^\gamma$. Now we will follow the arguments in \cite{Catrina_Wang2} for the classical CKN inequality to study the best constant and existence of extremal solutions for the energy $E_{\alpha,\beta}(u)$ given by \eqref{defi-E}. The main idea is to use instead the energy in cylindrical coordinates $F_{\alpha,\beta}$, which is equivalent to $E_{\alpha,\beta}(u)$ by Lemma \ref{lemma:relation}.

Before we start with the proof of the Theorem let us show a simple symmetrization result, which is essentially contained in \cite{Ghoussoub-Shakerian}. However, we have decided to include it here, written in the notation of functional \eqref{eq-S}:

\begin{proposition}\label{prop:symmetrization} Let  $0< {\alpha} < \tfrac{n-2\gamma}{2}$
 and ${\alpha}\leq {\beta} < {\alpha}+\gamma$, or $\alpha=0$ and $ 0<\beta<\gamma$.
If an extremal solution for \eqref{eq-S} exists, then it is radially symmetric and non-increasing in the radial variable.
\end{proposition}

\begin{proof}
Let $u\in \mathcal{C}^{\infty}_c(\r^n\setminus\{ 0\})$ and consider
\begin{equation*}\label{wu_CKN}
w(x)=|x|^{-\alpha}u(x).
\end{equation*}
As in the proof of Lemma \ref{lemma:conjugation},
\begin{equation}\label{ineq_w}
\begin{split}
E_{\alpha,\beta}(u)&=\frac{ \displaystyle\int_{\r^n}\int_{\r^n}\frac{(|x|^{\alpha} w(x)-|y|^{\alpha} w(y))^2}{|x-y|^{n+2\gamma}|x|^{{\alpha}}|y|^{{\alpha}}}\,dy\,dx}
{\displaystyle \left(\int_{\r^n}\frac{|w(x)|^{p}}{|x|^{({\beta}-{\alpha}) {p}}}\,dx\right)^{2/p}}\\
=&\frac{\displaystyle 2\int_{\r^n}\int_{\r^n}\frac{(|x|^{\alpha} -|y|^{\alpha})w^2(x)}{|x-y|^{n+2\gamma}|y|^{{\alpha}}}\,dy\,dx+\int_{\r^n}\int_{\r^n}\frac{(w(x)-w(y))^2}{|x-y|^{n+2\gamma}}\,dy\,dx}
{\displaystyle \left(\int_{\r^n}\frac{|w(x)|^{p}}{|x|^{({\beta}-{\alpha}) {p}}}\,dx\right)^{2/p}},
\end{split}
\end{equation}
where we have used that $$w^2(x)=\left(1-\left|\frac{x}{y}\right|^{-{\alpha}}\right)w^{2}(x)+\left|\frac{x}{y}\right|^{-{\alpha}}w^2(x)\quad\text{and}\quad w^2(y)=\left(1-\left|\frac{y}{x}\right|^{-{\alpha}}\right)w^{2}(y)+\left|\frac{y}{x}\right|^{-{\alpha}}w^2(y).$$
Now, the first integral term in the expression \eqref{ineq_w} can be written as $2\int_{\r^N}w^2(x)I(x)\,dx$, where $I$ is the integral from Lemma \ref{integral:study} when  $\bar{\alpha}=-\alpha$. Thus {the energy expansion} \eqref{ineq_w} is equivalent to
\begin{equation*}\label{ineq_w2}
E_{\alpha,\beta}(u)= \frac{\displaystyle 2\kappa^{n,-\alpha}_{{\alpha},\gamma}\int_{\r^n}\frac{w^2(x)}{|x|^{2\gamma}}\,dx+\int_{\r^n}\int_{\r^n}\frac{(w(x)-w(y))^2}{|x-y|^{n+2\gamma}}\,dy\,dx}
{\displaystyle\left(\int_{\r^n}\frac{|w(x)|^{p}}{|x|^{({\beta}-{\alpha})p}}\,dx\right)^{\frac{2}{p}}}.
\end{equation*}
Let $\tilde w$ be the  decreasing rearrangement of $w$. A standard rearrangement inequality (see, for instance,  Theorem $3.4$ in Chapter $3$ of \cite{Lieb-Loss}) yields that
$$\int_{\r^n}\frac{|\tilde{w}(x)|^{p}}{|x|^{({\beta}-{\alpha}) p}}\,dx\geq\int_{\r^n}\frac{|w(x)|^{p}}{|x|^{({\beta}-{\alpha}) p}}\,dx\quad \text{and}\quad\int_{\r^n}\frac{|\tilde{w}(x)|^{2}}{|x|^{2\gamma}}\,dx\geq\int_{\r^n}\frac{|w(x)|^{2}}{|x|^{ 2{\gamma}}}\,dx.$$
In addition,  the fractional Polya-Szeg\"o inequality (see, for instance, the note \cite{Park}) implies that
$$\int_{\r^n}\int_{\r^n}\frac{(\tilde{w}(x)-\tilde{w}(y))^2}{|x-y|^{n+2\gamma}}\,dy\,dx\leq\int_{\r^n}\int_{\r^n}\frac{(w(x)-w(y))^2}{|x-y|^{n+2\gamma}}\,dy\,dx.$$
Because of Corollary \ref{c} we have that $\kappa^{n,-\alpha}_{{\alpha},\gamma}<0$ for ${\alpha}\in[0,\frac{n-2\gamma}{2})$, and we
conclude that the symmetrization decreases the energy $E_{\alpha,\beta}$. In addition, if  equality is attained by a function $w$, then  $u=|x|^{\alpha}w$ reaches equality in \eqref{ineq_w},  i.e., $E_{\alpha,\beta}(u)=E_{\alpha,\beta}(\tilde u)$, and by Theorem $3.4$ in \cite{Lieb-Loss} we must have $\tilde w=w$, which completes the proof.
%{\color{magenta}When $\kappa^{n,-\alpha}_{{\alpha},\gamma}=0$ the energy is non-increasing and the proof follows the same.}
\end{proof}

Now we recall the two following technical lemmas that correspond to Lemmas 3.1 and 3.2 in   \cite{Catrina_Wang2}. Their proofs are similar taking into account the embeddings from Proposition \ref{embedding} in the fractional case and therefore we skip them.

\begin{lemma}\label{lemma1}
Let $-2\gamma<\alpha_0<\frac{n-2\gamma}{2}$, $\alpha_0\leq \beta_0\leq \alpha_0+\gamma$. Then
\begin{equation*}
\limsup_{(\alpha,\beta)\to (\alpha_0,\beta_0)}S(\alpha,\beta)\leq S(\alpha_0,\beta_0).
\end{equation*}
\end{lemma}

\begin{lemma}\label{lemma2}
Let $(p_n)$ be a sequence convergent to $p$ satisfying $p_n\in [2, p]$. If a sequence $(v_n)$ is uniformly bounded in $\tilde{D}^\gamma(\mathcal C)$, then
\begin{itemize}
\item[(i)] if $p \in (2, 2_\gamma^*)$, we have
\begin{equation*}
\lim_{n\to \infty}\int_{\mathcal C}(|v_n|^{p_n}-|v_n|^p)\,d\mu=0;
\end{equation*}
\item[(ii)]if $p=2$ or $p=2_\gamma^*$, we have
\begin{equation*}
\lim_{n\to \infty}\int_{\mathcal C}(|v_n|^{p_n}-|v_n|^p)\,d\mu\leq 0.
\end{equation*}
\end{itemize}
\end{lemma}

\noindent{\bf Proof of Theorem \ref{thm1} \emph{i}. } By Lemma \ref{lemma1}, one has
\begin{equation*}
\limsup_{(\alpha,\beta)\to (\alpha_0,\beta_0)}S(\alpha,\beta)\leq S(\alpha_0, \beta_0).
\end{equation*}
So it suffices to prove
\begin{equation*}
\liminf_{(\alpha,\beta)\to (\alpha_0,\beta_0)}S(\alpha,\beta)\geq S(\alpha_0, \beta_0).
\end{equation*}
Assume, by contradiction, that there exists a sequence $(\alpha_n, \beta_n)\to (\alpha_0, \beta_0)$ such that \begin{equation*}
\lim_{n\to \infty}S(\alpha_n, \beta_n)<S(\alpha_0,\beta_0),
\end{equation*}
then there exists $\epsilon>0$ and $(v_n)\in \tilde{D}^\gamma(\mathcal C)$ such that
\begin{equation*}
\int_{\mathcal C}|v_n|^{p_n}\,d\mu=1
\end{equation*}
and
\begin{equation*}
S(\alpha_0,\beta_0)-\epsilon \geq F_{\alpha_n, \beta_n}(v_n).
\end{equation*}
Since $v_n$ is bounded in $\tilde{D}^\gamma(\mathcal C)$, from Lemma \ref{lemma2} { and \eqref{functional-F}}, we get
\begin{equation*}
F_{\alpha_n, \beta_n}(v_n)+o(1)\geq F_{\alpha_0,\beta_0}(v_n)\geq S(\alpha_0, \beta_0),
\end{equation*}
which yields a contradiction.

\medskip

\begin{remark}\label{remark:Hardy}
 If $\beta=\alpha+\gamma$, then $p=2$ and we are in the linear setting. Clearly by the definition of $F_{\alpha, \alpha+\gamma}(v)$ given in \eqref{eq-F}, we have
\begin{equation*}
F_{\alpha, \alpha+\gamma}(v)\geq 2\kappa_{\alpha, \gamma}^n
\end{equation*}
for all $v\in \tilde{D}^\gamma(\mathcal C)$. Choose $v_R$ to be a cutoff function such that $v_R(t)=1$ for $|t|\leq R$ and $v_R(t)=0$ in $\r\setminus (-2R, 2R)$. One can check that $v_R\in \tilde{D}^\gamma(\mathcal C)$ for each $R$ and $F_{\alpha, \alpha+\gamma}(v_R)\to 2\kappa_{\alpha, \gamma}^n$ as $R\to \infty$. Therefore,
\begin{equation*}
S(\alpha, \alpha+\gamma)=2\kappa_{\alpha, \gamma}^n.
\end{equation*}
In order to prove the nonexistence of extremal solutions note that,  for $p=2$, the Euler-Lagrange equation  \eqref{EL_cylindrical}  reduces to
\begin{equation*}\label{eq-v_prov}
\int_{\mathcal C}K(t,\tilde t,\theta,\tilde\theta)(v(t,\theta)-v(\tilde t,\tilde\theta))\,d\tilde\mu=0
\end{equation*}
which is, using the definition of the conformal fractional Laplacian on the cylinder,
\begin{equation}\label{eq-v}
P_\gamma v=c_{n, \gamma}v.
\end{equation}
According to the arguments in Section 6 of \cite{acdfgw} (see also  Corollary 4.2 in \cite{acdfgw:survey}), the only solution to \eqref{eq-v} in $\tilde{D}^{\gamma}(\mathcal C)$ is zero. Therefore, the infimum of $S_{\alpha, \alpha+\gamma}$ is not achieved.\\
\end{remark}

\noindent{\bf Proof of Theorem \ref{thm1} \emph{ii}. }  For  $\alpha=\beta$ and $0< \alpha<\frac{n-2\gamma}{2}$, from Proposition \ref{prop:symmetrization}, we know that if the extremal exists, it must be radially symmetric.  In addition, existence of a radially symmetric minimizer is proved in Proposition \ref{prop:radial-minimizer}.\\

\medskip

\noindent{\bf Proof of Theorem \ref{thm1} \emph{iii}.} The approach follows that of \cite{Catrina_Wang2}, only more technical. When $\alpha=\beta=0$, \eqref{ineq_u} reduces to the well known fractional Sobolev inequality in $\mathbb R^n$, which corresponds to the  fractional Yamabe problem, i.e.,
\begin{equation*}
(-\Delta)^{\gamma}=u^{\frac{n+2\gamma}{n-2\gamma}}\text{ in }\mathbb R^n.
\end{equation*}
In this case, it is well-known that the extremal is attained by the so-called ``bubble"
functions
\begin{equation*}
U_{\lambda,x_0}(x)=\frac{\lambda^{\frac{n-2\gamma}{2}}}{
(\lambda^2+|x-x_0|^2)^{\frac{n-2\gamma}{2}}}, \ \lambda>0, \ x_0\in \r^n.
\end{equation*}
We can use this $U_{\lambda,x_0}$ as a test function in our problem. Indeed:

\begin{lemma}
Given
\begin{equation}\label{alpha1}
\alpha\in\Big(\frac{-n+2\gamma-2}{2}, \frac{n-2\gamma}{2}\Big),
\end{equation}
we have that $U_{\lambda, x_0}\in {D}^\gamma_{\alpha}(\r^n)$.
Moreover, for $x_0\neq 0$,
\begin{equation*}
S(0,0)=\lim_{\lambda\to 0}E_{\alpha, \alpha}(U_{\lambda, x_0}).
\end{equation*}
\end{lemma}

\begin{proof}
This is a direct calculation. By the definition of $E_{\alpha,\alpha}(u)$, one has
\begin{equation*}
E_{\alpha,\alpha}(U_{\lambda, x_0})=\frac{\|U_{\lambda,x_0}\|_{\gamma,\alpha}^2}{\Big(\int_{\r^n} |x|^{-\alpha p}|U_{\lambda,x_0}|^p dx\Big)^\frac{2}{p}}
\end{equation*}
where $p=2^*_\gamma$ in this case. First note that
\begin{equation*}
\begin{split}
\|U_{\lambda,x_0}\|_{\gamma,\alpha}^2&=\int_{\mathbb R^n}\int_{\mathbb R^n}
\frac{(U_{\lambda,x_0}(x)-U_{\lambda,x_0}(y))^2}{|x-y|^{n+2\gamma}
|x|^\alpha|y|^\alpha}\,dxdy\\
&(x_1=x-x_0, \ y_1=y-x_0)\\
&=\int_{\mathbb R^n}\int_{\mathbb R^n}
\frac{\lambda^{n-2\gamma}\Big(\frac{1}{(\lambda^2+|x_1|^2)^{\frac{n-2\gamma}{2}}}-\frac{1}{(\lambda^2+|y_1|^2
)^{\frac{n-2\gamma}{2}}}\Big)^2}{|x_1-y_1|^{n+2\gamma}
|x_0+x_1|^\alpha|x_0+y_1|^\alpha}\,dx_1dy_1\\
&(x_1=\lambda x_2, \ y_1=\lambda y_2)\\
&=\int_{\mathbb R^n}\int_{\mathbb R^n}
\frac{\Big(\frac{1}{(1+|x_2|^2)^{\frac{n-2\gamma}{2}}}-\frac{1}{(1+|y_2|^2
)^{\frac{n-2\gamma}{2}}}\Big)^2}{|x_2-y_2|^{n+2\gamma}
|x_0+\lambda x_2|^\alpha|x_0+\lambda y_2|^\alpha}\,dx_2dy_2.
\end{split}
\end{equation*}
Similarly, one can derive the following:
\begin{equation*}
\begin{split}
\int_{\mathbb R^n} |x|^{-\alpha p}|U_{\lambda,x_0}|^p \, dx&=\int_{\mathbb R^n}\Big(\frac{\lambda^{\frac{n-2\gamma}{2}}}{|x|^\alpha(\lambda^2+|x-x_0|^2)^{\frac{n-2\gamma}{2}}} \Big)^p\,dx\\
&(x_1=x-x_0)\\
&=\int_{\mathbb R^n}\Big(\frac{\lambda^{\frac{n-2\gamma}{2}}}{|x_0+x_1|^\alpha(\lambda^2+|x_1|^2)^{\frac{n-2\gamma}{2}}} \Big)^p\,dx_1\\
&(x_1=\lambda x_2)\\
&{=\int_{\mathbb R^n}\Big(\frac{\lambda^{\frac{n-2\gamma}{2}}}{\lambda^{n-2\gamma}|x_0+\lambda x_2|^\alpha(1+|x_2|^2)^{\frac{n-2\gamma}{2}}} \Big)^p\lambda^n\,dx_2
}\\
&=\int_{\mathbb R^n}\Big(\frac{1}{|x_0+\lambda x_2|^\alpha(1+|x_2|^2)^{\frac{n-2\gamma}{2}}} \Big)^p\,dx_2.
\end{split}
\end{equation*}
Combining the above estimates, one has
\begin{equation*}
\begin{split}
\lim_{\lambda\to 0}E_{\alpha,\alpha}(U_{\lambda, x_0})&=\lim_{\lambda\to 0}\frac{\int_{\mathbb R^n}\int_{\mathbb R^n}
\frac{\Big(\frac{1}{(1+|x_2|^2)^{\frac{n-2\gamma}{2}}}-\frac{1}{(1+|y_2|^2
)^{\frac{n-2\gamma}{2}}}\Big)^2}{|x_2-y_2|^{n+2\gamma}
|x_0+\lambda x_2|^\alpha|x_0+\lambda y_2|^\alpha}\,dx_2dy_2}{\Big(\int_{\mathbb R^n}\Big(\frac{1}{|x_0+\lambda x_2|^\alpha(1+|x_2|^2)^{\frac{n-2\gamma}{2}}} \Big)^p\,dx_2\Big)^{\frac{2}{p}}}\\
&=\frac{\int_{\mathbb R^n}\int_{\mathbb R^n}\frac{(U_{1,0}(x_2)-U_{1,0}(y_2))^2}{|x_2-y_2|^{n+2\gamma}|x_0|^{2\alpha}}\,
dx_2dy_2}{\Big(\int_{\mathbb R^n}|x_0|^{-\alpha p}U_{1,0}^p(x_2)\,dx_2\Big)^{\frac{n-2\gamma}{n}}}\\
&=\frac{\|U_{1,0}\|_{\gamma,0}^2}{\|U_{1,0}\|_{L^{\frac{2n}{n-2\gamma}}}^2}=S(0,0).
\end{split}
\end{equation*}

\end{proof}

We conclude then that for $\alpha$ in the range \eqref{alpha1}
\begin{equation*}
S(\alpha, \alpha)\leq S(0, 0).
\end{equation*}
On the other hand, if
\begin{equation}\label{alpha2}
\frac{-n+2\gamma-2}{2}<\alpha<0,
\end{equation}
 then for any $v\in \tilde{D}^\gamma(\mathcal C)\setminus\{0\}$, one has
$F_{\alpha, \alpha}(v)>F_{0,0}(v)\geq S(0,0)$ where we use the fact (see Corollary \ref{cor_nu}) that $\kappa_{\alpha, \gamma}^n $ is decreasing in $\alpha$ if $\alpha<0$. From this, one has that
\begin{equation*}
S(\alpha, \alpha)\geq S(0, 0)
\end{equation*}
for $\alpha$ as in \eqref{alpha2} and hence $S(\alpha, \alpha)=S(0, 0)$ for such values of $\alpha$.\\

Next we fix $\alpha_1 \in (\frac{-n+2\gamma-2}{2}, 0)$. For any $\alpha\in (-2\gamma, \frac{-n+2\gamma-2}{2}] $, in case it is non empty, and $\epsilon>0$, there exists $v\in \tilde{D}^\gamma(\mathcal C)$ such that
\begin{equation*}
F_{\alpha_1, \alpha_1}(v)\leq S(0,0)+\frac{\epsilon(\kappa_{\alpha, \gamma}^n-\kappa_{0,\gamma}^n)}{{2}(\kappa_{\alpha, \gamma}^n-\kappa_{\alpha_1,\gamma}^n)}.
\end{equation*}
Since $S(0,0)\leq F_{0,0}(v)\leq F_{\alpha_1,\alpha_1}(v)$, we have
\begin{equation*}
\frac{\int_{\mathcal C}v^2\,d\mu}{
(\int_{\mathcal C}v^p\,d\mu)^{2/p}}\leq \frac{\epsilon}{{2}(\kappa_{\alpha, \gamma}^n-\kappa_{\alpha_1,\gamma}^n)}.
\end{equation*}
Then
\begin{equation*}
\begin{split}
F_{\alpha, \alpha}(v)&=F_{\alpha_1,\alpha_1}(v)+{2}(\kappa_{\alpha,\gamma}^n
-\kappa_{\alpha_1,\gamma}^n)\frac{\int_{\mathcal C}v^2\,d\mu}{
(\int_{\mathcal C}v^{p}\,d\mu)^{2/p}}\\
&\leq S(0,0)+\epsilon+\frac{\epsilon(\kappa_{\alpha, \gamma}^n-\kappa_{0,\gamma}^n)}{{2}(\kappa_{\alpha, \gamma}^n-\kappa_{\alpha_1,\gamma}^n)}.
\end{split}
\end{equation*}
As $\epsilon\to 0$, one has $S(\alpha, \alpha)\leq S(0, 0)$ and thus,  $S(\alpha,\alpha)=S(0, 0)$ for $\alpha\in (-2\gamma, 0)$.\\

Finally we will show $S(\alpha, \alpha)$ is not achieved for $\alpha<0$. If this is not true, then there exists $v\in \tilde{D}^\gamma(\mathcal C)$, such that $S(\alpha, \alpha)=F_{\alpha, \alpha}(v)$. But using $F_{\alpha, \alpha}(v)>F_{0,0}(v)\geq S(0,0)$, we arrive to a contradiction.

\medskip

\noindent{\bf Proof of Theorem \ref{thm1} \emph{iv}.} Here, since the case $\alpha\geq 0$ has been considered in Proposition \ref{prop:symmetrization}, we take $-2\gamma<\alpha<0$  and $\alpha<\beta<\alpha+\gamma$. The proof follows the ideas of Theorem 1.2 in \cite{Catrina_Wang2} for the fractional energy. We need the following lemma:

\begin{lemma}\label{lemma3}
Let $r>0$ and $2\leq q<2_\gamma^*$. Let $(\omega_j)$ be a bounded sequence in $\tilde{D}^\gamma(\mathcal C)$. If
\begin{equation*}
\sup_{y\in \mathcal C}\int_{B_r(y)\cap \mathcal C}|\omega_j|^q\,d\mu \to 0 \quad \text{as}\quad j \to \infty,
\end{equation*}
where $B_r(y)$ denotes the ball in $\r^{n+1}$ with radius $r$ centered at $y$,
then $\omega_j\to 0$ in $L^p(\mathcal C)$ for $2<p<2_\gamma^*$.
\end{lemma}

\begin{proof}
In the local case, this is written in Lemma 1.21 of \cite{Willem}, although originally it is due to Lions \cite{Lions1,Lions2}. The non-local case only requires minor variations and the (continuous) Sobolev embedding
$\tilde D^\gamma(\mathcal C)\hookrightarrow L^{2^*_\gamma}(\mathcal C)$
from Proposition \ref{embedding}.

\end{proof}

Now let $-2\gamma<\alpha<0$ and $\alpha<\beta<\alpha+\gamma$ be fixed. Consider a minimizing sequence $(v_j)\in \tilde{D}^\gamma(\mathcal C)$, such that
\begin{equation*}
\int_{\mathcal C}v_j^p\,d\mu=1,
\end{equation*}
where we recall that $2<p<2_\gamma^*$, and
\begin{equation*}
\int_{\mathcal C}\int_{\mathcal C}
K(t,\tilde t,\theta,\tilde\theta)(v_j(t,\theta)-v_j(\tilde t,\tilde\theta))^2\,d\mu d\tilde\mu+2\kappa_{\alpha,\gamma}^n\int_{\mathcal C}v_j^2\,d\mu\to S(\alpha,\beta) \quad\text{as }j\to \infty.
\end{equation*}
According to Lemma \ref{lemma3}, there exists $r>0$ such that
\begin{equation*}
\delta:=\liminf_{j\to \infty}\sup_{y\in \mathcal C}\int_{B_r(y)\cap \mathcal C}|v_j|^q\,d\mu>0,
\end{equation*}
otherwise $v_j \to 0$ in $L^p(\mathcal C)$ which is a contradiction. Then there exists a sequence $(y_j)\in \mathcal C$ and $y_0\in\mathcal C$ such that $y_j \to y_0$,  and  $\bar v_j(x):=v_j(x-y_j+y_0)$, where, for $x=(t, \theta), y=(\bar{t}, \bar{\theta})$, the ``translation" in cylinder is understood to be $(x-y)=(t-\bar{t}, \theta\bar{\theta})$, i.e. translation in $t$ and rotation in $\theta$. Then $\bar{v}_j$ has the following property:
\begin{equation}\label{eq-delta}
\int_{B_r(y_0)\cap \mathcal C}|\bar v_j|^2\,d\mu\geq \frac{\delta}{2}.
\end{equation}
We have that
\begin{equation*}
\int_{\mathcal C}\bar v_j^p\,d\mu=1
\end{equation*}
and
\begin{equation*}
\int_{\mathcal C}\int_{\mathcal C}K(t,\tilde t,\theta,\tilde\theta)(\bar v_j(t,\theta)-\bar v_j(\tilde t,\tilde\theta))^2\,d\mu d\tilde\mu+2\kappa_{\alpha,\gamma}^n\int_{\mathcal C}\bar v_j^2\,d\mu\to S(\alpha,\beta).
\end{equation*}
Here we have used that our energy is translation invariant in the sense explained above. A way to see this is by looking at the kernel \eqref{kernel}  both in the $\theta$ and $t$ coordinates.

By the compact embedding from Proposition \ref{embedding}, we may assume that
\begin{equation*}
\begin{split}
&\bar v_j \to \bar v \mbox{ weakly \ in } \tilde{D}^\gamma(\mathcal C),\\
&\bar v_j \to \bar v \mbox{ in }L^2_{loc}(\mathcal C),\\
&\bar v_j\to \bar v\mbox{ a.e. \ in }\mathcal C.
\end{split}
\end{equation*}
Moreover, we have
\begin{equation*}
1=\|\bar v\|_{L^p}^p+\lim_{j\to \infty}\|\bar v_j-\bar v\|_{L^p}^p.
\end{equation*}
Hence
\begin{equation*}
\begin{split}
S(\alpha, \beta)&=\lim_{j\to \infty}\int_{\mathcal C}\int_{\mathcal C}K(t,\tilde t,\theta,\tilde\theta)(\bar v_j(t,\theta)-\bar v_j(\tilde t,\tilde\theta))^2\,d\mu \tilde d\mu+2\kappa_{\alpha,\gamma}^n\int_{\mathcal C}\bar v_j^2\,d\mu\\
&=\int_{\mathcal C}\int_{\mathcal C}K(t,\tilde t,\theta,\tilde\theta)(\bar v(t,\theta)-\bar v(\tilde t,\tilde\theta))^2
\,d\mu d\tilde \mu+2\kappa_{\alpha,\gamma}^n\int_{\mathcal C}\bar v^2\,d\mu\\
&+\lim_{j\to \infty}2\kappa_{\alpha,\gamma}^n\int_{\mathcal C}(\bar v_j^2-\bar v^2)\,d\mu\\
&+\lim_{j\to \infty}\int_{\mathcal C}\int_{\mathcal C}K(t,\tilde t,\theta,\tilde\theta)
\Big((\bar v_j-\bar v)(t,\theta)-(\bar v_j-\bar v)(\tilde t,\tilde\theta)\Big)^2
\,d\mu d\tilde\mu\\
&\geq S(\alpha, \beta)\Big(\|\bar v\|_{L^p}^2+\|\bar v_j-\bar v\|_{L^p}^2 \Big)\\
&=S(\alpha,\beta)\Big(\|\bar v\|_{L^p}^2+(1-\|\bar v\|_{L^p}^p)^{\frac{2}{p}}\Big).
\end{split}
\end{equation*}
Since $\bar v\neq 0$ by \eqref{eq-delta}, we derive that $\|\bar v\|_{L^p}=1$ and
\begin{equation*}
\int_{\mathcal C}\int_{\mathcal C}K(t,\tilde t,\theta,\tilde\theta)
(\bar v(t,\theta)-\bar v(\tilde t,\tilde\theta))^2
\,d\mu d\tilde\mu+2\kappa_{\alpha,\gamma}^n\int_{\mathcal C}\bar v^2\,d\mu= S(\alpha,\beta).
\end{equation*}
So $\bar v$ is the desired extremal solution. This completes the proof of the Theorem.

\qed

Finally, let us comment on positivity of a miminizer:

\begin{proposition}\label{prop:positivity}
In the setting of Theorem \ref{thm1}, if an extremal solution for \eqref{eq-S} exists, then it is strictly positive.
\end{proposition}

\begin{proof}
It is clear to see that, for the absolute value of $u$, it holds $\| \,|u|\, \|_{\gamma,\alpha}\leq \| u \|_{\gamma,\alpha}$. Now we show that an extremal solution must be strictly positive.  Changing to cylindrical coordinates as in \eqref{u-v}, we have a non-negative solution $v$ to
\begin{equation*}
P_\gamma v +C(\alpha)v=v^{p-1},\quad t\in\mathbb R, \,\theta\in\mathbb S^{n-1}.
\end{equation*}
Now, let us show that  $v$  is strictly positive. Indeed,
if there exists $(t_0,\theta_0)$ such that $v(t_0,\theta_0)=0$, from the integral expression for the operator $P_\gamma$ from \eqref{P-cylinder} we have
\begin{equation*}
\begin{split}
P_{\gamma}v(t_0,\theta_0)&=\varsigma_{n,\gamma}
 \int_{\mathcal C} K(t_0,\tilde t,\theta_0,\tilde \theta)
 (v(t_0,\theta_0)-v(\tilde t,\tilde \theta))\,d\tilde\mu + c_{n,\gamma} v(t_0,\theta_0)\\
& =v^{p-1}(t_0,\theta_0)=0.
\end{split}
\end{equation*}
The first line should be strictly negative unless $v$ is identically zero, which is a contradiction. We conclude that $v$ must be strictly positive everywhere.
\end{proof}

We notice here that in the paper \cite{Abdellaoui-Medina-Peral-Primo} they showed a weak Harnack inequality for the operator $\mathcal L_{\gamma,\alpha}(u)$ in the case $\alpha>0$. Although it should be possible to extend this result to negative values of $\alpha$, some numerology needs to be done.

\section{Proof of Theorem \ref{thm2}}

In this section, we will use the moving plane method to prove the modified inversion symmetry of the extremal solutions. As in the local case, there are two key points in applying the moving plane method: the maximum principle and Hopf's lemma for the conformal fractional Laplacian operator $P_\gamma$.

In the notation of Theorem \ref{thm2}, let $v$ be a positive solution of
\begin{equation*}
\tilde{\mathcal L}_{\gamma,\alpha} v(t,\theta)=\varsigma_{n,\gamma}\kappa_{\alpha,\gamma}^nv(t,\theta)^{p-1}
\end{equation*}
on the cylinder $\mathcal C$. Up to translation in the variable $t$, we will show that this solution is even in $t$ variable and decreasing for $t>0$.

For $\lambda<0$ and $z=(t, \theta)\in \mathcal C$, denote $z^\lambda=(2\lambda-t, \theta)\in\mathcal C$ the reflection of $z$ relative to the plane $t=\lambda$. We let
\begin{equation*}
w_\lambda(z):=v(z^\lambda)-v(z),
\end{equation*}
which is defined on $\Sigma_\lambda{:=}\{(t, \theta)\in \mathcal C, t<\lambda\}$. So $w_\lambda(z)=0$ for $z\in { T_\lambda:=\partial \Sigma_\lambda}=\{(t, \theta)\in \mathcal C,\, t=\lambda\}$. We also consider the odd extension of $w_\lambda$ in $\mathcal C$ with respect to $\Sigma_\lambda$.
By definition of $w_\lambda$, it will satisfy  the equation
\begin{equation}\label{contradiction1}
P_\gamma w_\lambda+c(z)w_\lambda=0, \quad w_\lambda(z)=-w_\lambda(z^\lambda)\quad \mbox{in }\Sigma_\lambda,
\end{equation}
where
\begin{equation}\label{equation110}
c(z):=\varsigma_{n,\gamma}\kappa_{\alpha,\gamma}^n-c_{n,\gamma}-\varsigma_{n,\gamma}\kappa_{\alpha,\gamma}^n
(p-1)\int_0^1[v(z)+s(v(z^\lambda)-v(z))]^{p-2}\,ds.
\end{equation}
Note that
%$w_\lambda(z)=0$ for $z\in T_\lambda$ and
$w_{\lambda}(z)\to 0$ as $t\to \pm\infty$, since $v\in\tilde{D}^{\gamma}(\mathcal C)$.\\

\begin{lemma}\label{lemma4}
If $w_\lambda$ has a negative minimum at $(t_0, \theta_0)$ for $t_0<\lambda$, then there exists $R>0$ such that $|t_0|<R$.
\end{lemma}

\begin{proof}
Since $v(t, \theta)\to 0$ as $t\to \pm\infty$, we can choose $R$ large enough such that
\begin{equation}\label{vbound}
|v(z)|<T=\Big(\frac{2}{(p-1)}\Big)^{\frac{1}{p-2}} \quad\mbox{for }|t|\geq R.
\end{equation}
If $z_0=(t_0, \theta_0)\in \Sigma_\lambda$ is a minimum point such that $w_\lambda(z_0)<0$ and $|t_0|>R$, then one has
\begin{equation*}
v(z_0^\lambda)<v(z_0)<T.
\end{equation*}
Therefore $c(z_0)>-c_{\gamma,n}$, looking at \eqref{equation110} and \eqref{vbound}. Since $z_0$ is a minimum, we know that $w_\lambda(z_0)<0$ and thus,
% \textcolor{magenta}{for $w_\lambda(x)=-w_\lambda(x^\lambda)$, one has
\begin{equation*}
\begin{split}
P_\gamma &w_\lambda(z_0)+c(z)w_\lambda(z_0)\\
&=\varsigma_{n,\gamma}\int_{\mathcal C}K(t_0,\tilde t,\theta_0,\tilde \theta)(w_\lambda(t_0,\theta_0)-w_\lambda(\tilde t, \tilde \theta))\,d\tilde \mu+c_{n,\gamma}w_\lambda(z_0)+c(z_0)w_\lambda(z_0)\\
&=\varsigma_{n,\gamma}\int_{\Sigma_\lambda}\Big\{[{K}(t_0,\tilde t, \theta_0,\tilde\theta)-{K}(-t_0,\tilde t, \theta_0,\tilde\theta)][w_\lambda(t_0, \theta_0)-w_\lambda(\tilde t, \tilde\theta)] \\
&\qquad\qquad\qquad+2{K}(-t_0,\tilde t, \theta_0,\tilde\theta)w_\lambda(t_0,\theta_0)        \Big\} \,d\tilde \mu
+c_{n,\gamma}w_\lambda(z_0)+c(z_0)w_\lambda(z_0)\\
&=:\int_{\Sigma_\lambda}\mathcal F(t_0,\theta_0, \tilde t,\tilde\theta)\,d\tilde \mu+2\varsigma_{n,\gamma}w_\lambda(t_0, \theta_0)\int_{\Sigma_\lambda}{K}(-t_0,\tilde t, \theta_0,\tilde\theta)\, d\tilde \mu+c_{\gamma,n}w_\lambda(z_0)+c(z_0)w_\lambda(z_0)\\
&\leq c_{n,\gamma}w_\lambda(z_0)+c(z_0)w_\lambda(z_0)<0
\end{split}
\end{equation*}
where, in the inequality in the last line, to handle the term with $\mathcal F$  we have used that $K(t,-\tilde{t}, \theta, \tilde{\theta})$ is even and decreasing in $|t-\bar{t}|$.
This yields a contradiction with \eqref{contradiction1}.\\
\end{proof}

\begin{lemma}[Maximum Principle]\label{lemma:max-principle}
Let $v$ a solution to
$$P_\gamma v=f(v) \text{ in } \Sigma_\lambda,$$
 satisfying $v\geq 0$ in $\Sigma_\lambda$, $f(v)\geq 0$ for $v\geq 0$ and $v$ is anti-symmetric with respect to $\partial \Sigma_\lambda$, i.e. $v(z^\lambda)=-v(z)$. Then $v\equiv 0$ or $v>0$ in $\Sigma_\lambda$.
\end{lemma}
\begin{proof}
Let us assume that there exists $(t_0, \theta_0)\in \Sigma_\lambda$ with $v(t_0, \theta_0)=0$. Then, as above,
\begin{equation*}
\begin{split}
P_\gamma v(t_0, \theta_0)&=\varsigma_{n,\gamma}\int_{\C}
K(t_0,\tilde t,\theta_0,\theta)(v(t_0, \theta_0)-v(\tilde t,\tilde \theta))
\,d\tilde \mu
+c_{n,\gamma}v(t_0,\theta_0)\\
&=-\varsigma_{n,\gamma}\int_{\Sigma_\lambda}
[K(t_0,\tilde t,\theta_0,\tilde \theta)-K(t_0,-\tilde t,\theta_0,\tilde \theta)]v(\tilde t,\tilde \theta)
\,d\tilde \mu
\\&\leq 0.
\end{split}
\end{equation*}
Thus $P_\gamma v=f(v)$ is satisfied if and only if $v\equiv 0$.

\end{proof}

Hopf's Lemma for anti-symmetric functions has been studied by \cite{Chen-Li} for the fractional Laplacian operator in Euclidean space (see also \cite{DelaTorre-Hyder-Martinazzi-Sire} for the one-dimensional case under weaker assumptions). In the following, we will adapt their ideas and derive the analogous Hopf lemma for $P_\gamma$.  Here, for simplicity of notation, we denote $\Sigma=\Sigma_\lambda$ and $w$ the odd extension of $w_{\lambda}$.

\begin{lemma}[Hopf Lemma for anti-symmetric functions]
Assume that $w\in C^3_{loc}(\Sigma)$,
\begin{equation*}
\limsup_{z\to \partial \Sigma} c(z)=o\Big(\frac{1}{[\distance(z, \partial \Sigma)]^2}\Big),
\end{equation*}
and
\begin{equation*}
\left\{\begin{array}{l}
P_\gamma w+c(z)w=0\quad\mbox{ in }\Sigma, \\
w(z)>0 \quad\mbox{in }\Sigma,\\
w(z^\lambda)=-w(z) \quad\mbox{in }\Sigma.
\end{array}
\right.
\end{equation*}
Then
\begin{equation*}
\frac{\partial w}{\partial \vec{n}}<0, \quad\text{for every } z\in \partial \Sigma,
\end{equation*}
where $\vec{n}$ is the outer unit normal vector of $\partial \Sigma$.
\end{lemma}
%\textcolor{green}{maybe we should use another letter for the normal $\nu$ since it is used earlier, but I couldn't think of any}

\begin{proof}
Without loss of generality, we may assume that $\lambda=0$. It suffices to show that $\frac{\partial w}{\partial t}(0)<0$. We argue by contradiction; thus suppose that $\frac{\partial w}{\partial t}(0)=0$, then it follows from the anti-symmetry of the function that $\frac{\partial^2 w}{\partial t^2}(0)=0$ and $w(t, \theta)=O(|t|^3)$ for $t$ close to $0$. We will derive a contradiction to the equation.

Recall that, by definition, for $z_0=(0,\theta_0)$,
$$P_\gamma w(z_0)=\int_{\C}
K(0,\tilde t,\theta_0,\tilde \theta)(w(0, \theta_0)-v(\tilde t,\tilde \theta))
\,d\tilde \mu
+c_{n,\gamma}w(0,\theta_0).$$
Using the oddness of $w$ with respect to $\partial \Sigma$,
\begin{equation*}
\begin{split}
&\int_{\C}K(t,\tilde{t},\theta,\tilde{\theta})(w(t, \theta)-w(\tilde{t},\tilde{\theta}))
\,d\tilde \mu\\
&=\int_{\Sigma}\Big\{[{K}(t,\tilde t, \theta,\tilde\theta)-{K}(-t,\tilde t, \theta,\tilde\theta)][w(t, \theta)-w(\tilde t, \tilde\theta)] +2{K}(-t,\tilde t, \theta,\tilde\theta)w(t,\theta)        \Big\} \,d\tilde \mu\\
&=:\int_{\Sigma}\mathcal F(t,\theta, \tilde t,\tilde\theta)\,d\tilde \mu+2w(t, \theta)\int_{\Sigma}{K}(-t,\tilde t, \theta,\tilde\theta) \,d\tilde \mu\\
&=:I_1+I_2.
\end{split}
\end{equation*}
We divide $\Sigma$ into several subregions and estimate the above integrals in each region. Let $\delta=|t|=\distance(z, \partial \Sigma)$, $z=(t, \theta)\in\Sigma:=\{t<0\}$. Define
\begin{equation*}
\begin{split}
A_1&=\{\tilde z, -2\delta\leq \tilde t\leq 0, |\tilde \theta|\leq \epsilon\}, \\
A_2&=\{\tilde z, -\epsilon\leq \tilde t\leq -2\delta, |\tilde \theta|\leq \epsilon\},\\
A_3&=\{-R<\tilde t<-\eta\}\setminus (A_1\cup A_2),\\
A_5&=\Sigma \setminus (A_1\cup A_2\cup A_3),
\end{split}
\end{equation*}
and the auxiliary
$$A_4=\{\tilde z, -2\leq \tilde t\leq -1\},$$
where $\delta, \epsilon, \eta, R$ are to be determined later.
We take  $\epsilon $  sufficiently small and $\delta $ with $\delta <<\epsilon$. In the following, we will always assume that $t$ is close to $0$.

We estimate $I_1$ first. Since ${K}(t,\tilde t, \theta, \tilde\theta)$ is even in $t-\tilde t$ and decreasing with respect to $|t-\tilde t|$, we have
\begin{equation*}
{K}(t,\tilde t, \theta,\tilde \theta)-{K}(-t,\tilde t, \theta,\tilde \theta)>0
\end{equation*}
and, since $w(z)=O(|t|^3)$ for $t$ small and $w(z)>0$ in $\Sigma$, in $A_4$, there exists some constant $c>0$ independent of small $\epsilon, \ \delta$ such that
$${K}(t,\tilde t, \theta,\tilde \theta)-{K}(-t,\tilde t, \theta,\tilde \theta)>c>0,$$
and moreover,
 $$
 w(t, \theta)-w(\tilde{t},\tilde{\theta})\leq O(|t|^3)-c<-\frac{1}{2}c<0,
 $$
 so one has
\begin{equation*}
\int_{A_4} \mathcal F(t, \theta, \tilde t,\tilde \theta)\,d\tilde\mu\leq  -c_1\delta
\end{equation*}
for some $c_1>0$ independent of $\epsilon$ small.
In addition, using the asymptotic behaviour of ${K}$,
\begin{equation*}
\Big|\int_{A_1}\mathcal  F(t,\tilde t,\theta,\tilde \theta)\,d\tilde\mu\Big|\leq c\max\{\delta^{2-2\gamma}, \epsilon^{2-2\gamma}\}\delta
\end{equation*}
and
\begin{equation*}
\int_{A_2}\mathcal F(t, \theta,\tilde t, \tilde\theta)\,d\tilde\mu\leq c\epsilon^{2-2\gamma}\delta.
\end{equation*}
By the above two estimates, we can choose $\delta$ and $\epsilon$ small enough such that
\begin{equation*}
\int_{A_1\cup A_2}\mathcal F(t, \theta, \tilde t, \tilde\theta)\,d\tilde\mu \leq \frac{c_1}{4}\delta.
\end{equation*}
Fixed $\epsilon$ and $\delta$,  using the fact that ${K}(t,\tilde t, \theta, \tilde\theta)$ is exponentially decaying in $t-\tilde t$ for $|t-\tilde{t}|$ large, one can choose $R$ large enough  and $\eta$ small enough such that, for $t\to 0$,
\begin{equation*}
\int_{A_5}\mathcal F(t, \theta, \tilde t, \tilde\theta)\,d\tilde\mu \leq \int_{\mathbb S^{n-1}}\int_{\{\tilde{t}<-R\}}\mathcal F(t, \theta, \tilde t, \tilde\theta)\,d\tilde\mu+\int_{\mathbb S^{n-1}}\int_{\{\tilde{t}>-\eta\}}\mathcal F(t, \theta, \tilde t, \tilde\theta)\,d\tilde\mu<\frac{c_1}{4}\delta.
\end{equation*}
For this $R$ and $\eta$ it holds that, for $t$ small enough, $w(t, \theta)-w(\tilde t,\tilde\theta)\leq0$ for all $(\tilde t,\tilde\theta)\in A_3$. Then one has
\begin{equation*}
\int_{A_3}\mathcal F(t, \theta,\tilde t,\tilde\theta)\,d\tilde\mu\leq \int_{A_4}\mathcal F(t, \theta, \tilde t,\tilde\theta)\,d\tilde\mu<-c_1\delta.
\end{equation*}
Combining the above estimates we arrive to
\begin{equation*}
\begin{split}
I_1=\int_{\Sigma}\mathcal F(t, \theta,\tilde t,\tilde\theta)\,d\tilde\mu
=\left\{\int_{A_1\cup A_2}+\int_{A_3}+\int_{A_5}\right\}\mathcal F(t, \theta,\tilde t,\tilde\theta)\,d\tilde\mu<-\frac{c_1}{2}\delta.
\end{split}
\end{equation*}

On the other hand, for $I_2$, since $w(t, \theta)=O(\delta^3)$ we have
\begin{equation*}
w(t, \theta)\int_{\Sigma}{K}(-t,\tilde t, \theta,\tilde\theta) \,d\tilde \mu=O(\delta^{3-2\gamma}).
\end{equation*}
In addition,
\begin{equation*}
c(z)w(z)=o(1)\delta
\end{equation*}
and thus, combining the above,
\begin{equation*}
P_\gamma w(z_0)+c(z_0)w(z_0)<0,
\end{equation*}
which is a contradiction. Therefore, we must have $\frac{\partial w}{\partial t}(0)<0$ and this concludes the proof of the Lemma.\\
\end{proof}

\medskip

\noindent{\bf Proof of Theorem \ref{thm2}.}
Now, by Lemma \ref{lemma4}, $w_\lambda(z)\geq 0$ for $\lambda\leq -R$. Let $\lambda_0$ be the largest $\lambda$ such that the property $w_\lambda\geq 0$ in $\Sigma_\lambda$ holds. Clearly $\lambda_0$ exists since $v(t, \theta)\to 0$ as $|t|\to \infty$. % and any bounded real sequence has a convergent subsequence.
We will prove that
\begin{itemize}
\item[\emph{a.}]
$w_{\lambda}(z)>0, \ z\in \Sigma_\lambda, \lambda<\lambda_0$;
\item[\emph{b.}]
$w_{\lambda_0}(z)=0 , \ z\in \Sigma_{\lambda_0}$.
\end{itemize}
To prove the first statement, assume that there exists $\delta>0$ small such that for some $z_0=(t_0, \theta_0)$ with $t_0<\lambda_0-\delta$, it holds that $w_{\lambda_0-\delta}(z_0)=0$. Then by the maximum principle Lemma \ref{lemma:max-principle}, we have $w_{\lambda_0-\delta}\equiv 0 $ in $\Sigma_{\lambda_0-\delta}$. This implies that $v(\lambda_0-2\delta, \theta_0)=v(\lambda_0, \theta_0)$. Since $\frac{\partial v}{\partial t}(t, \theta_0)\geq 0$, %Since we know that $w_{\lambda}\geq 0$, indeed we only want to exclude the equality, we have the sign of the derivative of v when we are close enough ($\delta$ small enough).
 it follows that
\begin{equation*}
\frac{\partial v}{\partial t}(t, \theta_0)=0\mbox{ for }t\in [\lambda_0-2\delta, \lambda_0].
\end{equation*}
Therefore,
\begin{equation*}
\frac{\partial w_{\lambda_0-2\delta}}{\partial t}(\lambda_0-2\delta, \theta_0)=0.
\end{equation*}
By the Hopf's lemma above, we get that $w_{\lambda_0-2\delta}\equiv 0$. Continuing this process one can show that $v$ is independent of $t$, which is impossible. Therefore, $\frac{\partial w_\lambda}{\partial t}<0$ in $T_\lambda$ for $\lambda<\lambda_0$ by Hopf's lemma again. Then $\frac{\partial v}{\partial t}>0$ in $\Sigma_\lambda$ and this yields claim \emph{a.}

For the second result, assume that $w_{\lambda_0}\not\equiv 0$. Then by the maximum principle, $w_{\lambda_0}>0$ in $\Sigma_{\lambda_0}$ and $\frac{\partial w_{\lambda_0}}{\partial t}<0$ on $\partial\Sigma_{\lambda_0}$ by Hopf's lemma. From the definition of $\lambda_0$, there exists a sequence $\lambda_k \to \lambda_0$ with $\lambda_k>\lambda_0$, and there exist points $z_k\in \Sigma_{\lambda_k}$ such that $w_{\lambda_k}(z_k)<0$. By Lemma \ref{lemma4}, we know that the sequence $(z_k)$ is bounded, hence it converges to a point $z_0$. It follows that $z_0\in \partial\Sigma_{\lambda_0}$ and $\frac{\partial w_{\lambda_0}}{\partial t}(z_0)=0$. Contradiction.\\

Therefore, after a translation in $t$, we can assume that $\lambda_0=0$ which implies that $v$ is even in $t$ and decreasing for $t>0$. This completes the proof of Theorem \ref{thm2}.
\qed

\section{Non-degeneracy}\label{section:nondegeneracy}

Let $\bar u$ be a positive, radially symmetric, energy solution to \eqref{eq-extremal} in $D_\alpha^\gamma(\r^n) $ and set $\bar v$ as in \eqref{u-v} using cylindrical coordinates. Then, $\bar v=\bar v(t)$
is an energy solution to
\begin{equation*}
\tilde {\mathcal  L}_{\gamma,\alpha}\bar v:=P_\gamma \bar v+C(\alpha)\bar v=\varsigma_{n,\gamma} \kappa_{\alpha,\gamma}^n \bar v^{p-1} \quad\text{in }\mathcal C
\end{equation*}
that, from Theorem \ref{thm2}, is positive, even and  decreasing for $t>0$.

We say that $\bar v$ is a \emph{ground state} if, in addition,  it is a stable solution, this is,
\begin{equation*}\label{second-variation}
\left.\frac{d^2}{d\epsilon^2}\right|_{\epsilon=0}F_{\alpha,\beta}(\bar v+\epsilon w)\geq 0\quad \text{for every }w\in  C^\infty _c (\mathcal C).
\end{equation*}
In particular, a minimizer $\bar v$ for $F_{\alpha,\beta}$ is a ground state.

%{\color{red}From now and on, when the exact constants are not needed, we will just use $C$ to denote a general constant which can change from line to line.}
\subsection{Non-degeneracy in the radial sector}

Let $\bar L$ be the linearized operator around such $\bar v$, which is given by
\begin{equation}\label{linearized10}
\bar Lw=P_\gamma w+C(\alpha)w-(p-1)\varsigma_{n,\gamma} \kappa_{\alpha,\gamma}^n \bar v^{p-2}w
\end{equation}
acting on $L^2(\mathcal C)$.

Here we study non-degeneracy in the space of radially symmetric functions on $\mathcal C$. We thus project over the $m=0$ eigenspace; our aim is to show that the kernel
\begin{equation}\label{equation-kernel}
\bar L^{(0)}w:=P_\gamma^{(0)} w+C(\alpha)w-(p-1)\varsigma_{n,\gamma} \kappa_{\alpha,\gamma}^n \bar v^{p-2}w=0,\quad w=w(t),
\end{equation}
is one-dimensional. Our initial observation is that the function $\bar w:=\bar v_t{=\partial_t \bar v(t)}$ is a solution to \eqref{equation-kernel}.
We will prove that this is the only possibility up to multiplication by constant. In addition, since by Theorem \ref{thm2}, $\bar v$ is even in $t$, this will imply that the linearization $\bar L^{(0)}$ is invertible in the space of even functions on $\mathbb R$.

Our proof is mostly contained in our previous papers \cite{acdfgw,acdfgw:survey}, but we present it here in detail for completeness.

The first step is to find the indicial roots of the problem as $t\to \pm\infty$. Since $\bar v$ is decaying as $t\to \pm \infty$, we need to look first at the associated ``constant coefficient" operator $P_\gamma^{(0)}+C(\alpha)$. Taking into account Proposition \ref{prop:symbol} above on the Fourier characterization of $P_\gamma^{(0)}$, its kernel is fully described in Corollary 4.2 of \cite{acdfgw:survey}. Indeed, it is generated by exponential functions $e^{\eta t}$ as in a regular second-order ODE. The precise exponents are given by $\eta =iz_0$ where $z_0$ is any solution in the complex plane of  equation
\begin{equation}\label{complex-zeroes}
\Theta_\gamma^{(0)}(z)+C(\alpha)=0,
\end{equation}
where we have defined the meromorphic function
\begin{equation*}
\Theta^{(0)}_{\gamma}(z)=2^{2\gamma}\frac{\Gamma\big(\frac{\gamma}{2}
+\frac{n}{4}+\frac{z}{2}i\big)\Gamma\big(\frac{\gamma}{2}
+\frac{n}{4}-\frac{z}{2}i\big)}
{\Gamma\big(-\frac{\gamma}{2}
+\frac{n}{4}+\frac{z}{2}i\big)\Gamma\big(-\frac{\gamma}{2}
+\frac{n}{4}-\frac{z}{2}i\big)},\quad z\in \mathbb C.
\end{equation*}
Contrary to the second-order setting, one may have an infinite number of indicial roots, not just two:

\begin{lemma}
The zeroes of \eqref{complex-zeroes} are of the  form $\{\tau_j\pm i\sigma_j\}$, $\{-\tau_j\pm i\sigma_j\}$, for some $\tau_j,\sigma_j>0$, $j=0,1,\ldots$, satisfying in addition that $\sigma_j>0$ is a strictly increasing sequence with no accumulation points. Moreover, $\tau_j=0$ for large $j$ and the first zero lies on the imaginary axis away from the origin ($\tau_0=0$, $\sigma_0>0$).

In particular,  $\Theta_\gamma^{(0)}(\xi)+C(\alpha)$ is bounded from below for $\xi\in\mathbb R$.
\end{lemma}

\begin{proof}
For the location of the zeroes, see Theorem 4.1 in \cite{acdfgw:survey}, based on Section 6 of \cite{acdfgw} (just take into account a sign difference with the notation there).  In any case, condition \eqref{Calpha} in the Appendix implies that there are no zeroes on the real axis. The asymptotic behavior is controlled from \cite[Section 6.4]{acdfgw}.

\end{proof}

Note that a radially symmetric, radially decreasing function in $\tilde D^\gamma$ must decay at infinity. In the case $\gamma\in(\frac{1}{2},1)$, this decay is exponential. Although we will not need this fact, let us give a precise statement:

\begin{lemma}\label{lemma:decay-minimizer}
Fix $\gamma\in(\frac{1}{2},1)$. Let $v=v(t)$ be any function in $\tilde D^\gamma$. Then, for every  $0<\delta<\sigma_0$, we have
\begin{equation*}\label{asymptotics-v}
v(t)=o(e^{-\delta |t|}),\quad\text{as }t\to\pm\infty.
\end{equation*}
\end{lemma}

\begin{proof}
Take $t\to+\infty$. We estimate
\begin{equation*}
e^{\delta t}v(t)=\frac{1}{\sqrt{2\pi}}\int_{\mathbb R} e^{t\xi i+\delta t} \hat v(\xi)\,d\xi
=\frac{1}{\sqrt{2\pi}}\int_{\mathbb R-\delta i} e^{t\tilde\xi i} \hat v(\tilde\xi+\delta i)\,d\tilde\xi .
\end{equation*}
By Cauchy-Schwarz inequality,
\begin{equation*}
\begin{split}
|e^{\delta t}v(t)|&\leq C\left(\int_{\mathbb R-\delta i}  |\Theta_\gamma^{(0)}(\tilde\xi+\delta i)+C(\alpha)|\ |\hat v(\tilde\xi+\delta i)|^2\,d\tilde\xi\right)^{\frac{1}{2}}
\left(\int_{\mathbb R-\delta i}  \frac{d\tilde\xi}{|\Theta^{(0)}_\gamma(\tilde\xi+\delta i)+C(\alpha)|}\right)^{\frac{1}{2}}\\
&=C\left(\int_{\mathbb R}  |\Theta_\gamma^{(0)}(\xi)+C(\alpha)|\ |\hat v(\xi)|^2\,d\xi\right)^{\frac{1}{2}}
\left(\int_{\mathbb R}  \frac{d\xi}{|\Theta_\gamma^{(0)}(\xi)+C(\alpha)|}\right)^{\frac{1}{2}}\\
&\leq C \|v\|_{\tilde D^\gamma},
\end{split}
\end{equation*}
where we have used that
\begin{equation*}
\int_{\mathbb R}  \frac{1}{|\Theta_\gamma^{(0)}(\xi)+C(\alpha)|}\,d\xi<\infty
\end{equation*}
since $\Theta_\gamma^{(0)}(\xi)$ is a positive function and behaves like $|\xi|^{2\gamma}$ as $|\xi|\to\infty$.

Finally, for this calculation to be rigorous we need that the function $\frac{1}{\Theta^{(0)}(\xi)+C(\alpha)}$ has no poles in the region $\{z\in\mathbb C\,:\, 0\leq \im z\leq \delta\}$, which is true as long as $0<\delta<\sigma_0$.

\end{proof}

Now we come to the study of  positive radial solutions of equation \eqref{EL-v}. As one can see in the previous lemma, a-priori,  any radial function in $\tilde{D}^\gamma$ decays exponentially for $\gamma\in (\frac{1}{2},1)$. Now we show that if in addition, $v$ is also a positive radial solution to \eqref{EL-v}, then it decays exponentially $t\to\pm\infty$ for all $\gamma\in (0,1)$. The main idea in the proof is to  relate the solutions of the variable coefficient operator $\bar L^{(0)}$ to its indicial roots. More precisely:

\begin{proposition}\label{prop-asymptotics}
We have
$$\bar v(t)=(a_0+o(1))e^{-\sigma_0 t}\quad\text{as}\quad t\to +\infty,$$
for some $a_0>0$
(and similarly for $t\to -\infty$).
\end{proposition}

\begin{proof}
We follow Proposition 5.2 in \cite{acdfgw:survey}. Note that $\bar v$ is a solution to
\begin{equation*}
P_\gamma^{(0)} \bar v+C(\alpha)\bar v=\varsigma_{n,\gamma} \kappa_{\alpha,\gamma}^n \bar v^{p-1}=:h.
\end{equation*}
Since $\bar v$ is at least bounded, the asymptotic behavior of $h$  is given by $O(e^{-\delta|t|})$ for some $\delta\geq 0$. Let us concentrate in the limit $t\to +\infty$.

Using the same ideas as in Proposition  4.3  in \cite{acdfgw:survey}, we know that
there exists a non-negative integer $j$ such that either
\[
\bar v(t)=(a_j+o(1)) e^{-\sigma_j t}
	\quad \text{ as }
t\to+\infty,
\]
for some real number $a_j\neq 0$, or
\[
\bar v(t)=
	\left(
		a_j^1 \cos(\tau_j t)
		+ a_j^2 \sin(\tau_j t)
		+ o(1)
	\right)
	e^{-\sigma_j t},
\]
for some real numbers $a_j^1, a_j^2$ not vanishing simultaneously. There is a similar expansion as $t\to -\infty$.
Finally, we have that
\begin{equation*}
a_0=c \int_{\mathbb{R}}
		e^{\sigma_0 t'}h(t')\,dt'>0
\end{equation*}
since, by hypothesis, $\bar v$ is non-negative.

\end{proof}

Recall that we had defined $\bar w=\bar v_t$. The previous Proposition implies that also
$$\bar w(t)=(a+o(1))e^{-\sigma_0 t}\quad\text{as}\quad t\to +\infty$$
for some $a\neq 0$ (and similarly as $t\to -\infty)$. We will compare any other solution of \eqref{equation-kernel}
to $\bar w$.

First we show a Frobenius-type theorem for \eqref{equation-kernel}:

\begin{proposition}\label{prop:Frobenius}
Let $w$ be a solution to \eqref{equation-kernel}
satisfying $w(t)=O(e^{-\delta |t|})$ as $|t|\to\infty$ for some $\delta>0$. Then there exists a non-negative integer $j$ such that either
\[
w(t)=(a_j+o(1)) e^{-\sigma_j t}
	\quad \text{ as }
t\to+\infty,
\]
for some real number $a_j\neq 0$, or
\[
w(t)=
	\left(
		a_j^1 \cos(\tau_j t)
		+ a_j^2 \sin(\tau_j t)
		+ o(1)
	\right)
	e^{-\sigma_j t},
\]
for some real numbers $a_j^1, a_j^2$ not vanishing simultaneously. There is a similar expansion as $t\to -\infty$.

\end{proposition}

\begin{proof}
This is essentially Theorem 4.9 in \cite{acdfgw:survey} with minor modifications. One just needs to take into account that the potential is different.
We rewrite equation \eqref{equation-kernel} as
\begin{equation*}
\bar L^{(0)} w=P_\gamma w+\mathcal V(t)w=0, \quad w=w(t),
\end{equation*}
for
\begin{equation*}
\mathcal V(t):=C(\alpha)-(p-1)\varsigma_{n,\gamma} \kappa_{\alpha,\gamma}^n \bar v^{p-2}.
\end{equation*}
From Proposition \ref{prop-asymptotics} we know that
\begin{equation*}
\mathcal V(t)=C(\alpha)+O(e^{-q|t|})
\end{equation*}
for some $q>0$. We write
$$P_\gamma^{(0)}w+C(\alpha)w=(-\mathcal V(t)+C(\alpha))w=:h,$$
for $h:=(p-1)\varsigma_{n,\gamma} \kappa_{\alpha,\gamma}^n \bar v^{p-2}w=O(e^{-\delta'|t|})$, which allows us to apply Theorem 4.4 in \cite{acdfgw:survey} and this yields the proof.
\end{proof}

We remark here that if $w\in L^2(\mathcal C)$ is a solution to \eqref{equation-kernel}, then automatically $w\in \tilde D^\gamma$ (we just need to multiply the equation by $w$ and integrate). Thus, similarly to Proposition \ref{prop-asymptotics}, $w$ has exponential decay as $t\to\pm\infty$ and we can use Proposition \ref{prop:Frobenius}.\\

One of the main results in \cite{acdfgw} is the interpretation of a non-local ODE as infinite system of second order ODEs. Since this formulation is particularly simple when all the $\tau_j$ are zero, we will restrict to this case and refer to \cite{acdfgw:survey} for the full theorem in order to avoid complex exponentials and simplify the notation (and, in any case, $\tau_0=0$ always).

Let  $w$ be a solution to  \eqref{equation-kernel}, this is,
$$P_\gamma^{(0)}w+\mathcal V(t)w=0,\quad w=w(t).$$
Then, it can be written as
\begin{equation*}\label{w-sum}
w(t)=\sum_{j=0}^\infty c_j w_j(t),
\end{equation*}
where
\begin{equation*}
w_j(t):=\int_{\mathbb R}e^{-\sigma_j |t-t'|}h(t')\,dt'.
\end{equation*}
%We remark that
Moreover, one may directly check that $w_j$ is a particular solution to the second order ODE
 \begin{equation}\label{ODE-j}
w_j''(t)-\sigma_j^2w_j(t)=-2\sigma_j h(t).
\end{equation}

Now define the Wro\'{n}skian of two solutions for the ODE \eqref{ODE-j} as
\begin{equation*}
\mathcal W_j[w,\tilde w]:=w_j\tilde w_j'-w_j'\tilde w_j,
\end{equation*}
and its weighted sum in $j=0,1,\ldots$ by
\begin{equation}\label{Wronskian}
\mathcal W[w,\tilde w]:=\sum_{j=0}^{\infty}
	\dfrac{c_j}{\sigma_j}\mathcal W_j[w,\tilde w],
\end{equation}
for the constants given in Theorem 4.4 in \cite{acdfgw:survey}. Then Lemma 5.4 in \cite{acdfgw:survey} yields:

\begin{lemma}\label{lemma:Wronskian2}
Let $w,\tilde w$ be two solutions of \eqref{equation-kernel}. Then the Wro\'{n}skian quantity from \eqref{Wronskian} satisfies
\[
\mathcal{W}[w,\tilde w]'=0.
\]
\end{lemma}

\medskip

\noindent
{\bf Proof of Theorem \ref{thm:non-degeneracy}. } From Lemma \ref{lemma:Wronskian2}, using the arguments from Proposition 5.2 in \cite{acdfgw:survey}, one can prove that any other radially symmetric solution to $\bar{L}^{(0)}w=0$ that decays both at $\pm \infty$ must be a multiple of $\bar{w}=\bar{v}_t$. This completes the proof for Theorem \ref{thm:non-degeneracy}.

\qed

\subsection{Stability of the extremals }

%As in \cite{Frank-Lenzmann-Silvestre}, we project over spherical harmonics. However, contrary to their setting, our operator is not translation invariant in $x\in\mathbb R^n$.

In this subsection we restrict to the symmetry range of $\alpha$ and $\beta$ to assure the existence of a minimizer on $\mathbb R^n$ (or equivalently, on $\mathcal C$) and study its stability. More precisely, we will  relate stability of $F_{\alpha,\beta}$ to the number of negative eigenvalues (Morse index) of $\bar L$ as in the local case \cite{Felli_Schneider}.

For this, consider the functional
\begin{equation*}
	F_0(V)=\frac{\frac{1}{2}\iint \rho^{1-2\gamma} \left\{e_1(\rho)(\partial_\rho V)^2+e_2(\rho)(\partial_t V)^2+e_3(\rho)|\nabla_\theta V|^2\right\} \,d\rho \,d\mu+\frac{\kappa}{2}\int v^2\,d\mu}{\Big( \int_{\mathcal{C}} |v|^p\,d\mu\Big)^{2/p}}
\end{equation*}
among functions $V=V(t,\theta,\rho)$, where $e_i$ are given by \eqref{e_i}. Note that we have dropped some constants with respect to the energy \eqref{energy-extension}, for simplicity. Here $\kappa>0$. For the rest of the paper we will use this notation, unless the exact value of the constants is required in the proofs.

Assume that we are in the symmetry range for $\alpha$ and $\beta$ and let $\bar V$ be a non-negative, radially symmetric, critical point. Then it is a solution of the Euler-Lagrange equation
%\begin{equation}\label{firstd-1}
%\left.\frac{d}{d\epsilon}\right|_{\epsilon=0}	J(\bar V+\epsilon W)=0,
%\end{equation}
%which is
\begin{equation}\label{EL-extension-1}
\left\{\begin{split}
& \partial_{\rho}\left(e_1(\rho)\rho^{1-2\gamma}\partial_{\rho}V\right)
+e_2(\rho)\rho^{1-2\gamma}\partial_{tt}V
=0,\quad \rho\in(0,\rho_0),t\in\mathbb R,\theta\in \mathbb S^{n-1},\\
&-\lim_{\rho\to 0}\rho^{1-2\gamma}\partial_\rho V(\rho,t)+\kappa v-c|v|^{p-2}v=0,\quad \text{on }\mathcal C,
\end{split}\right.	
\end{equation}
where $c=A/B$, for
\begin{equation*}
A=	c_{\mathbb S^{n-1}}\iint \left\{e_1(\rho)\rho^{1-2\gamma}(\partial_\rho \bar V)^2+e_2(\rho)\rho^{1-2\gamma}(\partial_t \bar V)^2\right\}\,d\rho dt+\kappa \int \bar v^2\,dt,\quad
B=c_{\mathbb S^{n-1}}\int \bar v^p  \,dt,
\end{equation*}
where $c_{\mathbb S^{n-1}}$ denotes the volume of $\mathbb S^{n-1}$.
%which obviously differs from  \eqref{EL-extension} by some constants. Also, by taking a suitable multiple of $\bar V$, we can normalize $c=\kappa$; actually, we need to take $\kappa=A/B$, where we have defined
%\begin{equation*}
%A=	c_{\mathbb S^{n-1}}\int \left\{e_1(\rho)\rho^{1-2\gamma}(\partial_\rho \bar V)^2+e_2(\rho)\rho^{1-2\gamma}(\partial_t \bar V)^2\right\}\,d\rho dt+\kappa \int \bar v^2\,dt,\quad
%B=c_{\mathbb S^{n-1}}\int \bar v^p  \,dt.
%\end{equation*}

Now we define $\bar L_+ w$ to be the linearized operator of \eqref{EL-extension-1} around $\bar V$:
\begin{equation*}
\bar L_+W=-\lim_{\rho\to 0}\rho^{1-2\gamma}\partial_\rho W+\kappa w-c(p-1)\bar v^{p-2}w,
\end{equation*}
for $W$ the (unique) solution of
\begin{equation*}\label{EL-extension1}
\partial_{\rho}\left(e_1(\rho)\rho^{1-2\gamma}\partial_{\rho}W\right)
+e_2(\rho)\rho^{1-2\gamma}\partial_{tt}W+e_3(\rho)\Delta_\theta W
%+e_3(\rho)\rho^{1-2\gamma} \Delta_\theta V
=0,\quad \rho\in(0,\rho_0),t\in\mathbb R,\theta\in\mathbb S^{n-1}
\end{equation*}
with boundary data $w$.

In the following we study (linear) stability. For this, we calculate the second variation of the energy functional. The proof is a straightforward but messy computation that the reader may skip.

\begin{lemma}\label{lemma:second-variation}
\begin{equation*}
\left.\frac{d^2}{d\epsilon^2}\right|_{\epsilon=0}F_0(\bar V+\epsilon W)=B^{-\frac{2}{p}}\Big\{ \langle \bar L_+w, w\rangle +\frac{A}{B^2}(p-2)\Big(\int \bar v^{p-1}w \,dt\Big)^2\Big\}.
\end{equation*}
\end{lemma}

\begin{proof}
By direct calculation,
using the equation satisfied by $\bar V$, we can simplify the above to
\begin{equation*}
\begin{split}	
\left.\frac{d^2}{d\epsilon^2}\right|_{\epsilon=0}&	F_0(V+\epsilon W)\\
%&B^{-\frac{2}{p}}\Big(\int \rho^{1-2\gamma}\left\{e_1(\rho)|\partial_\rho W|^2+e_2(\rho)|\partial_t W|^2 +e_3(\rho)|\nabla_\theta W|^2\right\}\, d\rho d\mu+\kappa \int w^2\,d\mu\Big)\\
%&\quad -\frac{4A}{B^{\frac{2}{p}+2}}\Big( \int\bar v^{p-1}w\, d\mu\Big)^2-(p-1)\frac{A}{B^{\frac{2}{p}+1}}\int \bar v^{p-2}w^2 d\mu+(p+2)\frac{A}{B^{\frac{2}{p}+2}}\Big(\int \bar v^{p-1}w\, d\mu\Big)^2\\
&=B^{-\frac{2}{p}}\left\{\int \rho^{1-2\gamma}\left\{e_1|\partial_\rho W|^2+e_2|\partial_t W|^2+e_3(\rho)|\nabla_\theta W|^2\right\}\, d\mu d\rho+\kappa \int w^2\,d\mu\right.\\
&\quad +\left.\frac{A}{B^2}(p-2)\Big(\int \bar v^{p-1}w\, d\mu\Big)^2-(p-1)\frac{A}{B}\int \bar v^{p-2}w^2\, d\mu
\right\}\\
&=B^{-\frac{2}{p}}\Big\{-\int w\lim_{\rho\to 0}\rho^{1-2\gamma}\partial_\rho W\, d\mu+\kappa\int w^2\,d\mu   +\frac{A}{B^2}(p-2)\Big(\int \bar v^{p-1}w \,d\mu\Big)^2\\
&\quad-(p-1)\frac{A}{B}\int \bar v^{p-2}w^2 \,d\mu\Big\}\\
&=B^{-\frac{2}{p}}\Big\{ \langle \bar L_+w, w\rangle +\frac{A}{B^2}(p-2)\Big(\int \bar v^{p-1}w \,d\mu\Big)^2\Big\},
\end{split}	
\end{equation*}
as desired.

\end{proof}

%\textcolor{blue}{ Note that if $\bar{V}$ is a minimizer of $F_0$, then $\hat{V}=a \bar{V}$ is also a minimizer for any $a>0$. It satisfies
%\begin{equation}\label{EL-extension-2}
%\left\{\begin{split}
%& \partial_{\rho}\left(e_1(\rho)\rho^{1-2\gamma}\partial_{\rho}\hat{V}\right)
%+e_2(\rho)\rho^{1-2\gamma}\partial_{tt}\hat{V}
%+e_3(\rho)\rho^{1-2\gamma} \Delta_\theta \hat{V},
%=0,\quad \rho\in(0,\rho_0),t\in\mathbb R,\theta\in \mathbb S^{n-1}\\
%&-\lim_{\rho\to 0}\rho^{1-2\gamma}\partial_\rho \hat{V}(\rho,t)+\kappa \hat{v}-ca^{2-p}|\hat{v}|^{p-2}\hat{v}=0,\quad \text{on }\mathcal C,
%\end{split}
%\right.	
%\end{equation}
%and the corresponding linearized operator becomes:
%\begin{equation*}
%\hat L_+W=-\lim_{\rho\to 0}\rho^{1-2\gamma}\partial_\rho W+\kappa w-ca^{2-p}(p-1)\hat v^{p-2}w.
%\end{equation*}
%In this case, one has
%}
%
%Let us look at the operator $\bar L$ (which is equivalent up to constant to $\bar L_+$).
%
%--------------------------------------

\begin{proposition}\label{prop:stability}
Let $\bar v$ be a stable energy solution, this is,
$$\left.\frac{d}{d\epsilon}\right|_{\epsilon=0}F_0(\bar v+\epsilon w)= 0, \quad \left.\frac{d^2}{d\epsilon^2}\right|_{\epsilon=0}F_0(\bar v+\epsilon w)\geq 0.$$
Then the number of negative eigenvalues (Morse index) of $\bar L$ must be exactly one.
\end{proposition}

\begin{proof}
From Lemma \ref{lemma:second-variation} we know that $\langle\bar L w,w\rangle_{L^2}\geq 0$ for every $w$ in the orthogonal complement of $\bar v^{p-1}$. This implies that $\bar L$ has at most one negative eigenvalue. Since we already know that
\begin{equation*}
\langle \bar v,\bar L\bar v\rangle_{L^2}=-(p-2)\varsigma_{n,\gamma} \kappa_{\alpha,\gamma}^n\|\bar v\|^{p-1}_{L^{p-1}},
\end{equation*}
the number of negative eigenvalues (Morse index) must be exactly one.
\end{proof}

\subsection{Perron-Frobenius}

Now we show a  Perron-Frobenius property for our problem. We could have followed the proof in \cite{Frank-Lenzmann} since $P_\gamma$ is just the conjugate operator of $(-\Delta)^\gamma$, so it is still a positivity preserving operator. We give an alternative proof instead:

Let $\bar v:=v_{\alpha,\beta}$ { be} a minimizer for $F_{\alpha,\beta}$ in the radially symmetric class.  It satisfies
\begin{equation*}\label{eq:ground-state1}
P_\gamma^{(0)} \bar v+C(\alpha)\bar v=\varsigma_{n,\gamma} \kappa_{\alpha,\gamma}^n \bar v^{p-1},\quad t\in\mathbb R.
\end{equation*}

\begin{proposition}\label{perron-frobenius}
The linear operator
\begin{equation*}\label{11}
\bar L^{(0)} \phi:=P_\gamma^{(0)} \phi+C(\alpha)\phi-(p-1)\varsigma_{n,\gamma} \kappa_{\alpha,\gamma}^n\bar v^{p-2} \phi
\end{equation*}	
satisfies a Perron-Frobenius property, i.e. the eigenspace corresponding to the lowest eigenvalue is simple and the eigenfunction can be chosen strictly positive.
\end{proposition}

\begin{proof} Let $\phi_0$ be the first eigenfunction corresponding the lowest eigenvalue $\lambda_0$. Existence of $\phi_0$ follows from the same arguments as in Proposition \ref{prop:radial-minimizer} using the Rayleigh quotient for $\lambda_0$. We have that
\begin{equation}\label{1111}
\bar L^{(0)} \phi_0=P_\gamma^{(0)} \phi_0+C(\alpha)\phi_0-(p-1)\varsigma_{n,\gamma} \kappa_{\alpha,\gamma}^n\bar v^{p-2} \phi_0=\lambda_0 \phi_0.
\end{equation}	
{\bf Step 1.} First we will show that $\phi_0$ is positive for all $t\in\mathbb R$.
Recall that
\begin{equation*}
	\begin{split}
\lambda_0&=\inf_\phi\frac
{\int_{\mathbb R} \big(\phi P_\gamma^{(0)} \phi+C(\alpha)\phi^2-(p-1)\varsigma_{n,\gamma} \kappa_{\alpha,\gamma}^n \bar v^{p-2}\phi^2 \big)\,d t}{\int_{\mathbb R} \phi^2\,dt}\\	
&=	\inf_\phi
\Big\{\frac{\varsigma_{n,\gamma}\int_{\mathbb R}\int_{\mathbb R} \mathcal{K}_0(t-\tilde{t})(\phi(t)-\phi(\tilde{t}))^2\,dtd\tilde{t}}{2\int_{\mathbb R} \phi^2\,dt}
\\&\qquad\qquad+\frac{\int_{\mathbb R} \big((C(\alpha)-c_{n,\gamma})\phi^2-(p-1)\varsigma_{n,\gamma} \kappa_{\alpha,\gamma}^n\bar v^{p-2}\phi^2 \big)\,dt}{\int_{\mathbb R} \phi^2\,dt}\Big\}.	
	\end{split}
\end{equation*}
From this expression one automatically knows that $\phi_0\geq 0$ (otherwise, just replace $\phi_0$ by its absolute value). Next, if there exists $t_0$ such that $\phi_0(t_0)=0$, then at $t=t_0$,
\begin{equation*}
P_\gamma^{(0)}\phi_0(t_0)=\varsigma_{n,\gamma}\int_{\mathbb R}\mathcal{K}_0(t_0-t)(\phi_0(t_0)-\phi_0(t))\,dt\leq 0
\end{equation*}
and the inequality is strict if $\phi_0$ is not identically zero. Thus one has the left hand side of \eqref{1111} is less than zero while the right hand side is zero at $t_0$, which is a contradiction. We conclude $\phi_0>0$, as desired.\\

\noindent {\bf Step 2.} Next we show that $\lambda_0$ is simple. For this, we first we claim that the eigenfunction $\phi_0$ is even in $t$. This follows the same idea as in Proposition \ref{prop:symmetrization} by considering the decreasing rearrangement.
% It is enough to observe that  $-(p-1)$ is negative \textcolor{green}{why do we need this?}.

Assume that there exist $\phi_{0,1}, \, \phi_{0,2}$ which are eigenfunctions for $\lambda_0$. We can take $\phi_{0,1}(0)=\phi_{0,2}(0)$ (up to multiplying by a constant). Consider the equation satisfied by $\phi=\phi_{0,1}-\phi_{0,2},$
\begin{equation*}
P^{(0)}_\gamma \phi+\big(C(\alpha)-\lambda_0\big)\phi-(p-1)\varsigma_{n,\gamma} \kappa_{\alpha,\gamma}^n\bar v^{p-2}\phi	=0.
\end{equation*}
It is equivalent to the following extension problem:
\begin{equation}\label{extension0000}
\left\{\begin{split}
&\partial_\rho(e_1\rho^{1-2\gamma}\partial_\rho W)+e_2\rho^{1-2\gamma}\partial_{tt}W	=0, \rho\in(0,\rho_0),t\in\mathbb R,\\
&-{\tilde{d_{\gamma}}}\lim_{\rho\to 0}\rho^{1-2\gamma}\partial_\rho W(\rho,t)+\big(\kappa-\lambda_0\big)\phi- (p-1) \varsigma_{n,\gamma} \kappa_{\alpha,\gamma}^n\bar v^{p-2}\phi=0 \mbox{ on }\{\rho=0\},
\end{split}	\right.
\end{equation}
where $W(t,0)=\phi(t)$  and $\kappa:=c_{n,\gamma}+C(\alpha)=\varsigma_{n,\gamma} \kappa_{\alpha,\gamma}^n>0$. As in Section $4.2$ of \cite{Frank-Lenzmann-Silvestre}, we use a Hamiltonian argument. For $t>0$, let
\begin{equation*}
H(t)=\frac{1}{2}\int_0^{\rho_0}\rho^{1-2\gamma}\left[e_2(\rho) (\partial_t W)^2-e_1(\rho)(\partial_\rho W)^2 \right]\, d\rho-\frac{1}{2}\mathcal{V}(t)W^2(t,0)
\end{equation*}
where
\begin{equation*}
\mathcal{V}(t)=	(\kappa-\lambda_0)- (p-1)\varsigma_{n,\gamma} \kappa_{\alpha,\gamma}^n\bar v^{p-2},
\end{equation*}
and it satisfies $\mathcal{V}'(t)>0$ for $t>0$. We note that $H(t)$ is well-defined and smooth enough and also satisfies
\begin{equation*}
H(+\infty)=0.	
\end{equation*}
%{\color{magenta}$$2\lim_{t\to\infty}H(t)=\lim_{t\to\infty}\int_0^{\rho_0}e_2\rho^{1-2\gamma} W_t^2(t,\rho)-e_1\rho^{1-2\gamma}W_\rho^2(t,\rho) \, d\rho-\mathcal{V}(t)W^2(t{\color{magenta},0})	=0???$$}
%{\color{red}Is this trivial? Don't we need estimates on the derivatives??? $v(t)$ has exponential decay, so it is ok. I meant for the integral terms. I think they are ok because of the hamiltonian, but no sure. Can we use the same procedure than in the Wronskian or, am I making a mess from a trivial thing?}
%
Multiplying the first equation in \eqref{extension0000} by $W_t$ and integrating by parts, we can get that
%
%Moreover one can check, as  in Section $5$ in \cite{DelaTorre-Gonzalez} (see also Section 3 in \cite{acdfgw:survey}) that
%$$\frac{1}{2}\int_0^{\rho_0}\rho^{1-2\gamma}[e_2(\rho) (\partial_t W)^2-e_1(\rho)(\partial_\rho W)^2 ]\, d\rho-\mathcal{V}(t)W(t,0)W_t(t,0)	$$
%is constant for $t>0$. Thus,
%
\begin{equation*}
H'(t)=-\frac{1}{2}\mathcal{V}'(t)W^2(t,0)\leq 0	.
\end{equation*}
Then $H(t)$ is decreasing for $t\geq 0$ and by our definition of $\phi$, we know that $\phi$ is even in $t$ and $\phi(0)=0$, so
\begin{equation*}
\begin{split}
0=H(+\infty)&\leq H(0){=\frac{1}{2}\int_0^{\rho_0}\rho^{1-2\gamma} [e_2(\rho)(\partial_t W)^2(0,\rho)-e_1(\rho)(\partial_\rho W)^2(0,\rho)] \, d\rho-\frac{1}{2}\mathcal{V}(0)W^2(0,0)	}\\
&\leq -\frac{1}{2}\mathcal{V}(0)\phi(0)^2=0,
\end{split}
\end{equation*}
since $\partial_t W(0,\rho) =0$ by symmetry.
As a consequence, one has $H(t)\equiv0$, $H'(t)\equiv 0$ and $W(t,0) \equiv 0$, i.e. $\phi(t)\equiv 0$. We obtain that the eigenspace for $\lambda_0$ is one-dimensional, and thus it is simple.\\

\end{proof}

\section{Symmetry breaking}\label{sec7}

Proposition \ref{prop:stability} highlights the relation between linear stability and the spectrum of $\bar L_+$, which is a key idea in the construction of the Felli-Schneider curve \cite{Felli_Schneider} in the local case. The main obstacle in the fractional setting is the lack of an explicit formula for the eigenvalues of $\bar L_+$. Even if we do not have a complete picture, Theorem \ref{thm3} gives some partial answers to the symmetry breaking issue.\\

\noindent{\bf Proof of Theorem \ref{thm3} \emph{(i)}.} From the definition of the energy functional \eqref{eq-F}, and using the fact that $\kappa_{\alpha, \gamma}^n$ is strictly decreasing in $\alpha$ for $\alpha<0$, one can readily see that
\begin{equation*}
R(\alpha, \alpha)>S(0,0)=S(\alpha, \alpha)	
\end{equation*}
for $\alpha<0$, where in the last equality we have used  of Theorem \ref{thm1}.\emph{iii.}  and $R(\alpha, \alpha)$ is the minimum in radial class given in \eqref{min-radial}. Then by the continuity of $S(\alpha, \beta)$ in $\alpha, \, \beta$, it is easy to see that for $(\alpha, \beta)$ close to $(\alpha, \alpha)$, one has
\begin{equation*}
	R(\alpha, \beta)>S(\alpha, \beta),
\end{equation*}
as desired.
\qed\\

Next we shall give the proof of the second statement in Theorem \ref{thm3}. The main idea is to perturb a radially symmetric solution in the direction of a negative eigenvalue $\lambda_1$ (corresponding to the mode $m=1$) in order to decrease the functional. Here the properties of the conformal fractional Laplacian on $\mathcal C$ described in Section \ref{section:preliminaries} and, in particular, Proposition \ref{prop:symbol} will prove to be crucial.

 Fixed $-2\gamma<\alpha<0$ and $\alpha<\beta<\alpha+\gamma$, let $\bar v:=v_{\alpha,\beta}$ { be} a minimizer in the radially symmetric class.  It satisfies
\begin{equation}\label{eq:ground-state}
P_\gamma^{(0)} \bar v+C(\alpha)\bar v=\varsigma_{n,\gamma} \kappa_{\alpha,\gamma}^n \bar v^{p-1},\quad t\in\mathbb R.
\end{equation}
As before, we consider the linearized operator \eqref{linearized10}, and its projection over the $m$-th eigenspace:
\begin{equation*}
\bar L^{(m)} \phi:=P_\gamma^{(m)} \phi+C(\alpha)\phi-(p-1)\varsigma_{n,\gamma} \kappa_{\alpha,\gamma}^n \bar v^{p-2}\phi,\quad t\in\mathbb R.
\end{equation*}
We look at the eigenvalue problem
\begin{equation}\label{equation120}
\bar L^{(m)}\phi_m=\lambda_m\phi_m, \quad m=0,1,\ldots.
\end{equation}
Note that one can construct eigenfunctions for the original $\bar L$ on $\mathcal C$ by simply taking
\begin{equation}\label{equation200}
w_m(t,\theta)=\phi_m(t)E_m(\theta),\quad m=0,1,\ldots.
\end{equation}
By Proposition \ref{perron-frobenius}, $\lambda_0$ is simple and $\phi_0>0$.

Now  we look at the first eigenvalue for \eqref{equation120} corresponding to the mode $m=1$. The Rayleigh quotient is given by
\begin{equation*}\label{lambda1}
\lambda_1=\inf_\phi\frac
{\int_{\mathbb R} \big(\phi P_\gamma^{(1)} \phi+C(\alpha)\phi^2-(p-1)\varsigma_{n,\gamma} \kappa_{\alpha,\gamma}^n \bar v^{p-2}\phi^2 \big)\,dt}
{\int_{\mathbb R} \phi^2\,dt}.
\end{equation*}
Note that test functions $\phi$ should be understood as defined in $\mathcal C$ by $\phi(t)E_1(\theta)$, which are orthogonal in $L^2(\mathcal C)$ to the zeroth-eigenfunction.

We use $\phi=\bar v$ as test function in the Rayleigh quotient above. Then, from \eqref{eq:ground-state},
we have that
\begin{equation}\label{lambda-def}
\begin{split}
\lambda_1&\leq\frac
{\int_{\mathbb R} \big(\bar v P_\gamma^{(1)} \bar v -(p-1)\bar v P_\gamma^{(0)}\bar v-(p-2)C(\alpha)\bar v^2\big)\,dt}
{\int_{\mathbb R} \bar v^2\,dt}\\
&:=I_p-(p-2)C(\alpha).
\end{split}
\end{equation}
We estimate the term $I_p$ above using Fourier transform in the variable $t$. For this, recall the formulas in Proposition \ref{prop:symbol} for the symbol of $P^{(0)}_\gamma$ and $P^{(1)}_\gamma$. Then
\begin{equation*}
I_p=\frac
{\int_{\mathbb R} \big\{\Theta_\gamma^{(1)}(\xi)-(p-1)\Theta_\gamma^{(0)}(\xi)\big\}|\hat{\bar v}|^2  \,d\xi}
{\int_{\mathbb R} |\hat{ \bar v}|^2\,d\xi}.
\end{equation*}
%
%Recall \eqref{equation20}, \textcolor{blue}{ using the fact that $\Theta_\gamma^{(m)}$ is increasing in $m$ and using that $\Gamma(1+x)=x\Gamma(x)$
%\begin{equation*}
%\frac{\Theta_\gamma^{(1)}(\xi)}{\Theta_\gamma^{(0)}(\xi)}<\frac{\Theta_\gamma^{(2)}(\xi)}{\Theta_\gamma^{(0)}(\xi)}
%=\frac{(\gamma+\frac{n}{2})^2+\xi^2}{(-\gamma+\frac{n}{2})^2+\xi^2}\leq \frac{(n+2\gamma)^2}{(n-2\gamma)^2},
%\end{equation*}
Thus  $I_p$ is bounded for $-2\gamma<\alpha<0$, and there exists a positive constant $M$ independent of $\gamma, \alpha$ such that
\begin{equation*}
|I_p|\leq M.	
\end{equation*}
%{\color{red}We should justify why this is bounded and how do we get the curve. We need to control $\Theta_\gamma^{(1)}(\xi)-(p-1)\Theta_\gamma^{(0)}(\xi)$ for that values of $\alpha$. Moreover, it is not clear why this $\beta$ makes $\lambda_1<0$, we have no relation with $I_p$}
Recalling the value of $p$ from \eqref{value-p}, we deduce from \eqref{lambda-def} that there exists a curve
\begin{equation}\label{h(alpha)}
h(\alpha):=\frac{4\gamma C(\alpha)-(n-2\gamma)M}{4C(\alpha)+2M}+\alpha	
\end{equation}
such that for $\beta<h(\alpha)$, we have $\lambda_1<0$.

% {\color{blue}Indeed, $$\lambda_1\leq M-(p-2)C(\alpha)%=M-\frac{4\gamma C(\alpha)-4(\beta-\alpha)C(\alpha)}{n-2\gamma+2(\beta-\alpha)}
%=\frac{M(n-2\gamma)-4\gamma C(\alpha)+(\beta-\alpha)(4C(\alpha)+2M)}{n-2\gamma+2(\beta-\alpha)}.$$}
 Moreover, since $C(\alpha)$ is a smooth function in $(-2\gamma,0)$ satisfying $C(\alpha)\to +\infty$ as $\alpha\to -2\gamma$ (recall Corollary \ref{cor:C-alpha}),  one can see that there exists $\alpha_0\in (-2\gamma, 0)$ such that for $-2\gamma<\alpha<\alpha_0$,
 $$\alpha<h(\alpha)<\alpha+\gamma.$$
  In fact, from the definition of $h(\alpha)$, we have that  $h(\alpha)-\alpha\to \gamma$ as $\alpha\to -2\gamma$.\\

%and this quantity is less than $p-1$ for all $\xi\in\mathbb R$ if $\beta$ is close enough to $\alpha$ . Now we look at the constant.  By definition, $C(0)=0$. Moreover, we know from Corollary \ref{cor:C-alpha} that for $-2\gamma<\alpha<0$, $-C(\alpha)<0$.
%We conclude that, fixed $\alpha$, if $\beta$ is close enough to $\alpha$, the first eigenvalue $\lambda_1$ is negative. \textcolor{blue}{this part need to be done.}

\medskip

\noindent{\bf Proof of Theorem \ref{thm3} \emph{(ii)}.}  This is now a relatively standard argument, using perturbation to relate the symmetry breaking phenomena to the sign of the eigenvalue $\lambda_1$ as in \cite{Catrina_Wang2}.

Let $\phi_0,\phi_1$ be  eigenfunctions of the linearized equation \eqref{equation120} corresponding to the eigenvalues $\lambda_0,\lambda_1$. Set also $w_0$, $w_1$ as in \eqref{equation200}.
We will use $\bar v +\delta w_0+sw_1$ as a test function in the energy functional. For this we set
\begin{equation*}
G(\delta,s)=\int_{\mathcal C} |\bar v +\delta w_0+sw_1|^p\,d\mu.
\end{equation*}
As in Lemma 5.1 of \cite{Catrina_Wang2}, we can find an open $s$ interval around 0 and a solution $\delta(s)$ such that
\begin{equation*}
G(\delta(s),s)=1.
\end{equation*}
Moreover, $\delta'(0)=0$ and
\begin{equation*}
\delta''(0)= -\frac{(p-1)\int_{\mathcal C} \bar v^{p-2}w_1^2\,d\mu}{\int_{\mathcal C} \bar v^{p-1} w_0\,d\mu},
\end{equation*}
so that we have
\begin{equation}\label{delta-s}
\delta(s)=-s^2 \frac{(p-1)\int_{\mathcal C} \bar v^{p-2}w_1^2\,d\mu}{2\int_{\mathcal C} \bar v^{p-1} w_0\,d\mu}+o(s^2).
\end{equation}

The trickier part is to evaluate the functional on  our test function, this is
\begin{equation}\label{equation130}
\begin{split}
F_{\alpha,\beta}(\bar v +\delta w_0+sw_1)&
%=F_{\alpha,\beta}(\bar v)+\textcolor{magenta}{4\delta}\varsigma_{n,\gamma}\kappa_{\alpha,\gamma}^n\int_{\C}\bar{v}(z)w_0(z)\,d\mu\\&
%+\textcolor{magenta}{2\delta^2}\varsigma_{n,\gamma}\kappa_{\alpha,\gamma}^n\int_{\C}w^2_0(z)\,d\mu+\textcolor{magenta}{2s^2}\varsigma_{n,\gamma}\kappa_{\alpha,\gamma}^n\int_{\C}w^2_1(z)\,d\mu\\
%&+\delta^2\varsigma_{n,\gamma}\int_{\mathcal C}\int_{\mathcal C} K(z,\tilde z)(w_0(z)-w_0(\tilde z))^2\,d\mu d\tilde\mu\\
%&+2\delta\varsigma_{n,\gamma}\int_{\mathcal C}\int_{\mathcal C} K(z,\tilde z)(\bar v(z)-\bar v(\tilde z))(w_0(z)-w_0(\tilde z))\,d\mu d\tilde\mu \\
%&+
%2\delta s\varsigma_{n,\gamma}\int_{\mathcal C}\int_{\mathcal C} K(z,\tilde z)(\bar w_0(z)-\bar w_0(\tilde z))(w_1(z)-w_1(\tilde z))\,d\mu d\tilde\mu \\
%&+s^2\varsigma_{n,\gamma}\int_{\mathcal C}\int_{\mathcal C} K(z,\tilde z)(w_1(z)-\bar w_1(\tilde z))^2\,d\mu d\tilde\mu\\&
=F_{\alpha,\beta}(\bar v)+4\delta\kappa_{\alpha,\gamma}^n\int_{\C}\bar{v}(z)w_0(z)\,d\mu+2s^2\kappa_{\alpha,\gamma}^n\int_{\C}w^2_1(z)\,d\mu\\&+
2\delta\int_{\mathcal C}\int_{\mathcal C} K(z,\tilde z)(\bar v(z)-\bar v(\tilde z))(w_0(z)-w_0(\tilde z))\,d\mu d\tilde\mu \\
&+s^2\int_{\mathcal C}\int_{\mathcal C} K(z,\tilde z)(w_1(z)- w_1(\tilde z))^2\,d\mu d\tilde\mu+o(s^2),
\end{split}
\end{equation}
where we have used that $w_1$ is an odd function over the unit sphere in order to cancel some terms.
On the one hand, recall that the pair $(\lambda_1, w_1)$ is a solution to equation \eqref{equation120}
and thus,
\begin{equation*}
\int_{\mathcal C}\int_{\mathcal C} K(z,\tilde z)(w_1(z)-w_1(\tilde z))^2+2\kappa_{\alpha,\gamma}^n\int_{\mathcal C} w_1^2
=2(p-1)\kappa_{\alpha,\gamma}^n\int_{\mathcal C}\bar v^{p-2}w_1^2\,d\mu+2\lambda_1 \varsigma_{n,\gamma}^{-1}\int_{\mathcal C}w_1^2.
\end{equation*}
On the other hand, $\bar{v}$ is a solution to the nonlinear equation \eqref{eq:ground-state} and hence, taking $w_0$ as a test function in the weak formulation, we must have
\begin{equation*}
\int_{\mathcal C}\int_{\mathcal C}K(z,\tilde z)(\bar v(z)-\bar v(\tilde z))(w_0(z)-w_0(\tilde z))\,d\mu d\tilde\mu+2\kappa_{\alpha,\gamma}^n\int_{\mathcal C} \bar v w_0\,d\mu=
2\kappa_{\alpha,\gamma}^n \int_{\mathcal C}\bar v^{p-1}w_0\,d\mu.
\end{equation*}
Substituting into \eqref{equation130} we arrive at
\begin{equation*}
\begin{split}F_{\alpha,\beta}(\bar v +\delta w_0+sw_1)&=F_{\alpha,\beta}(\bar v)+4\delta \kappa_{\alpha,\gamma}^n \int_{\mathcal C}\bar v^{p-1}w_0\,d\mu
+2s^2(p-1) \kappa_{\alpha,\gamma}^n \int_{\mathcal C}\bar v^{p-2}w_1^2\,d\mu\\
&+2\lambda_1 \varsigma_{n,\gamma}^{-1}s^2 \int_{\mathcal C}w_1^2\,d\mu+o(s^2)\\
%&=F_{\alpha,\beta}(\bar v)-2s^2 \frac{(p-1)\int_{\mathcal C} \bar v^{p-2}w_1^2\,d\mu}{\int_{\mathcal C} \bar v^{p-1} w_0\,d\mu}\varsigma_{n,\gamma} \kappa_{\alpha,\gamma}^n \int_{\mathcal C}\bar v^{p-1}w_0\,d\mu\\&
%+2s^2\left((p-1)\varsigma_{n,\gamma} \kappa_{\alpha,\gamma}^n \int_{\mathcal C}\bar v^{p-2}w_1^2\,d\mu+\lambda_1 \int_{\mathcal C}w_1^2d\mu d\tilde\mu\right)+o(s^2)\\
&=F_{\alpha,\beta}(\bar v)+2\lambda_1 \varsigma_{n,\lambda}^{-1}s^2\int_{\mathcal C}w_1^2d\,\mu+o(s^2),
\end{split}
\end{equation*}
where we have used the relation for the parameters $s$ and $\delta$ given by \eqref{delta-s}.
% Thus one has
%\begin{equation*}
%F_{\alpha,\beta}(\bar v +\delta w_0+sw_1)=F_{\alpha,\beta}(\bar v)+2\lambda_1 s^2\int_{\mathcal C} w_1^2+o(s^2).
%\end{equation*}

The proof of Theorem \ref{thm3} \emph{(ii)} is completed by knowing that $\lambda_1<0$ for our choice of parameters.

\qed

\begin{remark}\label{remark:FS}
We conjecture that there should exist a Felli-Schneider curve corresponding to the case that $\lambda_1$ is zero, and such that on one side it is negative and corresponds to the symmetry breaking region.
\end{remark}

\section{Uniqueness of minimizers}\label{sec8}

In this Section we give the proof of Theorem \ref{thm:uniqueness}. So assume that we are in the symmetry region and let $\bar  u$ be radially symmetric minimizer of $E_{\alpha,\beta}$ (in fact, as one can see below, our proof works for positive radial solutions in the energy space, not only energy minimizers).  By Lemma \ref{lemma:relation}, it is enough to consider minimizers $\bar v=\bar v(t)$ of the functional in cylindrical coordinates $F_{\alpha,\beta}$. We have shown in Proposition \ref{prop:positivity} that $\bar v$ is positive. In addition, by Theorem \ref{thm2}, $\bar v$ is an even function in the $t$ variable.

Such $\bar v$ is a solution to the one-dimensional problem
\begin{equation}\label{equation30}
	P_\gamma^{(0)} \bar v+C(\alpha) \bar v= \varsigma_{n,\gamma} \kappa_{\alpha,\gamma}^n\bar v^{p-1}, \quad \bar v=\bar v(t).
\end{equation}
%By extension this is equivalent to
%\begin{equation}\label{extension0}
%\left\{\begin{split}
%&\partial_\rho(e_1\rho^{1-2\gamma}\partial_\rho \bar V)+e_2\rho^{1-2\gamma}\partial_{tt}\bar V 	=0, %\rho\in(0,\rho_0),t\in\mathbb R,\\
%&-\lim_{\rho\to 0}\rho^{1-2\gamma}\partial_\rho \bar V(\rho,t)+\kappa \bar v- \bar v^{p-1}=0 \mbox{ on %}\{\rho=0\}.
%\end{split}	\right.
%\end{equation}
The proof of Theorem \ref{thm:uniqueness} will follow the general scheme of \cite{Frank-Lenzmann,Frank-Lenzmann-Silvestre} for the non-local equation
\begin{equation*}
(-\Delta)^\gamma v+v=v^{p-1} \quad \text{in }\mathbb R,
\end{equation*}
performing a continuation argument in $\gamma$ in order to use the known uniqueness results in the local case $\gamma=1$. Note also that our approach is sometimes closer to that of \cite{Frank-Lenzmann-Silvestre} since we do not introduce a Lagrange multiplier as \cite{Frank-Lenzmann} does.

 A small side remark is that, since we need to work also on the local case $\gamma=1$, we can only allow $n>2$ in the Theorem instead of the usual $n>2\gamma$.

The first step is to set up function spaces. We would like to work on a fixed space
$\mathcal F$ on $\mathbb R$ defined by
\begin{equation*}
\mathcal F:=\{v\in L^2(\mathbb R)\cap L^{p}(\mathbb R)\,:\, v \text{ is even and real valued}\}
\end{equation*}
with the norm
\begin{equation*}
\|v\|_{\mathcal F} :=\|v\|_{L^2(\mathbb R)}+\|v\|_{L^p(\mathbb R)}.
\end{equation*}
However, since the original definition of $p$ depends on $\gamma$, we rewrite problem \eqref{equation30} (modulo a fixed rescaling constant) as
\begin{equation}\label{equation31}
P_\gamma^{(0)} v+c_0 v=|v|^{p_0-2}v, \quad v\in \mathcal F,
\end{equation}
for constants $c_0\in\mathbb R$, $p_0\in[2,2^*)$ fixed (independently of $\gamma$).

Finally, note that a Perron-Frobenius property still holds for the linearized operator thanks to Proposition \ref{perron-frobenius}. This is an essential ingredient in the proof of \cite{Frank-Lenzmann}.

\subsection{Local invertibility}

Here we fix $\gamma_0\in(0,1)$ and use $\gamma$ as a variable parameter. We show that one can find a unique solution to \eqref{equation31} for $\gamma$ sufficiently close to $\gamma_0$.

Consider the linearized operator
\begin{equation}\label{lin}
L_\gamma w:=P_\gamma^{(0)} w+c_0 w-(p_0-1)v^{p_0-2} w,\quad w\in L^2(\mathbb R).
\end{equation}
We know by Theorem \ref{thm:non-degeneracy} that $L_\gamma$ is non-degenerate, this is, its kernel consists only on multiples of $\bar v_t$. As a consequence, $L_\gamma$ has zero-kernel in the space of $t$-even functions, denoted by  $L^2_{even}(\mathbb R)$. Thus, by standard arguments, $\bar L^{(0)}_\gamma$ is invertible (with bounded inverse) in $L^2_{even}(\mathbb R)$. In addition:

%\textcolor{blue}{optimal regularity of $\bar v$ ?? $\rightarrow$ optimal regularity for $L^{-1}$}do we need it?

\begin{proposition}\label{prop:local-branch}
Assume that we have a solution $\bar v_{\gamma}$ of \eqref{equation31} with non-degenerate kernel for $\gamma=\gamma_0$. Then, for some $\delta > 0$, there exists a map in $v\in C^1(I,\mathcal F)$ defined on the interval
$I=[\gamma_0,\gamma_0+\delta])$ and denoted by $v_\gamma:=v(\gamma)$, such that the following holds:
\begin{itemize}
\item[a.] $v_\gamma$ solves \eqref{equation31} for all $\gamma\in I$, with $v_\gamma|_{\gamma=\gamma_0}=\bar v_{\gamma_0}$.
\item[b.] There exists $\epsilon> 0$ such that $v_\gamma$ is the unique solution of \eqref{equation31} for $\gamma\in I$ in the neighborhood $\{v\in\mathcal F \,:\, \|v-\bar v_{\gamma_0}\|_{\mathcal F}<\epsilon\}$.
\end{itemize}
\end{proposition}

\begin{proof}
The proof is exactly as Proposition 8.1 in \cite{Frank-Lenzmann-Silvestre}.

\end{proof}

\subsection{A priori bounds}

Assume that the local branch $v_\gamma$ constructed in the previous subsection can be continued for all $\gamma\in[\gamma_0,\gamma_*)$ for some $\gamma^*$ in the same conditions as in the proof of Proposition \ref{prop:local-branch}, in particular, satisfying the condition that $L_\gamma$ acting on $L^2(\mathbb R)$ has a bounded inverse on
on $L^2_{even}(\mathbb R)$.
 Thus we have a branch $v_\gamma\in C^{1}([\gamma_0,\gamma_*),\mathcal F)$.  We would like to prove that $\gamma_*=1$.

In order to prove some a-priori estimates it is helpful to write the original equation \eqref{equation31} in the extension, as given in Proposition \ref{prop:divV*}. This is,
\begin{equation}\label{extension-pohozaev}
\left\{\begin{split}
&\partial_\rho(e_1\rho^{1-2\gamma}\partial_\rho V_\gamma)+e_2\rho^{1-2\gamma}\partial_{tt}V_\gamma 	=0, \quad \rho\in(0,\rho_0),\ t\in\mathbb R,\\
&-\tilde{d}_{\gamma}\lim_{\rho\to 0}\rho^{1-2\gamma}\partial_\rho V_\gamma(\rho,t)+\kappa_\gamma v_\gamma- |v_\gamma|^{p_0-2}v_\gamma=0 \quad\text{on }\{\rho=0\},
\end{split}\right.	
\end{equation}
where $V_\gamma(t,\rho)$ has trace $v_\gamma(t)$ at $\{\rho=0\}$, and the constant is given by
\begin{equation}\label{kappa4}
\kappa_\gamma:=c_0+c_{n,\gamma}=\varsigma_{n,\gamma_0}\kappa_{\alpha,\gamma_0}^n-c_{n,\gamma_0}+c_{n,\gamma},
\end{equation}
 which is positive since $c_{n,\gamma}$, defined in \eqref{constants1}, is an increasing function of $\gamma$. In addition, $\kappa_\gamma$ is uniformly bounded above and below by a positive constant as $\gamma\to 1$.

We recall the following Pohozaev identities from \cite[Proposition 6.1]{acdfgw:survey} for this extension problem:

\begin{proposition}\label{prop:Pohozaev1}
If $V=V(t,\rho)$ is a solution of \eqref{extension-pohozaev}, then we have the following Poho\v{z}aev identities:
\begin{equation}\label{pohoz1}
\tilde d_\gamma\iint \rho^{1-2\gamma}\left\{e_1(\rho)(\partial_\rho V)^2+e_2(\rho)(\partial_t V)^2\right\}\,d\rho dt+\kappa_\gamma \int v^2\,dt=\left(\frac{1}{2}+\frac{1}{p_0}\right)\int |v|^{p_0}\,dt
\end{equation}
and
\begin{equation}\label{pohoz2}
\frac{\kappa_\gamma}{2}\int v^2\,dt+\left(1-\frac{1}{p_0}\right)\int |v|^{p_0}\,dt=\frac{\tilde d_\gamma}{2}\iint \rho^{1-2\gamma}\left\{-e_1(\rho)(\partial_\rho V)^2+e_2(\rho)(\partial_t V)^2\right\}		\,d\rho dt.	\\
\end{equation}

\end{proposition}

It is then natural to consider, for the branch $v_\gamma$, $\gamma\geq \gamma_0$,  the energy
\begin{equation*}
I_\gamma(v):=\tilde d_\gamma\iint \rho^{1-2\gamma} \left\{e_1(\rho)(\partial_\rho V)^2+e_2(\rho)(\partial_t V)^2\right\} \,d\rho \,dt.
\end{equation*}
From Proposition \ref{prop:Pohozaev1} above we conclude that
\begin{equation}\label{equation401}
I_\gamma(v_\gamma)\sim \int v_\gamma^2\,dt\sim \int  |v_\gamma|^{p_0}\,dt
\end{equation}
uniformly as $\gamma\to 1$. Indeed, from \eqref{pohoz1} and \eqref{pohoz2} we readily have
\begin{equation*}
I_\gamma(v_\gamma)\sim \int |v_\gamma|^{p_0}\,dt .
\end{equation*}
Looking again at \eqref{pohoz1}, this yields $\int |v_\gamma|^{p_0}\,dt\sim \int v_\gamma^2\,dt$, and the claim is proved. \\

\begin{proposition}
There exists $\sigma(\gamma)>0$ such that
\begin{equation*}
\int |v_\gamma|^{p_0} \,dt \leq \sigma(\gamma) \left\{I_\gamma(v_\gamma)\right\}^{\frac{(p_0-2)}{4\gamma}}\Big(\int v_\gamma^2\,dt\Big)^{\frac{p_0}{2}-\frac{(p_0-2)}{4\gamma}}.
\end{equation*}
Moreover, $\sigma(\gamma)$ is uniformly bounded for $\gamma>\gamma_0$.
\end{proposition}

\begin{proof}
First, Lemma A.4 in \cite{Frank-Lenzmann} shows that there exists $K_\gamma$ uniformly bounded as $\gamma\to 1$ such that
\begin{equation}
\int_{\r}|v|^{p_0}\,dt\leq K_\gamma \Big(\int_{\r}|(-\Delta)^{\frac{\gamma}{2}}v |^2\,dt \Big)^{\frac{p_0-2}{4\gamma}}\Big(\int v^2\,dt\Big)^{\frac{p_0}{2}-\frac{(p_0-2)}{4\gamma}}
\end{equation}
for every $v\in L^2(\mathbb R)$.
Next we prove that
\begin{equation*}
\int_{\r}|(-\Delta)^{\frac{\gamma}{2}}v_\gamma |^2\,dt\leq CI_\gamma (v_\gamma).
\end{equation*}
for $C$ uniformly bounded with respect to $\gamma$. It follows from the argument in Lemma \ref{lemma8.4} below that
\begin{equation*}
\begin{split}
&\int_{\r}|(-\Delta)^{\frac{\gamma}{2}}v_\gamma |^2\,dt=\|\,|\xi|^{\gamma}\hat{v}_\gamma\|_{L^2(\mathbb R)}^2
\leq
\int_{\r}(\Theta_{\gamma}+c_0)\,\hat{v}_\gamma^2\,d\xi=\int_{\r}(v_\gamma P_{\gamma}v_\gamma+c_0v_\gamma^2)\,dt.
\end{split}
\end{equation*}
Then from Proposition \ref{prop:Pohozaev1} and the discussion on the extension problem in Section \ref{section:cylindrical}, one has
\begin{equation*}
\int_{\r}|(-\Delta)^{\frac{\gamma}{2}}v_\gamma |^2\,dt\leq CI_\gamma(v_\gamma),
\end{equation*}
and this completes the proof.

\end{proof}

We have shown that \eqref{equation401} is improved to
\begin{equation}\label{lower-bound}
I_\gamma(v_\gamma)\sim \int v_\gamma^2\,dt\sim \int |v_\gamma|^{p_0}\,dt \gtrsim 1.
\end{equation}
We will  now check the upper bound. For this, we need some preliminary regularity estimates:

\begin{lemma}\label{lemma8.4}
Let
\begin{equation*}
\vartheta:=\gamma-\frac{(p_0-2)}{4p_0}
\end{equation*}
and note that $\vartheta>\gamma/2$ for $\gamma\geq\gamma_0$.
Then
\begin{equation*}
\|(-\Delta_\mathbb R)^\vartheta v_\gamma\|_{L^2(\mathbb R)}^2\lesssim \Big(\int |v_\gamma|^{p_0}\,dt\Big)^{\frac{2(p_0-1)}{p_0}}.
\end{equation*}
\end{lemma}

\begin{proof}
This is essentially equation (8.13) in \cite{Frank-Lenzmann-Silvestre} taking $N=1$ there (see also Lemma 5.4 in \cite{Frank-Lenzmann}). Let us point the necessary modifications. Recall that $P_\gamma v_\gamma+c_0 v_\gamma =|v_\gamma|^{p-2}v_\gamma$. Thus
\begin{equation*}
\|(-\Delta_{\mathbb R})^\vartheta v_\gamma\|_{L^2(\mathbb R)}^2=\left\|\frac{(-\Delta_{\mathbb R })^\vartheta}{P_\gamma+c_0} |v_\gamma|^{p_0-2}v_\gamma\right\|_{L^2(\mathbb R)}^2=\left\|\frac{|\xi|^{2\vartheta}}{\Theta^{(0)}_\gamma(\xi)+c_0}\widehat{|v_\gamma|^{p_0-2}v_\gamma}\right\|_{L^2(\mathbb R)}^2
\end{equation*}
Now we use Theorem 4.1 in \cite{acdfgw:survey}. It implies that, for any $k<c_{n,\gamma}$, the function $\frac{1}{\Theta_\gamma^{(0)}(\xi)-k}$ is a meromorphic function in $\xi\in\mathbb C$ and has no poles on the real line. In our case we take $k=-c_0<c_{n,\gamma}$ (recall the discussion around \eqref{kappa4}).
%this reads $0>c_0>-\Lambda_{n,\gamma}=-c_{n,\gamma} $, which is equivalent to $\kappa_{\gamma}>0$ and so, it is satisfied.

For the uniformity in $\gamma$, one only needs to consider $\Theta^{(0)}_\gamma(0)$. Since
\begin{equation*}
	\Theta^{(0)}_\gamma(0)=\frac{\Gamma(\frac{n+2\gamma}{4})^2}{\Gamma(\frac{n-2\gamma}{4})}
\end{equation*}
and
\begin{equation*}
\Theta^{(0)}_0(0)=1,\quad \Theta^{(0)}_1(0)=\frac{\Gamma(\frac{n+2}{4})^2}{\Gamma(\frac{n-2}{4})},
\end{equation*}
using these facts and that $\Theta^{(0)}_\gamma(0)+c_0>0$, one has $ \Theta^{(0)}_\gamma(0)+c_0>0 $ for $\gamma\in [0,1]$, so $\Theta^{(0)}_\gamma(0)+c_0\geq C_0>0$ and using the inequality (7.13) in \cite{acdfgw}, we can prove that $\Theta^{(0)}_\gamma(\xi)+c_0\geq C_0>0$ uniformly in $\gamma$.

In addition,  $\Theta_\gamma^{(0)}(\xi)$ behaves like the usual fractional Laplacian $(-\Delta_{\mathbb R})^\gamma$ as $|\xi|\to\infty$ thanks to \eqref{symbol-limit}. In order to have uniform bounds in $\gamma$, {recall that $\Theta^{(0)}(\xi)\to |\xi|^2+{(\frac{n}{2}-1)^2}$ as $\gamma\to 1$, one can divide the integral into two regions $\{|\xi|<R\}$, and $\{|\xi|>R\}$. For the second part,
\begin{equation*}
\int_{\{|\xi|>R\}}\Big(\frac{|\xi|^{2\vartheta}}{\Theta^{(0)}_\gamma(\xi)+c_0}\widehat{|v_\gamma|^{p_0-2}v_\gamma	}\Big)^2\,d\xi\leq C\int_{\{|\xi|>R\}}\Big(|\xi|^{2\vartheta-2\gamma}\widehat{|v_\gamma|^{p_0-2}v_\gamma	}\Big)^2\,d\xi,
\end{equation*}
and using $\Theta_\gamma+c_0$ has no zeros in the real line, uniformly in $\gamma$,
\begin{equation*}
\int_{\{|\xi|<R\}}\Big(\frac{|\xi|^{2\vartheta}}{\Theta^{(0)}_\gamma(\xi)+c_0}\widehat{|v_\gamma|^{p_0-2}v_\gamma	}\Big)^2\,d\xi\leq C\int_{\{|\xi|<R\}}\Big(|\xi|^{2\vartheta-2\gamma}\widehat{|v_\gamma|^{p_0-2}v_\gamma	}\Big)^2\,d\xi.
\end{equation*}
So we arrive to}
\begin{equation*}
\|(-\Delta_{\mathbb R})^\vartheta v_\gamma\|_{L^2(\mathbb R)}^2\leq \|(-\Delta_{\mathbb R})^{\vartheta-\gamma} (|v_\gamma|^{p_0-2}v_\gamma)\|^2_{L^2(\mathbb R)},
\end{equation*}
which suits our purposes even if not the best possible bound.

The rest of the proof goes as in \cite{Frank-Lenzmann-Silvestre}, just by noting  that
$p_0<2^*_{\gamma_0}\leq 2^*_{\gamma}<\frac{2}{1-2{\gamma}}$ for $\gamma_0\leq \gamma<1/2$ (resp. $p_0<+\infty$ if $\gamma\geq 1/2$).

\end{proof}
Finally, a delicate argument using symbol calculus yields the upper bound for \eqref{lower-bound}. More precisely, it is enough to obtain an upper bound for $\int v_\gamma^{p_0}\,dt$. Since $v_\gamma$ will be a positive function, we drop the absolute value in the notation.

Recall that the equation satisfied by $v_\gamma$ is
\begin{equation*}
P_\gamma^{(0)} v_\gamma+c_0 v_\gamma=v_\gamma^{p_0-1}.
\end{equation*}
Differentiate with respect to $\gamma$ (this differentiation will be denoted by a dot). Dropping superindex $(0)$ for simplicity, one has
\begin{equation}\label{equation41}
L_\gamma\dot v_\gamma+\dot P_\gamma  v_\gamma=0,
\end{equation}
where $L_\gamma$ was defined in \eqref{lin}. Using that
\begin{equation*}
L_\gamma v_\gamma=(2-p_0) v_\gamma^{p_0-1},
\end{equation*}
we calculate
\begin{equation*}
\begin{split}
\frac{d}{d\gamma} \int (v_\gamma)^{p_0}\,dt&=p_0 \int (v_\gamma)^{p_0-1} \dot v_\gamma\,dt
=\frac{p_0}{2-p_0} \int \dot v_\gamma L_\gamma v_\gamma \,dt=\frac{p_0}{2-p_0}  \int  v_\gamma L_\gamma \dot v_\gamma \,dt,
\end{split}
\end{equation*}
where we have used that $P_\gamma$ is self-adjoint. Thus from  \eqref{equation41} we arrive to
\begin{equation*}
\begin{split}
\frac{d}{d\gamma} \int (v_\gamma)^{p_0}\,dt&=-\frac{p_0}{2-p_0} \int v_\gamma \dot P_\gamma v_\gamma \,dt\\
&= -\frac{p_0}{2-p_0} \int \dot\Theta_\gamma^{(0)}(\xi)|\hat v_\gamma(\xi)|^2\,d\xi
\end{split}
\end{equation*}
in Fourier variables.

We need to estimate $\dot\Theta_\gamma^{(0)}$. For this, we refer to Step 3 in the proof of Lemma 8.2 in \cite{Frank-Lenzmann-Silvestre} and point out the (minor) modifications. First, split
\begin{equation*}
\int \dot\Theta_\gamma^{(0)}(\xi)|\hat v_\gamma(\xi)|^2\,d\xi=\int_{\{|\xi|< R\}}+\int_{\{|\xi|\geq R\}}=:I_<+I_>.
\end{equation*}
Since $\Theta_\gamma^{(0)}(\xi)$ is a meromorphic function in $\xi\in\mathbb C$, it has no poles on the real line, and behaves like the usual fractional Laplacian $(-\Delta)^\gamma$ as $\xi\to\infty$ thanks to \eqref{symbol-limit}, the calculation for $I_>$ is exactly as in \cite{Frank-Lenzmann-Silvestre}.

The estimate for  $I_<$ is actually easier since the symbol is smooth at the origin and the bounds can be taken independently of $\gamma$ (actually, $\Theta_\gamma^{(0)}(\xi)\to |\xi|^2+{(\frac{n}{2}-1)^2}$ as $\gamma\to 1$).

Summarizing the results in this subsection, we have proved:
\begin{lemma}\label{lemma:a-priori}
\begin{equation}\label{equation400}
I_\gamma(v_\lambda)\sim \int v_\gamma^2\,dt\sim \int v_\gamma^{p_0}\,dt\sim 1
\end{equation}
independently of $\gamma$, for $\gamma\in[\gamma_0,\gamma_*)$.
\end{lemma}

Now we show positivity:
\begin{lemma}
Suppose that $\bar v_{\gamma_0}> 0$ is positive, even function. Then, for all $\gamma\in[\gamma_0,\gamma_*)$, we have that  $v_\gamma$ is also a positive, even function on $\mathbb R$.
\end{lemma}

\begin{proof}
This is a small variation of Lemma 5.5 in \cite{Frank-Lenzmann}, and we only point out the necessary modifications. The general scheme is to show that positivity of $v_\gamma> 0$ is an open and closed property along the branch $(v_\gamma)$.

\emph{1. Open property:} Define the family of operators
\begin{equation*}
A_\gamma:=P_\gamma+c_0-\mathcal V_\gamma(t),\quad \text{for}\quad \mathcal V(t):=|v_\gamma|^{p_0-2}.
\end{equation*}
acting on $L^2(\mathbb R)$. They satisfy
\begin{equation}\label{Agamma}
A_\gamma v_\gamma=0 \text{ in }\mathbb R,
\end{equation}
which means that $v_\gamma$ is an eigenfunction with eigenvalue $0$.

Let $\lambda_{1,\gamma}$ denote the lowest eigenvalue of $A_\gamma$. We will prove that if $0$ is the lowest eigenvalue for $A_{\gamma'}$, then the same property holds for $A_\gamma$ for any  $|\gamma-\gamma'|<\epsilon$. In this case, by Lemma \ref{perron-frobenius}, the eigenvalue $\lambda_{1,\gamma}$ is nondegenerate and its corresponding eigenfunction $\Psi_{1,\gamma}$ is strictly positive.

As in \cite{Frank-Lenzmann}, one shows that if $\gamma\to \gamma'$, then $A_\gamma\to A_{\gamma'}$ uniformly in $\gamma$ in the norm resolvent sense. The argument is based in the fact that $P_\gamma$ is an operator of order $2\gamma$ as the fractional Laplacian $(-\Delta_{\mathbb R})^\gamma$, and the previous estimates. Then we have that $\lambda_{1,\gamma}\to\lambda_{1,\gamma'} $ as $\gamma\to\gamma'$
and that $\lambda_{1,\gamma}$ is simple and isolated from the rest of the spectrum for $\gamma$ close to $\gamma'$.  In this situation, \eqref{Agamma} implies that $\lambda_{1,\gamma}=0$ for any $|\gamma-\gamma'|<\epsilon$.

Now we claim that if $v_{\gamma'}$ is positive solution of \eqref{Agamma}, then 0 must be the first eigenvalue for $A_{\gamma'}$. Assume, by contradiction, that $\lambda_{1,\gamma'}<0$ and $\lambda_{j,\gamma'}=0$ for some $j>1$. Thus the orthogonality condition $\langle\Psi_{1,\gamma'},\Psi_{j,\gamma'}\rangle=0$ must hold, which is not possible by the sign condition.

We conclude that, if $v_{\gamma'}$ is positive, also $v_\gamma$ has a sign for every $|\gamma-\gamma'|<\epsilon$. Recalling that $v_\gamma\to v_{\gamma'}$ a.e. as $\gamma\to \gamma'$ we must have $v_\gamma=\Psi_{1,\gamma}>0$.\\

\emph{2. Closed property:}
Fix $\gamma'\in (\gamma_0,\gamma_*)$. Take a sequence $(\gamma_k) \subset [\gamma_0,\gamma')$ satisfying $\gamma_k\to \gamma'$ and $v_{\gamma_k}>0$. As in \cite{Frank-Lenzmann}, $v_{\gamma_k}\to v_{\gamma'}$ in $H^{\sigma_0}(\mathbb R)$ for any $0\leq\sigma_0<\gamma_0$, so it converges pointwise a.e. in $\mathbb R$, which implies that $v_{\gamma'}\geq 0$. Moreover, by our a-priori estimates from Lemma \ref{lemma:a-priori}, we know that $v_{\gamma'}\not\equiv 0$. Such $v_{\gamma'}$ is a solution to
\begin{equation*}
	P_{\gamma'}v_{\gamma'}+c_0v_{\gamma'}=v_{\gamma'}^{p_0-1}.
\end{equation*}
If there exists $t_0$ such that $v_{\gamma'}(t_0)=0$, by definition,
\begin{equation}\label{positive}
\begin{split}
P_{\gamma'} v_{\gamma'}(t_0)&=\varsigma_{n,\gamma'}\int_{\mathbb R}
\mathcal{K}_0(t_0-t)(v_{\gamma'}(t_0)-v_{\gamma'}(t))\,dt +c_0 v_{\gamma'}(t_0)\\
& =v_{\gamma'}^{p_0-1}(t_0)=0,
\end{split}
\end{equation}
while the first line should be strictly negative unless $v_{\gamma'}$ is identically zero, which is a contradiction. We conclude that $v_{\gamma'}>0$ everywhere.
\end{proof}

Next we consider a decay estimate (uniform in $\gamma$):

\begin{lemma}
Suppose that $v_{\gamma_0} > 0$ holds. Then, for all $\gamma\in[\gamma_0, \gamma_*)$,
\begin{equation*}
0 < v_\gamma(t) \leq \frac{1}{|t|},
\end{equation*}
for $|x| > R_0$. Here $R_0 > 0$ is some constant independent of $\gamma$.
\end{lemma}

\begin{proof}
This is just Lemma 8.3 in \cite{Frank-Lenzmann-Silvestre}. Just take into account that $P_\gamma^{(0)}$ is self-adjoint in $\mathbb R$.

\end{proof}

As a consequence of the previous results, we have:

\begin{lemma}\label{lemma:convergence}
Let $(\gamma_k)\subset [\gamma_0,\gamma_*)$ be a sequence such that $\gamma_k\to \gamma_*$, and suppose
that $v_{\gamma_k}>0$  for all $k\in\mathbb N$. Then, after passing to a subsequence if necessary, we
have $v_{\gamma_k}\to v_{\gamma_*}$  in $L^2(\mathbb R)\cap L^{p_0}(\mathbb R)$. Moreover, the function
$v_{\gamma_*}$ is strictly positive and satisfies
\begin{equation}\label{equation50}
P_{\gamma_*}^{(0)} v_{\gamma_*}+c_0 v_{\gamma_*}=v_{\gamma_*}^{p_0-1}\quad\text{in }\mathbb R.
\end{equation}
\end{lemma}

\begin{proof}
This is just Lemma 8.4 in \cite{Frank-Lenzmann-Silvestre} or Lemma 5.7 in \cite{Frank-Lenzmann}.

\end{proof}

\subsection{Global continuation}

We are ready for the proof of Theorem \ref{thm:uniqueness}. Fix $\gamma_0\in(0,1)$ and let $\bar v_{\gamma_0}$ be a minimizer for  $F_{\alpha,\beta}$ that is positive, radially
symmetric and decreasing in the radial variable. Let us show that its maximal branch extends to $\gamma_*=1$.

\begin{lemma}
Let $(v_\gamma)$ be the maximal branch starting at $\bar v_\gamma$ with $\gamma\in[\gamma_0,\gamma_*)$. Then $\gamma_*=1$.
\end{lemma}

\begin{proof}
We follow Proposition 5.2 in \cite{Frank-Lenzmann} and only give a rough sketch. The key idea is that the linearized operator at $\gamma_0$ is non-degenerate by Theorem \ref{thm:non-degeneracy}, so it has Morse index 1 when acting on even functions, and this property is satisfied along the branch.

Let $(\gamma_k) \subset [\gamma_0, \gamma_*)$ be a sequence such that $\gamma_k\to \gamma_*$, and consider the corresponding sequence $(v_{\gamma_k})$. Each $v_{\gamma_k}$ is an even, positive solution of
\begin{equation*}
P_\gamma v_\gamma+c_0 v_\gamma=v_\gamma^{p_0-1}\quad\text{in }\mathbb R.
\end{equation*}
By the previous Lemma \ref{lemma:convergence}, the sequence $v_{\gamma_k}\to v_{\gamma_*}$ in $L^2(\mathbb R)\cap L^{p_0}(\mathbb R)$ and $v_{\gamma_*}$ is a positive solution of \eqref{equation50}.

We show that $\gamma_*=1$ by a contradiction argument. Indeed, if $\gamma_*<1$, by passing to the limit one can prove that the Morse index of $L_{\gamma_*}$ when acting on even functions is exactly one. In any case, we have non-degeneracy for equation \eqref{equation50} as in Section \ref{section:nondegeneracy}. We conclude that $L_{\gamma_*}$ is invertible and the branch could be continued beyond $\gamma_*$, which yields a contradiction.\\
\end{proof}

\begin{proof}[\bf Proof of Theorem \ref{thm:uniqueness}]
Now fix $\gamma_0\in(0,1)$ and assume there exist two radially symmetric, radially decreasing, positive minimizers $v_{\gamma_0}$ and $\tilde v_{\gamma_0}$ not identically equal. Then the previous discussion yields the existence of two branches $(v_\gamma)$ and $\tilde v_\gamma$ that extend up to $\gamma_*=1$, and they cannot intersect at any $\gamma\in(\gamma_0,1)$ by the local uniqueness of Proposition \ref{prop:local-branch}. In addition, we have $v_{\gamma}\to v_{\gamma_*}$ and $\tilde v_{\gamma}\to \tilde  v_{\gamma_*}$ in $L^2(\mathbb R)\cap L^{p_0}(\mathbb R)$. Both $v_{\gamma_*}$ and $\tilde v_{\gamma_*}$ are radially symmetric and radially decreasing, positive solutions to
\begin{equation}\label{equation51}
P_1^{(0)} v+c_0 v=v^{p_0-1}\quad\text{in }\mathbb R.
\end{equation}
In fact here,
 $$P_1^{(0)}=-\partial_{tt}+\tfrac{(n-2)^2}{4} \mbox{ on }\mathbb R.$$

Once we arrive to the local equation, we conclude as in \cite{Frank-Lenzmann}. Equation \eqref{equation51}  is known to have a unique solution (see Proposition 5.2 in \cite{Frank-Lenzmann}, or the references \cite{Chang-Gustafson-Nakanishi-Tsai,Kwong}), hence $v_{\gamma_*}\equiv\tilde v_{\gamma_*}$. In fact, this solution is known in closed form and has the formula
\begin{equation*}
v_*(t)=\left(\frac{\frac{p_0-2}{2}+\frac{(n-2)^2}{4}+c_0}{\cosh^2(\frac{p_0-2}{2}t)}\right)^{\frac{1}{p_0-2}}.
\end{equation*}
In addition, $\|v_{\gamma}-\tilde v_{\gamma}\|_{L^2\cap L^{p_0}}\to 0$ as $\gamma\to 1^-$. But this is a contradiction since in the local case, $L_1$ is known to be non-degenerate, thus there must be a unique branch around $\gamma_*=1$.

\end{proof}

\section{Appendix A: Numerology}

The arguments here are relatively standard but we give full details for convenience of the reader.

\begin{lemma}\label{integral:study}
Let $\gamma\in(0,1)$ and $\alpha,\bar{\alpha}\in\r$.  The integral function given by
\begin{equation*}\label{I}
I(x)=P.V.\int_{\r^n}\frac{|x|^{-\bar{\alpha}} -|y|^{-\bar{\alpha}}}{|x-y|^{n+2\gamma}|y|^{{\alpha}}}\,dy
\end{equation*}
is radially symmetric and, in fact, it is the homogeneous distribution
\begin{equation}\label{I-constant}
I(x)=\frac{1}{|x|^{2\gamma+\alpha+\bar{\alpha}}}\kappa^{n,\bar{\alpha}}_{{\alpha},\gamma},
\end{equation}
where $\kappa^{n,\bar{\alpha}}_{{\alpha},\gamma}$ is the constant defined as
\begin{equation}\label{kappa_general}
\kappa^{n,\bar{\alpha}}_{{\alpha},\gamma}:=\text{P.V.}\int_{\r^n}\frac{1 -|{\zeta}|^{-\bar{\alpha}}}{|e_1-{\zeta}|^{n+2\gamma}|{\zeta}|^{{\alpha}}}\,d{\zeta}.
\end{equation}
This constant is finite when either $\bar\alpha=0$, so $\kappa^{n,\bar{\alpha}}_{{\alpha},\gamma}=0$, or
 \begin{equation*}\left\{
\begin{split}
-\bar{\alpha}-\alpha<2\gamma\text{ and } n>\alpha, &\quad\text{ if } \bar{\alpha}<0,\\
\alpha>-2\gamma \text{ and } n>\alpha+\bar{\alpha},&\quad\text{  if  }\bar{\alpha}>0.
\end{split}\right.
\end{equation*}
%\begin{color}{red} $\alpha<n$; $n-\alpha-\bar{\alpha}<n+2\gamma$ if $\bar{\alpha}<0$ or $n-{\alpha}<n+2\gamma$ if $\bar{\alpha}>0$; $\gamma\in(0,1)$. ESTAS SON LAS CONDICIONES, LAS DEJO POR SI LUEGO QUIERO AMPLIAR EL RANGO A NEGATIVO.
%Pongo algo de $\bar{\alpha}>0$??? No por ahora\end{color}
\end{lemma}

\begin{proof}
First, it is straightforward to check that  $I$ is radially symmetric, i.e, if $R$ denotes any rotation, then $I(x)=I(Rx)$. Then, \eqref{I-constant} follows by writing $e_1=(1,0,\cdots,0)$ and changing variables $y=|x|{\zeta}$.

\end{proof}

Moreover:

\begin{corollary}\label{c}
Let $\bar{\alpha}<0$. The constant ${\kappa}^{n,\bar{\alpha}}_{\alpha,\gamma}$ is positive for all $\alpha>\frac{n-2\gamma}{2}-\frac{\bar{\alpha}}{2}$, negative for all ${\alpha}<\frac{n-2\gamma}{2}-\frac{\bar{\alpha}}{2}$ and zero if $\alpha=\frac{n-2\gamma}{2}-\frac{\bar{\alpha}}{2}$. %In particular, the constant ${\kappa}^{n}_{\alpha,\gamma}$ given in \eqref{kappa} is positive for all $\alpha>\tfrac{n-2\gamma}{2}$.

Now let $\bar{\alpha}>0$. Then the constant ${\kappa}^{n,\bar{\alpha}}_{\alpha,\gamma}$ is positive for all $\alpha<\frac{n-2\gamma}{2}-\frac{\bar{\alpha}}{2}$, negative for all ${\alpha}>\frac{n-2\gamma}{2}-\frac{\bar{\alpha}}{2}$ and zero if $\alpha=\frac{n-2\gamma}{2}-\frac{\bar{\alpha}}{2}$.
\end{corollary}
\begin{proof}
We first use the polar coordinates for the variable $\zeta$: ${\varrho}=|\zeta|$, $\theta\in\s^{n-1}$, and represent $e_1$ by $\sigma\in\s^{n-1}$, then we have
\begin{equation}\label{kappa_2-1}
\kappa^{n,\bar{\alpha}}_{{\alpha},\gamma}=\text{P.V.}\int_{\r^n}\frac{(1 -|{\zeta}|^{-\bar{\alpha}})}{|e_1-{\zeta}|^{n+2\gamma}|{\zeta}|^{{\alpha}}}\,d{\zeta}
=\int_{\s^{n-1}}J(\theta)%\int_0^{\infty}\tfrac{(1 -{\varrho}^{-\bar{\alpha}}){\varrho}^{n-1-{\alpha}}}{(1+{\varrho}^2-2{\varrho}<\sigma,\theta>)^{\frac{{n+2\gamma}}{2}}}\,d{\varrho}
\,d\theta,
\end{equation}
where we have defined
\begin{equation}\label{J}
J(\theta)=\text{P.V.}\int_0^{\infty}\frac{(1 -{\varrho}^{-\bar{\alpha}}){\varrho}^{n-1-{\alpha}}}
{(1+{\varrho}^2-2{\varrho}\langle\sigma,\theta\rangle)^{\frac{n+2\gamma}{2}}}\,d{\varrho}.
\end{equation}
 We can easily write, using the change of variable $\tilde{{\varrho}}=1/{\varrho}$  in the first integral in the second line, as
\begin{equation}\label{J-split}
\begin{split}
J(\theta)&=\lim_{\epsilon\rightarrow 0}\int_0^{1-\epsilon}\frac{(1 -{\varrho}^{-\bar{\alpha}}){\varrho}^{n-1-{\alpha}}}{(1+{\varrho}^2-2{\varrho} \langle\sigma,\theta\rangle)^\frac{n+2\gamma}{2}}\,d{\varrho}+\int_{1+\epsilon}^{\infty}\frac{(1 -{\varrho}^{-\bar{\alpha}}){\varrho}^{n-1-{\alpha}}}{(1+{\varrho}^2-2{\varrho} \langle\sigma,\theta\rangle)^\frac{n+2\gamma}{2}}\,d{\varrho}\\
&=\lim_{\epsilon\rightarrow 0}\int_{1+\epsilon}^{\infty}\frac{-(1 -{\varrho}^{-\bar{\alpha}}){\varrho}^{2\gamma-1+\alpha+\bar{\alpha}}}{(1+{\varrho}^2-2{\varrho} \langle\sigma,\theta\rangle)^\frac{n+2\gamma}{2}}\,d{\varrho}+\int_{1+\epsilon}^{\infty}\frac{(1 -{\varrho}^{-\bar{\alpha}}){\varrho}^{n-1-{\alpha}}}{(1+{\varrho}^2-2{\varrho} \langle\sigma,\theta\rangle)^\frac{n+2\gamma}{2}}\,d{\varrho}\\
&=\lim_{\epsilon\rightarrow 0}\int_{1+\epsilon}^{\infty}\frac{(1 -{\varrho}^{-\bar{\alpha}}){\varrho}^{-1}}{(1+{\varrho}^2-2{\varrho} \langle\sigma,\theta\rangle)^\frac{n+2\gamma}{2}}({\varrho}^{n-{\alpha}}-{\varrho}^{2\gamma+\alpha+\bar{\alpha}})\,d{\varrho}.
\end{split}
\end{equation}
The corollary follows easily by studying the sign of this $J$.

\end{proof}

%\begin{remark}\label{remark1}
%From the expression of $J(\theta)$ above, one can see that for $\bar{\alpha}=\frac{n-2\gamma}{2}-\alpha$ and %$\alpha<0$, $\kappa_{\alpha, \gamma}^{n,\bar{\alpha}}$ is decreasing in $\alpha$.
%\end{remark}

\begin{corollary}\label{cor_nu}
Let
\begin{equation*}
\nu:=\frac{n-2\gamma}{2}-{\alpha},
\end{equation*}
and define
\begin{equation}\label{kappa}
{\kappa}^n_{\alpha,\gamma}:={\kappa}^{n,\nu}_{\alpha,\gamma}. \end{equation}
If $-2\gamma<\alpha<\frac{n-2\gamma}{2}$, then we have $0<\kappa^n_{\alpha,\gamma}<\infty$ and $\kappa_{\alpha, \gamma}^{n}$ is decreasing in $\alpha$.
Moreover, one has
\begin{equation*}
\varsigma_{n,\gamma}\kappa_{0,\gamma}^n	=c_{n,\gamma},  \mbox{ and }
 \varsigma_{n,\gamma}\kappa_{\alpha,\gamma}^n \to 0\mbox{ as }\alpha\to \frac{n-2\gamma}{2}.	\end{equation*}

In addition, the function $u(x)=|x|^{-\nu}$ is a solution of Euler-Lagrange equation \eqref{eq-extremal} with the constant normalized as $c=\kappa^n_{\alpha,\gamma}$.
\end{corollary}

\begin{proof}
This is a simple consequence of Lemma \ref{integral:study}, but a direct proof follows by calculating (in polar coordinates with $\varrho=\tfrac{|y|}{|x|}$ and $\theta,\sigma\in\s^{n-1}$ for $x,y$, respectively)
\begin{equation*}
\begin{split}
\int_{\r^n}\frac{|x|^{-\nu}-|y|^{-\nu}}{|x-y|^{n+2\gamma}|y|^{\alpha}}\,dy
&=|x|^{-\nu-2\gamma-\alpha}\int_{\s^n}\int_{0}^{\infty}
\frac{(1-{\varrho}^{-\nu}){\varrho}^{n-1-{\alpha}}}
{(1+{\varrho}^2-2{\varrho}\langle\sigma,\theta\rangle)^{\frac{n+2\gamma}{2}}}\,d\varrho\,d\theta\\
&=|x|^{-\nu-2\gamma-\alpha}{\kappa}^n_{\alpha,\gamma},
%={\color{magenta}|x|^{\nu-p(\nu+\beta)}{\kappa}^n_{\alpha,\gamma}},
\end{split}
\end{equation*}
and recalling the definition of the constant  \eqref{kappa_general} (and \eqref{J}).

\end{proof}

\begin{corollary}\label{cor:C-alpha}
We define the constant
\begin{equation}\label{constant}
C(\alpha):=\varsigma_{n,\gamma}\kappa_{\alpha,\gamma}^n-c_{n,\gamma}.
\end{equation}
Then $C(\alpha)$ is a smooth, decreasing function in $\alpha$ for $-2\gamma<\alpha<\frac{n-2\gamma}{2}$ and satisfies $C(0)=0$ and
\begin{equation}\label{Calpha}-C(\alpha)<c_{n,\gamma}
\end{equation}
for $-2\gamma<\alpha<0$.
\end{corollary}
\begin{proof}
This is just a straightforward consequence of \eqref{J-split} and the arguments above.\\

\end{proof}

\noindent\textbf{Acknowledgements:} The authors would like to thank R. Frank and A. Nazarov for pointing out relevant references, and the anonymous referees for all their valuable suggestions.

W. Ao is supported by NSFC of China No.11631011, No.11801421 and No.12071357. 
A. DelaTorre acknowledges financial support from Junta de Andaluc\'ia (FQM116), the Spanish Ministry of Science and Innovation (MICINN) through the grant Juan de la Cierva incorporaci\'on 2018 with ref. IJC2018-036320-I and through the ``Maria de Maeztu'' Excellence Unit IMAG with ref. CEX2020-001105-M, funded by MCIN/AEI/10.13039/501100011033/,
%IMAG-Maria de Maeztu Excellence Grant CEX2020-001105- M/AEI/10.13039/501100011033,
from the Spanish Government through the grant MTM2017-85757-P (MICIU) and through the grant PGC2018-096422-B-I00 (FEDER-MINECO).
M.d.M. Gonz\'alez  acknowledges financial support from the Spanish Government, grant  MTM2017-85757-P, and ``Severo Ochoa Programme for Centres of Excellence in R\&D'' (CEX2019-000904-S). 
%W. Ao is supported by NSFC of China No.11631011, No.11801421 and No.12071357. M. Gonz\'alez  acknowledges financial support from the Spanish  Government: MTM2017-85757-P, PID2020-113596GB-I00, RED2018-102650-T funded by MCIN/AEI/10.13039/501100011033, and
% the ``Severo Ochoa Programme for Centers of Excellence in R\&D'' (CEX2019-000904-S). A. DelaTorre is supported by grants
%%FQM116 (from the Junta de Andaluc\'ia),
%Juan de la Cierva incorporaci\'on 2018 with ref. IJC2018-036320-I,  MTM2017-85757-P and PGC2018-096422-B-I00. Additionally, this work is supported in part by the IMAG–Maria de Maeztu grant MCIN/AEI/ 10.13039/501100011033 and FEDER “Una manera de hacer Europa”.

\end{document}